\documentclass[]{article}
\usepackage{amsmath,amsthm,amssymb}
\usepackage{soul}
\usepackage{xcolor}
\usepackage[margin=2cm]{geometry}
\usepackage{cite}
\usepackage{dsfont}
\usepackage[title]{appendix}
\usepackage{graphicx}
\usepackage{esvect}
\usepackage{booktabs}
\usepackage[colorlinks=true, linkcolor=blue, citecolor=blue, urlcolor=red, linktoc=page]{hyperref}
\usepackage{bbm}
\usepackage{authblk} 
\usepackage{enumitem}

\newtheorem{lemma}{Lemma}
\newtheorem{theorem}{Theorem}
\newtheorem{corollary}{Corollary}
\newtheorem{remark}{Remark}
\newtheorem{proposition}{Proposition}
\newtheorem{definition}{Definition}

\usepackage{subcaption}
\usepackage{graphicx,xcolor}
\DeclareGraphicsExtensions{.pdf,.png,.jpg}
\usepackage{pgfplots}

\usepackage{import}
\usepackage{xifthen}
\usepackage{pdfpages}
\usepackage{transparent}

\pgfplotsset{compat=1.18}

\numberwithin{equation}{section}
\numberwithin{theorem}{section}
\numberwithin{lemma}{section}
\numberwithin{corollary}{section}
\numberwithin{proposition}{section}
\numberwithin{definition}{section}

\newcommand{\mypar}[1]{\noindent\textit{#1}\hspace{0.1em}}

\newcommand{\bx}{\boldsymbol{x}}
\newcommand{\by}{\boldsymbol{y}}

\newcommand{\bc}{\boldsymbol{c}}
\newcommand{\bg}{\boldsymbol{g}}
\newcommand{\be}{\boldsymbol{e}}
\newcommand{\blambda}{\boldsymbol{\lambda}}
\newcommand{\bsigma}{\boldsymbol{\sigma}}

\newcommand{\re}{\varepsilon}

\renewcommand{\O}{\mathcal{O}}

\newcommand{\bO}{\partial\O}

\renewcommand{\vec}{\boldsymbol}
\renewcommand{\Im}{\operatorname{Im}}

\newcommand{\inc}{\operatorname{inc}}
\newcommand{\app}{\operatorname{app}}
\newcommand{\trinorm}[1]{\vert\!\vert\!\vert #1 \vert\!\vert\!\vert}

\sethlcolor{cyan}

\title{Asymptotic models for time-domain scattering by small particles of
	arbitrary shapes} 
\date{}
\author[1]{Maryna Kachanovska}
\author{Adrian Savchuk}
\affil[1]{POEMS, CNRS, Inria, ENSTA, Institut Polytechnique de Paris, 91120 Palaiseau, France}

\begin{document}
	
	\maketitle
	
	\begin{abstract}
		In this work, we investigate time-dependent wave scattering by multiple small particles of arbitrary shape. To approximate the solution of the associated boundary-value problem, we derive an asymptotic model that is valid in the limit as the particle size tends to zero. Our method relies on a boundary integral formulation, semi-discretized in space using a Galerkin approach with appropriately chosen basis functions, s.t. convergence is achieved as the particle size vanishes
		\, rather than by increasing the number of basis functions. Since the computation of the Galerkin matrix involves double integration over particles, the method can become computationally demanding when the number of obstacles is large. To address this, we also derive a simplified model and consider the Born approximation to improve computational efficiency. For the high-order models, we provide an error analysis, supported and validated by numerical experiments.
	\end{abstract}

	\section{Introduction}
	In this work, we consider the problem of the three-dimensional time-dependent scattering by many particles of arbitrary shape, stated as a boundary-value problem for the wave equation. In the first place, we are interested in the asymptotic regime when the diameter of the particles tends to zero, while their number and the distance between particles remains fixed. 
	Our goal is to find an alternative, more computationally advantageous model for this problem. 
	
	Such models have already appeared in the literature, cf. in particular, \cite{sini_time_domain}, \cite{barucq}, \cite{korikov}.  Since the two latter papers propose an asymptotic model based on the matching asymptotic expansion, however, for a single particle only, we concentrate our discussion on the model suggested in \cite{sini_time_domain}, which handles multiple particle case. Its stated convergence rate\footnote{We would like to remark that the final goal of the work \cite{sini_time_domain} is to consider a different, more complicated asymptotic regime, for which this accuracy appears to be  sufficient since it is dominated by other sources of errors.} is $O(\varepsilon^2)$, with $\varepsilon$ being a characteristic size of the particles. However, for some geometric configurations, when, loosely speaking, $\varepsilon$ is not 'small enough', this model can suffer of spurious resonances, which lead to exponentially growing solutions in the time domain. See for a discussion of an analogous issue in two dimensions and \cite{adrian_phd} for an illustration of the existing issue in \cite{sini_time_domain}. 
	
	In the present manuscript we suggest a model that has a higher convergence rate $O(\re^3)$, and avoids geometry restrictions and source regularity requirements of \cite{sini_time_domain}. To do so, we adopt the approach of \cite{MK} to construct an asymptotic model. We reformulate it using a single-layer boundary integral equation and semi-discretize it in space with the help of a suitable 'asymptotic' Galerkin method, so that the convergence of the approximate solution is achieved by decreasing the asymptotic parameter rather than by increasing the number of basis functions.
	This solution is a priori stable in the time domain due to the coercivity properties of the underlying operator.
	
	The main difficulty of this method lies in the choice of appropriate basis functions for the asymptotic Galerkin space. In particular, it was shown in \cite{MK} that for the case of circular particles, using constant basis functions yields quadratic-order convergence for the scattered field. However, this choice is insufficient to ensure convergence in the general case of arbitrarily shaped obstacles.
	An alternative idea, inspired by the works of \cite{sini_frequency_domain, sini_time_domain}, is to use equilibrium densities, i.e., solutions of the single-layer boundary integral equation for the Laplacian with a constant right-hand side, as basis functions. We prove that this choice leads to the third-order convergence of the scattered field in 3D, and confirm this result by the numerical experiments.  
	
	Unlike the case of circles or spheres, where all entries of the Galerkin matrix can be computed explicitly, see \cite{MK}, the new method requires double integration, which is computationally expensive, especially as the number of particles increases.
	To address this issue, we also derive a simplified model that preserves the stability and convergence properties of the GFL model. We compare the GFL model with the simplified model and the Born approximation, and provide a theoretical error analysis, further supported by numerical experiments.
	
	This paper is structured as follows. In Section \ref{sec:prob_set_and_main_result}, we introduce the problem setting, notation, and definitions, as well as the asymptotic models considered in this paper, together with the main result on their stability and convergence. Section \ref{sec:model} presents the derivation of the GFL model along with its error analysis.
	A simplified model and the Born approximation are studied in Section \ref{sec:simplified_models}. Section \ref{sec:numerical_experiments} provides numerical validation of the theoretical results, along with a comparison of the asymptotic models.
	
	\section{Problem setting and main result}
	\label{sec:prob_set_and_main_result}
	\subsection{Problem setting}
	Let $ N \in \mathbb{N} $, and for each $ k = 1, \dots, N $, let $ \Omega_k \subset \mathbb{R}^3 $ be a Lipschitz domain with the boundary $\Gamma_k := \partial \Omega_k$.  Let $\Omega_k$ be pairwise disjoint, i.e. it holds that $\overline{\Omega_k} \cap \overline{\Omega_\ell} = \emptyset$, for all $k \neq \ell$. We denote by $ \boldsymbol{c}_k \in \Omega_k$ a fixed point in $\Omega_k$. 
	
	Let $ \re \in (0, 1] $ be a scaling parameter. The rescaled domains $ \Omega_k^{\re} $ are defined by
	\begin{align}
		\label{eq:rescaled_omega_k}
		\Omega^{\re}_k := \{ \boldsymbol{x} \in \mathbb{R}^3: \boldsymbol{x} = \re (\hat{\boldsymbol{x}} - \boldsymbol{c}_k) + \boldsymbol{c}_k, \,\hat{\boldsymbol{x}} \in \Omega_k \}.
	\end{align}
	By this definition, the rescaled domain is given by $ \Omega^{\re} = \bigcup_{k=1}^N \Omega_k^{\re} $, with boundary $ \Gamma^{\re} = \bigcup_{k=1}^N \Gamma_k^{\re} $. For brevity, we use the notations $\Omega$ and $\Gamma$ instead of $\Omega^1$ and $\Gamma^1$. We assume additionally that  
	\begin{align}
		\label{eq:omega_ball}
		\Omega_k \subseteq B_k, \quad B_k:=B(\boldsymbol{c}_k, r_k), \text{ so that }\Omega\subseteq B:=\cup_{k=1}^N B_k,
	\end{align}
	where $B(\boldsymbol{c}_k, r_k)$ denotes a ball centered at $\boldsymbol{c}_k$ of radius $r_k>0$. Evidently, $r_k\geq \operatorname{diam} \Omega_k/2$.  It is straightforward to see that $\Omega^{\re}_k \subseteq B^{\re}_k:= B(\boldsymbol{c}_k, r^{\re}_k)$, where $r^{\re}_k:=\re r_k$. 
	Additionally, let us assume that for all $k \neq \ell$,  $\overline{B_k} \cap \overline{B_{\ell}} = \emptyset$. This implies that $\overline{B_k^{\re}}  \cap \overline{B_{\ell}^{\re}}=\emptyset$ for all $0<\varepsilon<1$.  
	The minimal distance between balls $B_k^{\re}$ is denoted by
	\begin{align}
		\label{eq:min_distance}
		d_*^{\re}  := \min_{k \neq {\ell}} \operatorname{dist}(B_k^{\re}, B_{\ell}^{\re}).
	\end{align}
	Remark that for all $0<\re<1$, $d_*^{\re}>0$ and $d_*^{\re} \leq \min_{k\neq \ell} \text{dist}(\Gamma^{\re}_k, \Gamma^{\re}_{\ell})$.
	\begin{figure}[h]
		\centering
		\fontsize{10}{12}\selectfont 
		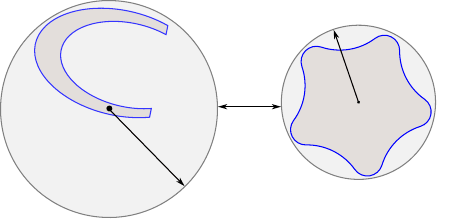
		\caption{An example of domains $\Omega_k$ and $\Omega_{\ell}$, where $d_{k \ell} := \operatorname{dist}(B_k, B_{\ell})$.}
		\vspace{-15pt}
	\end{figure}
	We assume that the field $u^{\operatorname{inc}}: \,\mathbb{R}^3\times \mathbb{R}_{\geq 0}\rightarrow \mathbb{R}$ satisfies the free space wave equation:
	\begin{align}
		\label{eq:fsw}
		\begin{array}{ll}
			\partial^2_t u^{\operatorname{inc}}(\boldsymbol{x}, t) - \Delta u^{\operatorname{inc}}(\boldsymbol{x}, t) = 0, & \quad (\boldsymbol{x}, t) \in \mathbb{R}^3 \times \mathbb{R}_{>0}, \\
			u^{\operatorname{inc}}(\boldsymbol{x}, 0) = u_0(\boldsymbol{x}), \quad \partial_t u^{\operatorname{inc}}(\boldsymbol{x}, 0) = u_1(\boldsymbol{x}), & \quad \boldsymbol{x} \in \mathbb{R}^3,
		\end{array}
	\end{align}
	with real-valued $u_0\in H^2(\mathbb{R}^3), \, u_1\in H^1(\mathbb{R}^3)$ being compactly supported in $B^c:=\mathbb{R}^3\setminus\overline{B}$. 
	We seek the solution to the sound-soft scattering problem. Namely, we look for the total field $ u^{\re, \operatorname{tot}} $ in the form $u^{\re} + u^{\operatorname{inc}} $, where $ u^{\re} $ is the scattered field that satisfies the wave equation in the exterior of the domain $\Omega^{\re}$, namely $
	\Omega^{\re, c}:=\mathbb{R}^3\setminus \overline{\Omega^{\re}}, $
	with the inhomogeneous Dirichlet boundary conditions:
	\begin{align}
		\label{main_problem}
		\begin{array}{ll}
			\partial^2_t u^{\re}(\boldsymbol{x}, t) - \Delta u^{\re}(\boldsymbol{x}, t) = 0, & \quad (\boldsymbol{x}, t) \in \Omega^{\re, c} \times \mathbb{R}_{>0}, \\
			\gamma_0 u^{\re} (\boldsymbol{x}, t)= g^{\re} (\boldsymbol{x}, t):=- \gamma_0 u^{\operatorname{inc}} (\boldsymbol{x}, t), & \quad (\bx, t)\in\Gamma^{\re} \times \mathbb{R}_{>0}, \\
			u^{\re}(\boldsymbol{x}, 0) = 0, \quad \partial_t u^{\re}(\boldsymbol{x}, 0) = 0, & \quad \boldsymbol{x} \in \Omega^{\re, c}. \\
		\end{array}
	\end{align}
	Remark that, because of the finite speed of the wave propagation and the support of the source term, $\left.u^{\inc}(t)\right|_B=0$ for all $t\in [0, \delta)$, with $\delta>0$ independent of $\re$. 
	
	It is well-known that solving the above problem numerically using classical methods, such as Finite Element Methods (FEM) combined with explicit time-stepping schemes, requires meshing finely in the vicinity of the small obstacles $\Omega^{\re}$, which can make efficient computations prohibitively expensive, see \cite{barucq}.
	
	On the other hand, a frequency-domain argument shows that $\| u^{\re} \|_{L^{\infty}(0,T; L^{\infty}(K))}$ is $O(\re)$ for any compact subset $K$ of $\Omega^{c}$. Nonetheless, in many practical applications (cf. e.g. \cite{imaging_time_reversal,dort,dort_original}), simply approximating $u^{\re}$ by zero appears to be insufficient. Therefore, our goal is to construct an approximate solution $u^{\re}_{\operatorname{app}}$ of the scattered filed $u^{\re}$ by deriving a higher-order asymptotic model valid in the regime as $ \re \to 0 $. 
	
	\subsection{Preliminary notation and definitions}
	
	\subsubsection{Functions on $\Gamma^{\re}_k$} First of all, where applicable, by bold symbols (e.g. $\boldsymbol{\varphi}$) we will denote quantities defined on $\Gamma^{\re}$, and with the index $k$ (e.g. $\varphi_k$) and a regular font we will mark their restrictions to $\Gamma_k^{\re}$; in other words, $\varphi_k:=\left.\boldsymbol{\varphi}\right|_{\Gamma_k^{\re}}$. 
	
	\paragraph{$H^{1/2}(\Gamma^{\re})$, $H^{-1/2}(\Gamma^{\re})$ spaces} We define the space $H^{1/2}(\Gamma^{\re})$ as a space of functions
	\begin{align*}%
		H^{1/2}(\Gamma^{\re}):=\{\boldsymbol{\varphi}\in L^2(\Gamma^{\re}):\, \|\boldsymbol{\varphi}\|^2_{H^{1/2}(\Gamma^{\re})}<+\infty\},
	\end{align*} 
	with $\|\boldsymbol{\varphi}\|^2_{H^{1/2}(\Gamma^{\re})}:=\sum\limits_{k=1}^N \|\varphi_k\|^2_{H^{1/2}(\Gamma^{\re}_k)}$.
	Each of the norms $H^{1/2}(\Gamma^{\re}_k)$ is defined as a (usual) Sobolev-Slobodeckij norm: $	\| \psi \|^2_{H^{1/2}(\Gamma^{\re}_k)}:= \| \psi \|^2_{L^2(\Gamma^{\re}_k)} + |\psi|^2_{H^{1/2}(\Gamma^{\re}_k)}$, with
	\begin{align}
		\label{eq:norm_H_half}
		| \psi |_{H^{1/2}(\Gamma^{\re}_k)}^2 = \iint_{\Gamma^{\re}_k\times \Gamma^{\re}_k}\frac{|\psi(\boldsymbol{x})- \psi(\boldsymbol{y})|^2}{|\boldsymbol{x}-\boldsymbol{y}|^3} d\Gamma_{\boldsymbol{x}} d\Gamma_{\boldsymbol{y}}.
	\end{align}
	By $H^{-1/2}(\Gamma^{\re})$, we denote the dual space to $H^{1/2}(\Gamma^{\re})$ with the following norm:
	\begin{align}
		\label{eq:norm_H_minus_half}
		\| \vec{\psi} \|_{H^{-1/2} (\Gamma^{\re})} =  \sup_{\boldsymbol{\varphi}  \in H^{1/2}(\Gamma^{\re}) \setminus \{0\}}  \frac{| \langle \vec{\psi}, \vec{\varphi} \rangle_{-1/2, 1/2} |}{\| \vec{\varphi} \|_{H^{1/2}(\Gamma^{\re})}}, \quad \vec{\psi} \in H^{-1/2}(\Gamma^{\re}),
	\end{align}
	where $\langle \cdot, \cdot \rangle_{-1/2, 1/2}$, refers to the duality pairing between $H^{-1/2}(\Gamma^{\re})$ and $H^{1/2}(\Gamma^{\re})$ spaces, and for $\vec{\varphi}, \vec{\psi} \in L^2(\Gamma^{\re})$, we have that
	\begin{align*}%
		\langle \vec{\varphi}, \vec{\psi} \rangle_{-1/2, 1/2} = \int_{\Gamma^{\re}} {\vec{\varphi}(\bx)} {\vec{\psi}(\bx)} d \Gamma_{\bx}=	\langle \vec{\psi}, \vec{\varphi} \rangle_{1/2, -1/2}.
	\end{align*}
	Often we will write instead $\langle \cdot, \cdot \rangle$. A standard argument shows that
	\begin{align}
		\label{eq:Hminhalf}
		\begin{split}
			\| \vec{\psi} \|^2_{H^{-1/2}(\Gamma^{\re})} = \sum_{k=1}^{N} \| \psi_k \|^2_{H^{-1/2} (\Gamma^{\re}_k)}.
		\end{split}
	\end{align}	
	
	\subsubsection{Fourier-Laplace transform}
	\label{sec:ft}
	Let $X$ be a separable Hilbert space and $f: \mathcal{S}(\mathbb{R}) \to X$ be a tempered $X$-valued distribution; 
	$f$ is a causal if $\operatorname{supp}f\subseteq [0, +\infty)$. 
	This class of distributions is denoted by $\operatorname{CT}(X)$ (see \cite[Section 2.1]{sayas}). 
	For sufficiently regular $ f \in \operatorname{CT}(X) $,  the Fourier-Laplace transform is given by 
	\begin{align}
		\label{eq:fourier_laplace_transforme}
		\hat{f}(\omega) = \int_{0}^{\infty} e^{i \omega t} f(t) \, dt, \quad \omega\in \mathbb{C}^+:=\{z\in \mathbb{C}: \, \Im z>0\},
	\end{align}
	via the Bochner integral. We will need additionally a class of vector-valued distributions  $\operatorname{TD}(X)$ which are defined as follows 
	\begin{align*}%
		&\operatorname{TD}(X) :=\{\varphi\in \operatorname{CT}(X): ~ \exists m\geq 0, ~ \text{and} ~ F\in \mathcal{C}^0_{pol}(\mathbb{R}; X), ~ \text{s.t.} ~ \varphi=\left(\frac{d}{dt}\right)^mF\},  \\
		&\mathcal{C}^0_{pol}(\mathbb{R}; X) =\{f\in \mathcal{C}^0(\mathbb{R}; X): f=0 ~ \text{on} ~(-\infty,0) \\
		&\hspace{140pt}\text{and} ~ \exists p\geq 0: ~ \|(1+t)^{-p}f(t)\|_{L^{\infty}(\mathbb{R}_{>0}; X)}<\infty\},
	\end{align*}
	where $d/dt$ is the distributional derivative. In other words, $\operatorname{TD}(X)$ are distributional derivatives of continuous causal functions of polynomial growth.
	The distributions $\operatorname{TD}(X)$ satisfy the following property (cf. Proposition 3.1.2 and 3.1.3 of \cite{sayas}): 
	\begin{align*}%
		\varphi\in \operatorname{TD}(X) \iff &\omega\mapsto\hat{\varphi}(\omega) \text { is $X$-analytic in }\mathbb{C}^+\\
		&\text{ and } \|\hat{\varphi}(\omega)\|_{X}\leq C (1+|\omega|)^p\max(1,(\Im \omega)^{-k}), ~ \text{for some}  ~ p,k\geq 0. 
	\end{align*}
	In what follows, we will also need a larger space \cite{bamberger_ha_duong} 
	\begin{align*}%
		L'_{+}(X):=\{\varphi\in \mathcal{D}'(X), \,\text{ for which there exists }\sigma_0\geq 0, \text{ s.t. } \mathrm{e}^{-\sigma_0 t}\varphi\in \operatorname{CT}(X)\},
	\end{align*}
	which, roughly speaking, allows for the exponential growth of distributions, and for which we can define the Fourier-Laplace transform in the usual manner. The Fourier-Laplace transforms of such distributions are characterized by the following property, cf. \cite{bamberger_ha_duong} or \cite[Chapter 16]{dautray_lions_vol5}
	\begin{align*}%
		\varphi\in {L}'_+(X) \iff &\exists\, \sigma\geq 0, \text{s.t. } \omega\mapsto\hat{\varphi}(\omega) \text { is $X$-analytic in }\mathbb{C}^+_{\sigma}=\{z\in \mathbb{C}: \Im z>\sigma\}\\
		& \text{ and }
		\|\hat{\varphi}(\omega)\|_{X}\leq C (1+|\omega|)^p,  \quad 
		\text{for some} \quad p\geq 0.  
	\end{align*}
	Let us recall as well that the Laplace transform is injective on $L'_+(X)$.
	
	While all the quantities in this paper are defined for $t\geq 0$, in order to apply the above formalism to them, we extend them by $0$ to $(-\infty,0)$. 
	
	\textbf{Convolutions. }In what follows, we will use the following notation for convolutions. Recall that given two causal tempered distributions $\operatorname{CT}(\mathbb{R})$, $\varphi, \psi$, their convolution is denoted by $\varphi*\psi$. In particular, for $\varphi, \psi$ sufficiently regular, we have 
	\begin{align*}
		(\varphi*\psi)(t)=\int_0^t\varphi(t-\tau)\psi(\tau)d\tau.
	\end{align*}
	We will also use the notation $\varphi(t)*\psi(t)$, where it facilitates the understanding. 
	\subsubsection{Auxiliary notaiton}
	Occasionally, we will write $a \lesssim b$ to indicate that there exists a constant $C>0$ that depends on the shapes of the particles $\Omega_j$ only (i.e. independent of $\re, N, \underline{d}^{\re}_{*}$, $\omega$ in the Fourier-Laplace transform), s.t. $a\leq Cb$. 
	\subsection{Principal results of the paper}
	\label{sec:principal_result}
	
	In the present paper we suggest three asymptotic models. The first one will serve as a basis, for which we will derive all the results. As we see later, an intuition for it comes from the numerical analysis, see \cite{MK}, and the convergence analysis is inspired by similar ideas. The second model is much simpler; it will be obtained as a perturbation of the first model.  Finally, the third model is a Born-style approximation, which appears to have a rather good convergence order in our setting (compared to two dimensions, cf. \cite{hazard_cassier}, where the existing estimates indicate that it converges very slowly). To state these models, let us introduce, for $t, r \geq 0$, the kernel of the retarded potential operator in $\mathbb{R}^3$:
	\begin{align*}%
		\mathcal{G} (r, t) = \frac{\delta(t -r)}{4\pi r}.
	\end{align*} 
	\paragraph{An asymptotic model 1}The approximate scattered field takes the following form:
	\begin{align}%
		\label{eq:ureg_sc}
		u^{\re}_{G,\operatorname{\operatorname{app}}} (\boldsymbol{x}, t) &= \sum_{k=1}^N \int_{\Gamma^{\re}_k} 
		\mathcal{G}( |\boldsymbol{x}-\boldsymbol{y}|,t)*\lambda^{\re}_{G,k}(t)  \sigma_k^{\re}(\boldsymbol{y}) \, d\Gamma_{\boldsymbol{y}}\\
		&= 
		\sum_{k=1}^N \int_{\Gamma^{\re}_k} 
		\frac{\lambda^{\re}_{G,k}(t-|\bx-\by |)}{4\pi | \bx-\by |}  \sigma_k^{\re}(\boldsymbol{y}) \, d\Gamma_{\boldsymbol{y}}, 
		\quad (\boldsymbol{x}, t) \in \Omega^{\re, c} \times \mathbb{R}_{>0},
	\end{align}
	where the time-dependent real-valued functions $ \{ \lambda^{\re}_{G,k}(\cdot) \}_{k=1}^N $ solve the following linear convolutional in time system:
	\begin{align}
		\label{eq:tdsys}
		\begin{split}
			& \sum_{k=1}^N (K^{\re}_{G,\ell k }*\lambda^{\re}_{G,k})(t)= q_{G,\ell }^{\re}(t), \text{ where }\\
			&(K^{\re}_{G,\ell k }*\lambda^{\re}_{G,k})(t)=\int_{\Gamma_k^{\re}}\int_{\Gamma_{\ell}^{\re}}\frac{\lambda^{\re}_{G,k}(t-|\bx-\by|)}{4\pi|\bx-\by|}\sigma_k^{\re}(\bx)\,\sigma_{\ell}^{\re}(\by)d\Gamma_{\bx}\, d\Gamma_{\by},\\
			&q^{\re}_{G, \ell}(t)=-\int_{\Gamma^{\re}_{\ell}} u^{\operatorname{inc}} (\boldsymbol{x}, t)
			\sigma_{\ell}^{\re}(\boldsymbol{x})d\Gamma_{\vec{x}} =\int_{\Gamma^{\re}_{\ell}} {g}^{\re}_{\ell}(\boldsymbol{x},t)
			\sigma_{\ell}^{\re}(\boldsymbol{x})d\Gamma_{\vec{x}},
		\end{split}
	\end{align}
	and $\sigma_k^{\re}\in H^{-1/2}(\Gamma^{\re}_k)$  satisfies the following problem: $ \int_{\Gamma_k^{\re}} \frac{\sigma_k^{\re}(\boldsymbol{y})}{4 \pi | \boldsymbol{x} - \boldsymbol{y}|} \, d\Gamma_{\boldsymbol{y}} = 1 \text{ on }\Gamma_k^{\re}$. This model is derived in Section \ref{sec:model}.
	\begin{remark}
		The quantities functions ${\sigma}^{\re}_k$ depend on $\re$ trivially, see \eqref{eq:sigmakeps}. This means that in practice the equilibrium densities $\sigma_k^{\re}$ can be computed once for $\re=1$ and reused for all $\re \in (0,1)$.
		
		The numerical implementation of the above model relies as well on computing surface integrals $\Gamma_{k}^{\re}\times \Gamma_{\ell}^{\re}$ in \eqref{eq:tdsys}. This can be computationally quite heavy, that is why in what follows we propose a model with similar properties, which avoids evaluation of these surface integrals. 
	\end{remark}
	
	\paragraph{An asymptotic model 2}
	Let $\sigma_k^{\re}$ be like in \eqref{eq:tdsys}.
	Set
	\begin{align}
		\label{eq:ckre}
		c_k^{\varepsilon}:=\int_{\Gamma^{\varepsilon}_k}\sigma^{\varepsilon}_k(\vec{y})d\Gamma_{\vec{y}}\in \mathbb{R}, \quad \vec{p}_k^{\varepsilon}:=\int_{\Gamma^{\varepsilon}_k}(\vec{y}-\bc_k)\sigma_k^{\varepsilon}(\vec{y})d\Gamma_{\vec{y}}\in \mathbb{R}^3.
	\end{align}
	With these definitions, the approximated scattered field takes the following form (where the index '$s$' stands for simplified): 
	\begin{align*}%
		{u}^{\varepsilon}_{s,\operatorname{\operatorname{app}}}(\vec{x},t)&=\sum\limits_{k=1}^N\frac{{\lambda}^{\varepsilon}_{s,k}(t-|\vec{x}-\bc_k|)}{4\pi|\vec{x}-\bc_k|}\left(c_k^{\varepsilon}+\frac{(\vec{x}-\bc_k)}{4\pi|\vec{x}-\bc_k|^3}\cdot \vec{p}_k^{\varepsilon}\right)\\
		&+\sum\limits_{k=1}^N\frac{\partial_t{\lambda}^{\varepsilon}_{s,k}(t-|\vec{x}-\bc_k|)}{4\pi|\vec{x}-\bc_k|^2}(\vec{x}-\bc_k)\cdot\vec{p}_k^{\varepsilon},
	\end{align*}
	where the time-dependent real-valued functions $ \{ \lambda^{\re}_{s,k}(\cdot) \}_{k=1}^N $ solve the following linear convolutional in time system:
	\begin{align}
		\label{eq:tdsys2}
		\begin{split}
			& \sum_{k=1}^N (K^{\re}_{s, \ell k }*\lambda^{\re}_{s, k})(t)={q}_{s}^{\re,\ell }(t), \text{ with }	\\
			&q^{\re}_{s, \ell}(t)=- u^{\operatorname{inc}}(\vec{c}_{\ell},t)c_{\ell}^{\re}-\nabla u^{\operatorname{inc}}(\vec{c}_{\ell}, t)\cdot\vec{p}_{\ell}^{\re}, \\
			& K^{\re}_{s, \ell k }(t) = \mathcal{F}^{-1} (\mathbb{M}^{\re}_{s, \ell k})(t),
		\end{split}
	\end{align}
	and the matrix $\mathbb{M}^{\re}_{s} $ is given by
	\begin{align*}%
		\mathbb{M}^{\re}_{s, \ell k}(\omega)=
		\left\{
		\begin{array}{ll}
			4 \pi i \omega (\rho^{\re}_k)^2 j_0(\omega \rho^{\re}_k) h^{(1)}_0(\omega \rho^{\re}_k), &k = \ell, \\
			4 \pi i \omega \rho^{\re}_k \rho^{\re}_{\ell} j_0(\omega \rho^{\re}_k) j_0(\omega \rho^{\re}_{\ell}) h^{(1)}_0(\omega |\bc_k - \bc_{\ell}| ), &k \neq \ell.
		\end{array}
		\right.
	\end{align*}
	with $\rho^{\re}_k:= c^{\re}_k/4 \pi$, and  $j_0$, $h^{(1)}_0$ being the spherical Bessel and Hankel functions of the first kind \cite[Chapter 10.47]{nist}. This model will be derived in Section \ref{sec:simplified}.
	\begin{remark}
		Numerical realization of this model with the help of convolution quadrature methods \cite{lubich88_1, lubich88_2, lubich94}, compared to \eqref{eq:tdsys}, relies solely on evaluation of $c_k^{\re}$ and $\vec{p}_k^{\re}$. 
	\end{remark}
	\paragraph{An asymptotic model 3 (Born approximation)}
	Let $\sigma_k^{\re}$ be like in \eqref{eq:tdsys} and $c_k^{\re}$ as in \eqref{eq:ckre}. With these definitions, the approximated scattered field is given by: 
	\begin{align}
		\label{eq:tdsysborn}
		{u}^{\varepsilon}_{B,\operatorname{\operatorname{app}}}(\vec{x},t)=-\sum\limits_{k=1}^N\frac{u^{\operatorname{inc}}(\vec{c}_k,t-|\vec{x}-\bc_k|)}{4\pi|\vec{x}-\bc_k|}c_k^{\varepsilon}.
	\end{align}
	\begin{remark}
		Compared to the previous models, the above model is computationally most efficient, and relies solely on knowing the capacitances $c_k^{\re}$.
	\end{remark}
	\paragraph{Well-posedness, stability and the error estimates} 
	Let us introduce the following definitions. The first definition concerns the geometry.
	\begin{definition}
		\label{def:admissible}
		A family $(\Omega^{\re})_{0<\re\leq 1}$ is admissible if $d^{\re}_*>0$ for all $0<\re\leq 1$. 
	\end{definition}
	It is straightforward to see that $\re\mapsto d^{\re}_*$ is increasing in $\re>0$, thus $(\Omega^{\re})_{0<\re\leq 1}$ is admissible if and only if $d^{1}_*>0$.
	
	Next, we need to define a notion of well-posedness for those models that are based on the computation of auxiliary quantities $t\mapsto\vv{\lambda}^{\re}(t)\in \mathbb{R}^N$, namely \eqref{eq:tdsys} and \eqref{eq:tdsys2}.
	\begin{definition}
		\label{def:well_posedness}
		We will call an asymptotic model well-posed if, for any admissible geometry, the following holds true.
		
		There exist $\ell_T,\ell_S, \ell\in \mathbb{N}_0$ (where S and T stand for "space" and "time", resp.), s.t., provided a solution to \eqref{eq:fsw} $u^{\operatorname{inc}}\in \operatorname{TD}(H^{\ell_S}(\mathbb{R}^3))\cap C^{\ell_T}(\mathbb{R}_{\geq 0}; H^{\ell_S}(\mathbb{R}^3))$, the asymptotic model admits a unique solution $\vv{\lambda}^{\re}\in L'_{+}(\mathbb{R}^N)$, and, moreover,  $\vv{\lambda}^{\re}\in \operatorname{TD}(\mathbb{R}^N)\cap C^{\ell}(\mathbb{R}_{\geq 0}; \mathbb{R}^N)$.
	\end{definition}
	%
	To define a stable and convergent models, we will restrict our attention to the behaviour of the scattered field on the compact sets outside\footnote{This is done for technical reasons, since some of our convergence estimates rely on the fact that $K$ does not intersect any of convex hulls of $\Omega^{\re}_k$.} of $B=\cup_{k=1}^NB_k$, cf. \eqref{eq:omega_ball}. 
	\begin{definition}
		We will call an asymptotic model which yields an approximate solution $u^{\re}_{\operatorname{\operatorname{app}}}(\vec{x},t)$  uniformly stable if there exist constants $m, \ell_S, \ell_T\in \mathbb{N}_0$, s.t. the following holds true. \textit{For any admissible geometric configuration} $\Omega^{\re}$, all compacts $K$ with $K\cap \overline{B}=\emptyset$, there exists $C_{\operatorname{geom},K}:=C_{\operatorname{geom}}(d_*^{\re}, \Omega, K)\leq C_{\operatorname{geom}}(d^1_*, \Omega, K)<+\infty$,  s.t. the following bound holds true uniformly for all $0<\varepsilon\leq 1$, $t\geq 0$:
		\begin{align*}%
			\sup\limits_{\vec{x}\in K}|u^{\re}_{\operatorname{\operatorname{app}}}(\vec{x},t)|\leq (1+t)^mC_{\operatorname{geom},K}\|u^{\operatorname{inc}}\|_{H^{\ell_T}(0,t; H^{\ell_S}(\mathbb{R}^3))}.
		\end{align*}
		We will call such a model convergent of order $p$, if, additionally, it holds that 
		\begin{align*}%
			\sup\limits_{\vec{x}\in K}|u^{\re}_{\operatorname{\operatorname{app}}}(\bx,t)-u^{\re}(\bx,t)|\leq \varepsilon^{p+1}(1+t)^mC_{\operatorname{geom},K}\|u^{\operatorname{inc}}\|_{H^{\ell_T}(0,t; H^{\ell_S}(\mathbb{R}^3))},.
		\end{align*}
	\end{definition}
	for all $0<\varepsilon\leq 1$ and $t\geq 0$. From the frequency-domain reasoning, it follows that for a single-particle case $N=1$, for a fixed $\vec{x}\in \Omega^c$, the scattered field behaves as $u^{\re}(\bx,t)=O(\re)$ (with a hidden constant that depends on $t$), thus $p$ in the above definition refers to a relative error. E.g., with this convention, the model of \cite{sini_time_domain} under our assumptions is convergent of order $p=1$.  
	\begin{remark}
		At a first sight, the above notion of stability is quite weak since it does not reflect the decay of scattered waves in compact sets of $\mathbb{R}^3$, allowing for a polynomial growth in time. However,  ensuring that asymptotic models do not exhibit even exponential growth is not fully trivial, cf. \cite{MK} and geometrical conditions in \cite{sini_time_domain}. 
	\end{remark}
	\begin{theorem}
		\label{theorem:convergence_of_models}
		The models \eqref{eq:tdsys}, \eqref{eq:tdsys2} are well-posed; together with the model \eqref{eq:tdsysborn}, they  are uniformly stable. 
		
		The models \eqref{eq:tdsys} and \eqref{eq:tdsys2} are convergent of order $p=2$.
		
		The model \eqref{eq:tdsysborn} is convergent of order $p=1$. 
	\end{theorem}
	\begin{proof}
		The statement about the model \eqref{eq:tdsys} is detailed and proven in Theorems \ref{theorem:well_posedness}, \ref{theorem:point_wise_estimate}. The result for the model \eqref{eq:tdsys2} is proven in Section \ref{sec:simplified}. The results for the Born model \eqref{eq:tdsysborn} can be deduced from the estimates of Section \ref{sec:simplified}, see Appendix \ref{app:born_analysis} for details.
	\end{proof}
	The statement about the convergence of the model \eqref{eq:tdsysborn} for spherical particles can be found in \cite{barucq}, but, up to our knowledge, no proof exists for scattering by multiple particles. It is quite interesting to compare this convergence statement with the known results in 2D, where the Born model in the frequency domain seems to yield a slow convergence (approximating the field with $O(\log^{-2}\re)$ absolute accuracy), see \cite{hazard_cassier}. 
	
	In \cite{sini_frequency_domain}, it was remarked that the first order model can be constructed by knowing particle capacitances only; in other words, the field scattered by spheres and by domains of the same capacitances would agree up to $O(\re^2)$-term. The model \eqref{eq:tdsys2} yields the approximation of the scattered field up $O(\re^3)$-term, and requires knowing capacitances and the first moments $\vec{p}_k^{\re}$ only. It is not difficult to see that in the case when all particles have a reflectional symmetry with respect to coordinate axes centered in $\vec{c}_k$, the first moments $\vec{p}_k^{\re}=\vec{0}$, and in this case approximating particles by spheres of equivalent capacitances yields the model of order $p=2$ \footnote{The authors are grateful to David Hewett (UCL, UK) for this question.}. 
	
	\begin{remark}
		The models \eqref{eq:tdsys} and \eqref{eq:tdsysborn} can be shown to be well-posed and stable for any geometric configuration with  $\operatorname{dist}(\Omega_j^{\re}, \Omega_k^{\re})>0$ rather than $\operatorname{dist}(B^{\re}_j, B^{\re}_k) >0 $; however, such a stability proof for \eqref{eq:tdsys} is more technical, and hence we omit it here. 
	\end{remark}
	
	\section{Derivation of the asymptotic model (\ref{eq:tdsys})}
	\label{sec:model}	
	
	To derive the asymptotic model \eqref{eq:tdsys}, we seek the scattered field $u^{\re}$ as the convolution of the retarded potential with a density $\vec{\lambda}^{\re}: \, \Gamma^{\re}\times \mathbb{R}_{>0}\rightarrow \mathbb{R}$,
	\begin{align}
		\label{eq:retarded_formulation}		
		u^{\re}(\boldsymbol{x}, t) = (\mathcal{S}^{\re} * \blambda^{\re})(\boldsymbol{x}, t) = \sum_{k=1}^{N} \int_{\Gamma^{\re}_k}  \mathcal{G}(|\boldsymbol{x} - \boldsymbol{y}|, t)* \lambda^{\re}_k (\boldsymbol{y}, t)  d\Gamma_{\boldsymbol{y}},
	\end{align}
	for $(\boldsymbol{x}, t) \in \Omega^{\re,c} \times \mathbb{R}_{>0}$. For the moment we will not make precise the function spaces, referring the interested reader to the monographs \cite{sayas, banjai_sayas}, or the seminal work  on well-posedness of TDBIE \cite{bamberger_ha_duong}, see as well \cite{lubich94}. In order to restore $\lambda^{\re}_k$, we apply the trace operator to both sides of \eqref{eq:retarded_formulation}, and using the Dirichlet boundary conditions \eqref{main_problem} and continuity of the single-layer potential across the boundary yields the new problem:
	\begin{align}
		\label{eq:SL_BIE_TD}
		(S^{\re} * \blambda^{\re}) (\boldsymbol{x},t) = \bg^{\re}(\boldsymbol{x}, t), \quad (\boldsymbol{x}, t) \in \Gamma^{\re} \times \mathbb{R}_{>0}.
	\end{align}
	Here, for sufficiently regular $\blambda^{\re}$<
	\begin{align*}%
		(S^{\re} * \blambda^{\re}) (\boldsymbol{x},t) =\int_{\Gamma^{\re}} \frac{\blambda^{\re}(\boldsymbol{y}, t-|\boldsymbol{x}-\boldsymbol{y}|)}{4\pi |\boldsymbol{x}-\boldsymbol{y}|} d\Gamma_{\boldsymbol{y}}\equiv\sum\limits_{k=1}^N\int_{\Gamma_k^{\re}} \frac{\lambda_k^{\re}(\vec{y}, t-|\vec{x}-\vec{y}|)}{4\pi|\vec{x}-\vec{y}|}d\Gamma_{\vec{y}}.
	\end{align*}
	The problem \eqref{eq:SL_BIE_TD} can be shown to be  well-posed in a suitable functional framework (Laplace transformable causal distributions), cf. the seminal work of Bamberger and Ha-Duong  \cite{bamberger_ha_duong} or the monograph \cite{sayas} for more details.  In particular, we have
	\begin{theorem}[\cite{bamberger_ha_duong}, \cite{sayas}]
		For all $\vec{g}^{\re}\in \operatorname{TD}(H^{1/2}(\Gamma^{\re}))$, there exists a unique solution to \eqref{eq:SL_BIE_TD} $\vec{\lambda}^{\re}\in L_+'(H^{-1/2}(\Gamma^{\re}))$. Additionally, this solution lies in $\operatorname{TD}(H^{-1/2}(\Gamma^{\re}))$. 
	\end{theorem}
	\subsection{An abstract asymptotic Galerkin method}
	As suggested in \cite{MK}, we will derive an asymptotic model as a semi-discretization of the above problem in space by a suitable Galerkin method, with a set of basis functions given by $\mathcal{V}_G^{\re}$. This method satisfies the following conditions:
	\begin{itemize}[leftmargin=*]
		\item 
		$\text{dim}\mathcal{V}^{\re}_G= N_B$, where $N_B$ is fixed for all $\re \in (0,1)$; 
		\item the basis functions in $\mathcal{V}^{\re}_G$ behave in a trivial manner with respect to $\re$ (the notion of 'trivial' will be clear in Section \ref{sec:asymtpotic_galerkin});
		\item 
		the convergence will be ensured by $\re \to 0$, instead of $N_B \to \infty$.
	\end{itemize}
	The main advantage of this approach lies in the fact that, due to coercivity properties of the single-layer boundary integral operators, such Galerkin semi-discretizations yield stable in the time domain models, see \cite{banjai_sauter}.
	
	As in the classical Galerkin semi-discretization in space, let us choose a finite-dimensional space $\mathcal{V}^{\re}_G \subseteq H^{-1/2}(\Gamma^{\re})$, which is defined as
	$	\mathcal{V}^{\re}_G = \text{span}\{ \boldsymbol{e}_k^{\re}\}_{k=1}^N,
	$
	where the basis linear forms $\boldsymbol{e}_k^{\re}\in H^{-1/2}(\Gamma^{\re})$ are non-trivial on $\Gamma_k^{\re}$ and vanish on $\Gamma^{\re} \setminus \Gamma_k^{\re}$, i.e. $e^{\re}_{k,\ell}=0$ for $\ell\neq k$. 
	We look for the approximate density
	\begin{align}
		\label{eq:approximate_density}
		\blambda^{\re}_G(\boldsymbol{x}, t) = \sum_{k=1}^N \lambda^{\re}_{G, k}(t) \be_k^{\re}(\boldsymbol{x}), \quad (\boldsymbol{x}, t) \in \Gamma^{\re} \times \mathbb{R},
	\end{align}
	where $\{ \lambda^{\re}_{G, k} (\cdot) \}_{k=1}^N$ are unknown time-dependent causal functions.
	
	Testing the boundary equation \eqref{eq:SL_BIE_TD} with $\vec{e}_k^{\re}$ and using \eqref{eq:approximate_density}, for $k=1,\ldots,N$ and $t>0$,  yields the following convolutional in time linear system  
	\begin{align}
		\label{eq:general_asymtotic_model}
		\langle (S^{\re, kk} *  \lambda^{\re}_{G, k} e_{k,k}^{\re})(t), e_{k,k}^{\re} \rangle_{\Gamma^{\re}_k} + \sum_{\ell\neq k} \langle  (S^{\re, k\ell} *  \lambda^{\re}_{G, \ell} e_{\ell,\ell}^{\re})(t), e_{k,k}^{\re} \rangle_{\Gamma^{\re}_k} = \langle g^{\re}_k(\cdot, t), e_{k,k}^{\re}  \rangle_{\Gamma^{\re}_k},
	\end{align}
	where $S^{\re, k\ell}$ denotes the restricted to $\Gamma^{\re}_{k}$ single layer operator defined on $\Gamma^{\re}_{\ell}$, i.e., for causal and sufficiently regular $\varphi$, 
	\begin{align*}%
		(S^{\re, k{\ell}} * \varphi) (\boldsymbol{x}, t) := \int_{\Gamma^{\re}_{\ell}}  \frac{ \varphi(\boldsymbol{y},t-|\bx-\by|)}{4\pi|\bx-\by|} d\Gamma_{\boldsymbol{y}}, \quad (\boldsymbol{x}, t) \in \Gamma^{\re}_k \times \mathbb{R}_{>0}.
	\end{align*}
	Once the approximate densities are found, the field can be approximated by applying the single layer potential to \eqref{eq:approximate_density}:
	\begin{align}
		u^{\re}_{G,\operatorname{app}}(\vec{x},t)=\sum\limits_{k=1}^N \int_{\Gamma^{\re}_k}\frac{\lambda_{G,k}^{\re}(t-|\vec{x}-\vec{y}|)}{4\pi|\vec{x}-\vec{y}|}e_{k,k}^{\re}(\vec{y})d\Gamma_{\vec{y}}, \quad (\vec{x},t)\in \Omega^{\re,c}\times \mathbb{R}_{>0}.
	\end{align}
	
	\subsection{Choice of basis functions for $\mathcal{V}^{\re}_G$}
	\label{sec:asymtpotic_galerkin}
	In this section, we derive basis functions $\vec{e}_k^{\re}$ for $\mathcal{V}^{\re}_G$. This choice is inspired by the ideas from \cite{sini_time_domain}, namely by the proof of Theorem 2.5. For convenience of the reader, we present a very formal derivation of the basis functions here, referring to Theorem 2.5 of \cite{sini_time_domain} for a rigorous statement. 
	
	\paragraph{The choice of basis functions for the single-particle case} We start by considering the single-particle case $N=1$ with $\boldsymbol{c}_1=\boldsymbol{0}$.  Provided sufficiently regular data, the exact causal density $\lambda^{\re}\in C^1(\mathbb{R}; H^{-1/2}(\Gamma^{\re}))$ satisfies the boundary integral equation 
	\begin{align*}%
		\int_{\Gamma^{\re}} \frac{ \lambda^{\re}(\boldsymbol{y},t-|\bx-\by|)}{4 \pi |\boldsymbol{x} - \boldsymbol{y}|}d\Gamma_{\by}=-u^{\operatorname{inc}}(\boldsymbol{x}, t), \quad (\bx, t) \in \Gamma^{\re} \times \mathbb{R}_{>0}.
	\end{align*}
	The key idea is to rewrite the above on a rescaled domain $\Gamma$, by defining $\hat{\by}:=\re^{-1}\by$: 
	\begin{align*}%
		\int_{\Gamma}\frac{\lambda^{\re}(\re\hat{\by}, t-\re|\hat{\bx}-\hat{\by}|)}{4\pi|\hat{\bx}-\hat{\by}|}d\Gamma_{\hat{\by}}=-\re^{-1}u^{\operatorname{inc}}(\re\hat{\bx}, t), \quad (\hat{\bx}, t) \in \Gamma \times \mathbb{R}_{>0}.
	\end{align*}
	Next, we formally approximate 
	\begin{align*}%
		\lambda^{\re}(\re\hat{\by}, t-\re|\hat{\bx}-\hat{\by}|)\approx \lambda^{\re}(\re\hat{\by}, t), \quad u^{\operatorname{inc}}(\re\hat{\bx}, t)\approx u^{\operatorname{inc}}(\mathbf{0}, t),
	\end{align*}
	which yields the new identity 
	\begin{align*}%
		\int_{\Gamma}\frac{\lambda^{\re}(\re\hat{\by}, t)}{4\pi|\hat{\bx}-\hat{\by}|}d\Gamma_{\hat{\by}}\approx -\re^{-1}u^{\operatorname{inc}}(\mathbf{0}, t).
	\end{align*}
	We recognize in the left-hand side of the above the single layer boundary integral operator for the Laplace equation, defined via ${S}_0\varphi(\bx)=\int_{\Gamma}\frac{\varphi(\by)}{4\pi|\bx-\by|}d\Gamma_{\by}, \, \bx\in \Gamma$, while in the right-hand side a function constant on $\Gamma$, and thus 
	\begin{align*}%
		\lambda^{\re}(\re\hat{\by},t)\approx -\re^{-1}u^{\operatorname{inc}}(\mathbf{0}, t){S}_0^{-1}1.
	\end{align*}
	In other words, we will look for $\lambda^{\re}(\re \cdot,t)\in \operatorname{span}\{\sigma\}$, where $\sigma$ is a unique $H^{-1/2}(\Gamma)$ solution to ${S}_0\sigma =1$.  Now, let  $\sigma^{\re}$ denote $\sigma$ defined on $\Gamma^{\re}$, i.e. $\sigma^{\re}$ satisfies
	\begin{align*}%
		\int_{\Gamma^{\re}} \frac{\sigma^{\re}(\boldsymbol{y})}{4\pi |\boldsymbol{x} - \boldsymbol{y}|} d\Gamma_{\boldsymbol{y}} = 1, \quad \boldsymbol{x} \in \Gamma^{\re}.
	\end{align*}
	The scaling argument yields  
	\begin{align}
		\label{eq:scarg}
		\re \int_{\Gamma} \frac{\sigma^{\re}(\re \hat{\boldsymbol{y}})}{4\pi |\hat{\boldsymbol{x}} - \hat{\boldsymbol{y}}|}  d \Gamma_{\hat{\boldsymbol{y}}}= \int_{\Gamma} \frac{\sigma(\hat{\boldsymbol{y}})}{4\pi| \hat{\boldsymbol{x}} - \hat{\boldsymbol{y}} |} d \Gamma_{\hat{\boldsymbol{y}}}, \quad  \hat{\boldsymbol{x}} \in \Gamma,
	\end{align}
	which implies that $ \sigma^{\re}(\re \hat{\boldsymbol{x}}) = \re^{-1} \sigma(\hat{\boldsymbol{x}})  $, for $\hat{\boldsymbol{x}}  \in \Gamma$.
	Thus, we chose the equilibrium density $\sigma^{\re}$ as a basis function for $\mathcal{V}^{\re}_G$, and using the definition of the asymptotic model \eqref{eq:general_asymtotic_model} for the single-particle case yields that
	\begin{align*}%
		\langle (S^{\re} * \lambda^{\re}_G \sigma^{\re}) (\cdot, t) , \sigma^{\re} \rangle_{\Gamma^{\re}} = \langle g^{\re}(\cdot,t), \sigma^{\re} \rangle_{\Gamma^{\re}}, \quad t \in \mathbb{R}_{>0}.
	\end{align*}
	\paragraph{Generalization to the many-particle case}  In the general multiple-particle case, the Galerkin space takes the following form:
	\begin{align}%
		\nonumber
		&\mathcal{V}^{\re}_G = \text{span} \{ \bsigma^{\re}_k \}_{k=1}^N, \text{ where }\\
		\label{eq:bas}
		&\sigma^{\re}_{k,\ell}(\vec{x})=0 \text{ for }\vec{x}\in \Gamma^{\re}_{\ell},  \quad	\text{ and solves }\quad \int_{\Gamma_k^{\re}} \frac{\sigma^{\re}_{k,k}(\boldsymbol{y})}{4 \pi |\boldsymbol{x} - \boldsymbol{y}|} d\Gamma_{\boldsymbol{y}}=1,~\text{for}~ \bx \in \Gamma_k^{\re}.
	\end{align}
	For brevity, let us define $	\sigma_k^{\re}:=\sigma_{k,k}^{\re}$. 
	\begin{lemma}
		\label{lemma:Galerkin_space}
		The Galerkin space $\mathcal{V}^{\re}_G$ is well-defined and its dimension equals $N$. Moreover, $\boldsymbol{\sigma}_k^{\re}\in L^2(\Gamma^\re)$, $k=1, \ldots, N$. 
	\end{lemma}
	\begin{proof}
		Let $k \in \{1, \ldots, N \}$. By the well-known result \cite[Corollary 8.13]{mclean}, the equation in \eqref{eq:bas} admits a unique solution $\sigma^{\re}_{k}\in H^{-1/2}(\Gamma^{\re}_k)$. By Theorem 3 in \cite{costabel} (see also the references therein), $\sigma^{\re}_{k} \in L^2(\Gamma^{\re}_k)$, since $\int_{\Gamma^{\re}_k} \frac{\sigma^{\re}_{k}(\boldsymbol{y})}{4\pi |\boldsymbol{x} - \boldsymbol{y}|} d\Gamma_{\boldsymbol{y}} = 1 \in H^{1}(\Gamma^{\re}_k)$.
	\end{proof}
	
	%
	%
	%
	%
	Plugging in this definition into \eqref{eq:general_asymtotic_model} yields the final asymptotic model \eqref{eq:tdsys}:
	\begin{align}
		\label{eq:final_asymtotic_model}
		&\langle (S^{\re, kk} * \lambda^{\re}_{G,k} \sigma^{\re}_k) (t) , \sigma^{\re}_k \rangle_{\Gamma^{\re}_k} + \sum_{\ell\neq k} \langle (S^{\re, k\ell}  *\lambda^{\re}_{G, \ell} \sigma^{\re}_{\ell})(t) , \sigma^{\re}_k \rangle_{\Gamma^{\re}_k} = \langle g^{\re}(\cdot, t), \sigma^{\re}_k  \rangle_{\Gamma^{\re}_k},\\ 
		&u^{\re}_{G,\operatorname{\operatorname{app}}}(\vec{x},t)=\sum\limits_{k=1}^N\int_{\Gamma_k^{\re}}\frac{\lambda_{G,k}(t-|\vec{x}-\vec{y}|)}{4\pi|\vec{x}-\vec{y}|}\sigma_k^{\re}(\vec{y})d\Gamma_{\vec{y}}, \quad (\vec{x},t)\in \Omega^{\re,c}\times \mathbb{R}_{>0}.
	\end{align}
	
	\section{Stability and error analysis for the asymptotic model (\ref{eq:tdsys}) in the frequency domain}
	Our stability and error analysis relies on the link between the Galerkin semi-discretization in space of \eqref{eq:SL_BIE_TD} and the asymptotic model \eqref{eq:tdsys}. It is carried out in the Fourier-Laplace domain and subsequently translated into the time domain. We begin by reformulating the problem \eqref{eq:SL_BIE_TD} in the frequency domain using the Fourier-Laplace transform. The analysis is then divided into three parts. First, we derive an estimate for the density error in the Fourier-Laplace domain, expressed in terms of the norm of the data $\vec{g}^{\re}$. Next, we use this estimate to obtain a corresponding bound for the scattered field via the retarded potential formulation. Finally, we translate the resulting bounds into the time domain by applying the Plancherel identity. We will demonstrate that, in order to obtain the optimal bound on the scattered field, it is necessary to decompose the density error into two components using an appropriate decomposition of the solution space $H^{-1/2}(\Gamma^{\re})$.
	
	\subsection{Frequency-domain counterparts of (\ref{eq:tdsys}) and main results}
	The frequency-domain counterparts of \eqref{eq:retarded_formulation} with the single layer boundary integral equation \eqref{eq:SL_BIE_TD} in the Fourier-Laplace domain take the following form (we omit the dependence of unknowns and data  on $\omega\in \mathbb{C}^+$): with $G_{\omega}(r) = \frac{e^{i\omega r}}{4\pi r},\,  r > 0$, 
	\begin{align}
		\label{eq:SL_potential}
		&\hat{u}^{\re}(\boldsymbol{x}) = (\mathcal{S}^{\re}_{\omega} \hat{{\blambda}}^{\re})(\boldsymbol{x}) = \sum_{k=1}^{N} \int_{\Gamma^{\re}_k}  G_{\omega}(|\boldsymbol{x} - \boldsymbol{y}|) \hat{\lambda}^{\re}_k (\boldsymbol{y})  d\Gamma_{\boldsymbol{y}}, \quad \boldsymbol{x} \in \Omega^{\re, c}, \\
		\label{eq:SL_BIE_FD}
		&S^{\re}_{\omega} \hat{\boldsymbol{\lambda}}^{\re}(\boldsymbol{x}) 
		= \int_{\Gamma^{\re}} G_{\omega}(|\boldsymbol{x} - \boldsymbol{y}|)\, \hat{\boldsymbol{\lambda}}^{\re}(\boldsymbol{y})\, d\Gamma_{\boldsymbol{y}} 
		= \hat{\boldsymbol{g}}^{\re}(\boldsymbol{x}), \quad \boldsymbol{x} \in \Gamma^{\re}.
	\end{align}
	The Fourier-Laplace transform of the approximate density \eqref{eq:approximate_density} is defined via
	\begin{align}
		\label{eq:approximate_density_frequency_domain}
		\hat{\boldsymbol{\lambda}}^{\re}_G(\boldsymbol{x}, \omega) = \sum_{k=1}^N \hat{\lambda}^{\re}_{G, k}(\omega)\, \boldsymbol{\sigma}_k^{\re}(\boldsymbol{x}), \quad \boldsymbol{x} \in \Gamma^{\re}, 
	\end{align}
	and the final asymptotic model \eqref{eq:final_asymtotic_model} becomes
	\begin{align}
		\label{eq:as_model_fd}
		&\hat{\lambda}^{\re}_{G, k}(\omega)\, \langle S^{\re, kk}_{\omega} \sigma_k^{\re}, \sigma_k^{\re} \rangle_{\Gamma^{\re}_k}
		+ \sum_{\ell \neq k} \hat{\lambda}^{\re}_{G, \ell}(\omega)\, \langle S^{\re, k\ell}_{\omega} \sigma_{\ell}^{\re}, \sigma_k \rangle_{\Gamma^{\re}_k} 
		= \langle \hat{\vec{g}}^{\re}, \sigma_k^{\re} \rangle_{\Gamma^{\re}_k},  \\
		\label{eq:operator_S_kl}
		&S^{\re, k\ell}_{\omega} \varphi(\boldsymbol{x}) := \int_{\Gamma^{\re}_{\ell}} G_{\omega}(|\boldsymbol{x}-\boldsymbol{y}|)\varphi(\boldsymbol{y})d\Gamma_{\boldsymbol{y}}, \quad \boldsymbol{x} \in \Gamma^{\re}_k. 
	\end{align} 
	for $k = 1, \dots, N$. The approximate field then satisfies 
	\begin{align}
		\label{eq:fdsys}
		&\hat{u}^{\re}_{G,\operatorname{app}}(\boldsymbol{x},\omega) = (\mathcal{S}^{\re}_{\omega} \hat{{\blambda}}^{\re}_G)(\boldsymbol{x}) = \sum_{k=1}^{N} \hat{\lambda}^{\re}_{G,k} \int_{\Gamma^{\re}_k}  G_{\omega}(|\boldsymbol{x} - \boldsymbol{y}|) \sigma_k^{\re}(\vec{y}) d\Gamma_{\boldsymbol{y}} , \quad \boldsymbol{x} \in \Omega^{\re,c}.
	\end{align}
	In our analysis, we will not work with the coefficients of the expansion $\hat{\lambda}_{G,k}^{\re}$, but rather with quantities $\hat{\blambda}_G^{\re}$. 
	
	As a first step, we prove the following two results in the frequency domain, which, as we see later, by a standard argument, will yield Theorem \ref{theorem:well_posedness} and Theorem \ref{theorem:point_wise_estimate}. 
	\begin{theorem}
		\label{th:wp}
		Let $\omega\in \mathbb{C}^+$. Provided that $\hat{u}^{\inc}(\omega)\in H^3(B)$, the problem \eqref{eq:as_model_fd} admits a unique solution $\hat{\vec{\lambda}}_G^{\re}(\omega)\in H^{-1/2}(\Gamma^{\re})$, which additionally satisfies:
		\begin{align*}%
			\| \hat{\vec{\lambda}}^{\re}_G(\omega) \|_{H^{-1/2}(\Gamma^{\re})} \leq C (\underline{d}^{\re}_{*})^{-2} (1+  |\omega|)^2 \max(1, (\Im  \omega)^{-3}) \| \hat{{u}}^{inc}(\omega) \|_{H^3(B)}, 
		\end{align*}
		where the constant $C>0$ depends on $\Gamma_j, \, j=1,\ldots,N$ (shape of the scatterers) only. 
	\end{theorem}
	\begin{proof}
		Follows by combining Lemma \ref{lemma:stability_bound_lambda_G} and Proposition \ref{prop:data}.
	\end{proof}
	The next result links the error of the scattered field and the norm of the data.
	\begin{theorem}
		\label{th:conv}
		Let $K\subset B^c$ be compact; $\omega\in \mathbb{C}^+$. Provided that $\hat{u}^{\inc}(\omega)\in H^{3}(B)$, the error between $\hat{u}^{\re}$, defined in \eqref{eq:SL_potential}, and $\hat{u}^{\re}_{G,\operatorname{app}}$, defined in \eqref{eq:fdsys}, is bounded by
		\begin{align}
			\label{eq:bound_field_FD_v0}
			\sup_{\vec{x}\in K}|\hat{u}^{\re}(\bx,\omega) - \hat{u}^{\re}_{\operatorname{app}}(\bx,\omega)| &\leq\re^{3}  C \max(1, (\operatorname{Im}\omega)^{-9})(1 + |\omega|)^{10}  \|\hat{u}^{\inc}(\omega)\|_{H^3(B)},
		\end{align}
		where $C=C_{K,\Omega} N^{5/2} (\underline{d}^{\re}_*)^{-10}$, and $C_{K,\Omega}>0$ depends on $K$ and $\Gamma$.
	\end{theorem}
	\begin{proof}
		See Section \ref{sec:fd_bounds}.
	\end{proof}
	\begin{remark}
		The above result yields convergence in the frequency domain, and next we use the Plancherel identity to recover the time-domain bounds. 
	\end{remark}
	The rest of this section is dedicated to the proofs of these results. In order to avoid repeating some arguments, we chose to concentrate on the proof of convergence of the model, which relies on refined results about the behaviour of the single-layer boundary integral operator; a rougher version of these results implies the well-posedness. To derive (and even state) these results, we start by introducing the necessary tools. 
	
	\subsection{Decompositions of spaces}
	\label{sec:orthogonal_projectors}
	\subsubsection{Decomposition of spaces orthogonal with respect to the $L^2$-scalar product}
	\label{sec:decompL2}
	Just like in \cite{MK}, the heart of the analysis will lie in the decomposition of the space $H^{-1/2}(\Gamma^{\re})$ into two subspaces: space of obstacle-wise constant functions  $\mathbb{S}_0$, see \eqref{eq:S0} (which is also a subspace of $H^{1/2}(\Gamma^{\re})$) and linear functionals vanishing on $\mathbb{S}_0$, namely $H^{-1/2}_{\perp}(\Gamma^{\re}$). However, such a decomposition is well-suited for the analysis on spherical particles (where, in particular, $\mathbb{S}_0$ coincides with $\mathcal{V}_G^{\re}$), while is less straightforward to use in the analysis for general obstacles. Thus, at some point, we will need to employ an auxiliary decomposition, analogous to the above, but which will allow us to treat independently $H^{-1/2}_{\perp}(\Gamma^{\re})$ and $\mathcal{V}^{\re}_G$, see Section \ref{sec:decompVg}.
	
	We define the space of functions $\mathbb{S}_0$, constant on each of the obstacles:   
	\begin{align}%
		\label{eq:S0}
		\vec{1}_k^{\re}(\vec{x}):=\left\{
		\begin{array}{ll}
			1, & \vec{x}\in \Gamma_k^{\re},\\
			0, & \text{ otherwise, }
		\end{array}
		\right. \text{ so that }	\mathbb{S}_0:=\operatorname{span}\{\mathbf{1}_k^{\re},\, k=1, \ldots, N\}.	
	\end{align} 
	Let $\mathbb{P}_0: \, H^{1/2}(\Gamma^{\re})\rightarrow H^{1/2}(\Gamma^{\re})$ be an $L^2(\Gamma^{\re})$-orthogonal projector on $\mathbb{S}_0$;  $\mathbb{P}_{\perp}:=\operatorname{Id}-\mathbb{P}_0$. With 
	\begin{align*}%
		H^{1/2}_{\perp}(\Gamma^{\re}):=\{\vec{\psi} \in H^{1/2}(\Gamma^{\re}): \, (\vec{\psi}, \vec{\varphi})_{L^2(\Gamma^{\re})}=0 \text{ for all }\vec{\varphi}\in \mathbb{S}_0\},
	\end{align*}
	it holds that 
	\begin{align}
		\label{eq:p0perp}
		\operatorname{Ker}\mathbb{P}_0=\operatorname{Im} \mathbb{P}_{\perp}=H^{1/2}_{\perp}(\Gamma^{\re}), \qquad \operatorname{Im}\mathbb{P}_0=\operatorname{Ker} \mathbb{P}_{\perp}=\mathbb{S}_0.
	\end{align}
	The following result basically summarizes the above.
	\begin{proposition}
		\label{prop:norm_equivalence} 
		The space $H^{\frac{1}{2}}(\Gamma^{\re})$ writes as a direct sum: $
		H^{\frac{1}{2}}(\Gamma^{\re}) = \mathbb{S}_0\dotplus H^{\frac{1}{2}}_{\perp} (\Gamma^{\re}),
		$ 
		with the decomposition orthogonal with respect to the $L^2(\Gamma^{\re})$- and $H^{\frac{1}{2}}(\Gamma^{\re})$-scalar products. As a corollary, 
		$\|\varphi\|^2_{H^{\frac{1}{2}}(\Gamma^{\re})}=\|\mathbb{P}_0\vec{\varphi}\|_{L^2(\Gamma^{\re})}^2+\|\mathbb{P}_{\perp}\vec{\varphi}\|_{H^{\frac{1}{2}}(\Gamma^{\re})}^2$. 
	\end{proposition}
	We will need a counterpart of the above decomposition on the space $H^{-1/2}(\Gamma^{\re})$. For this let us define adjoint projectors, $\mathbb{P}_0^*,\; \mathbb{P}_{\perp}^*: \, H^{-1/2}(\Gamma^{\re})\rightarrow H^{-1/2}(\Gamma^{\re})$, via 
	\begin{align}
		\label{eq:P0Phi}
		\langle \mathbb{P}_0^*\vec{\varphi}, \vec{\psi}\rangle_{-1/2,1/2,\Gamma^{\re}}=\langle \vec{\varphi}, \mathbb{P}_0\vec{\psi}\rangle_{-1/2,1/2,\Gamma^{\re}}, \quad \mathbb{P}_{\perp}^*=\mathbb{I}-\mathbb{P}_0^*,
	\end{align}
	and the dual counterpart of $H^{1/2}_{\perp}$:
	$$H^{-1/2}_{\perp} (\Gamma^{\re}) := \{ \psi \in H^{-1/2}(\Gamma^{\re}):\quad \langle \psi, \varphi \rangle_{-1/2, \, 1/2} =0, \quad \forall \varphi\in \mathbb{S}_0\}. $$
	From the definition \eqref{eq:P0Phi}, it follows that  $\mathbb{P}_0^*\vec{\varphi}=\mathbb{P}_0\vec{\varphi}$ for $\vec{\varphi}\in H^{1/2}(\Gamma^{\re})$, in other words $\mathbb{P}_0^*$ is an extension by density of $\mathbb{P}_0$ to $H^{-1/2}(\Gamma^{\re})$.
	
	By \cite[p. 23, Lemma 2.10]{mclean}, from the above definitions and \eqref{eq:p0perp}, it is clear that
	\begin{align*}%
		&\operatorname{Ker} \, \mathbb{P}^*_{\perp} = \operatorname{Im }\mathbb{P}^*_0=\mathbb{S}_0, &&\operatorname{Im}\mathbb{P}^*_{\perp} = \operatorname{Ker} \mathbb{P}^*_0=H^{-1/2}_{\perp}(\Gamma^{\re}). 
	\end{align*}
	As a corollary,  we have the following result.
	\begin{proposition}
		\label{ortho_decompositionHm12}
		$
		H^{-\frac{1}{2}}(\Gamma^{\re})=\mathbb{S}_0\dotplus H^{-\frac{1}{2}}_{\perp} (\Gamma^{\re}), $ where $\dotplus$ is a direct sum. 
	\end{proposition}
	\subsubsection{Decomposition of spaces orthogonal with respect to the single-layer induced scalar product}
	\label{sec:decompVg} As discussed in Section \ref{sec:decompL2}, we will also need an alternative decomposition of the spaces $H^{\pm \frac{1}{2}}(\Gamma^{\re})$. This decomposition, as we will see later (Section \ref{sec:analysis_freq_domain}), is well-adapted to the analysis of the Galerkin asymptotic method. In particular, let us introduce the following, localized on each obstacle, scalar product induced by the single layer boundary integral operator for the Laplace equation:
	\begin{align}
		\label{scalar_product}
		(\vec{\psi}, \vec{\varphi})_{S_0} &:= \sum_{k=1}^N \langle \psi_k, \overline{S^{\re, kk}_{0}\varphi_k } \rangle_{-1/2, \, 1/2, \, \Gamma^{\re}_k}, \quad \vec{\psi}, \, \vec{\varphi} \in H^{-1/2}(\Gamma^{\re}), \text{ where } \\
		\nonumber
		S^{\re, kk}_{0} \varphi (\boldsymbol{x}) &= \int_{\Gamma^{\re}_k} \frac{\varphi(\boldsymbol{y})}{4\pi | \boldsymbol{x} - \boldsymbol{y} |} d\Gamma_{\boldsymbol{y}}, \quad \boldsymbol{x} \in \Gamma^{\re}_k, \quad \varphi\in H^{-1/2}(\Gamma^{\re}_k).
	\end{align}
	Remark that the expression \eqref{scalar_product} indeed induces a scalar product in $H^{-1/2}(\Gamma^{\re})$. This follows from the continuity and coercivity of $S_0^{\re,kk}$, see \cite[Theorem 6.2.2]{steinbach}:
	\begin{align}
		\label{eq:ellipticityS0}
		C_{\re}\|\psi\|^2_{H^{-1/2}(\Gamma_k^{\re})}\geq  \langle \psi, \overline{S_0^{\re,kk}\psi}\rangle_{H^{-1/2}(\Gamma^{\re}_k), H^{1/2}(\Gamma^{\re}_k)}\geq c_{\re}\|\psi\|^2_{H^{-1/2}(\Gamma_k^{\re})}.
	\end{align}
	Let us define the orthogonal projection operator $\mathbb{Q}_{\sigma}: H^{-1/2}(\Gamma^{\re}) \to H^{-1/2}(\Gamma^{\re})$ on the $N$-dimensional space $\mathcal{V}^{\re}_G$: for all $\vec{\psi} \in H^{-1/2}(\Gamma^{\re})$, $\mathbb{Q}_{\sigma} \vec{\psi}:=\vec{\psi}_{\sigma}$, where $\vec{\psi}_{\sigma}$ is a unique element in $\mathcal{V}^{\re}_G$ such that
	\begin{align*}
		(\vec{\psi} - \vec{\psi}_{\sigma}, \vec{\varphi})_{S_0} = 0, \quad \forall \vec{\varphi} \in \mathcal{V}^{\re}_G,
	\end{align*}
	and $\mathbb{Q}_{\perp}=\mathbb{I}-\mathbb{Q}_{\sigma}$. We have the following explicit characterization of $\mathbb{Q}_{\sigma}$. 
	\begin{lemma}
		\label{lem:expl_char}
		For $\vec{\psi}\in H^{-1/2}(\Gamma^{\re})$, $
		\mathbb{Q}_{\sigma}\vec{\psi}=\sum\limits_{k=1}^N \langle \vec{\psi},\vec{1}_k^{\re}\rangle \left(c_{k}^{\re}\right)^{-1} \vec{\sigma}_k^{\re}.$\\ Here $c^{\re}_k$ is the capacity: $c^{\re}_k:=\int_{\Gamma_k^{\re}}\sigma_{k}^{\re}(\vec{x})d\Gamma_{\vec{x}}>0$.
	\end{lemma}
	\begin{proof}
		Let $\vec{\psi}\in H^{-1/2}(\Gamma^{\re})$. By definitions of $\mathbb{Q}_{\sigma}$ and $\mathcal{V}^{\re}_G$, it follows that $
		\vec{\psi}_{\sigma}=\mathbb{Q}_{\sigma}\vec{\psi}=\sum\limits_{k=1}^N \alpha_k^{\re}\vec{\sigma}_k^{\re}, \, \alpha^{\re}_k\in \mathbb{C}$.  
		To find the coefficients $\alpha_k^{\re}$, we use the definition of $\mathbb{Q}_{\sigma}$: 
		\begin{align}
			\label{eq:vpsi}
			(\vec{\psi}_{\sigma}, \vec{\sigma}_{\ell}^{\re})_{S_0}=	(\vec{\psi}, \vec{\sigma}_{\ell}^{\re})_{S_0}, \quad \ell=1, \ldots, N. 
		\end{align}
		Next, since $\sigma_{\ell,m}=0$ for $m\neq \ell$, we have that:
		\begin{align*}%
			(\vec{\psi}, \vec{\sigma}_{\ell}^{\re})_{S_0}=\langle \psi_{\ell}, S_0^{\re,\ell\ell}\sigma^{\re}_{\ell}\rangle_{-1/2,1/2,\Gamma^{\re}_{\ell}}=\langle \psi_{\ell},1\rangle_{-1/2,1/2, \Gamma_{\ell}^{\re}}.
		\end{align*}
		Using the same identity with $\vec{\psi}$ replaced by $\vec{\psi}_{\sigma}$, and inserting the result into \eqref{eq:vpsi}, we recover the coefficients $\alpha_k^{\re}$:
		\begin{align*}%
			\alpha_{\ell}^{\re}\langle \sigma_{\ell}^{\re},1\rangle_{-1/2,1/2, \Gamma_{\ell}^{\re}}=\langle \psi_{\ell},1\rangle_{-1/2,1/2, \Gamma_{\ell}^{\re}}.
		\end{align*}
		It remains to remark that	by \eqref{eq:ellipticityS0}
		\begin{align}
			\label{eq:remark_cap_positive}
			c_{\ell}^{\re}:=\langle \sigma_{\ell}^{\re},1\rangle_{-1/2,1/2, \Gamma_{\ell}^{\re}}= \langle \sigma_{\ell}^{\re},S_0^{\re,\ell\ell}\sigma_{\ell}^{\re}\rangle_{-1/2,1/2, \Gamma_{\ell}^{\re}}>0. 
		\end{align}
	\end{proof}
	From the above lemma it follows immediately that 
	\begin{align}
		\label{eq:KerQsigma}
		\begin{split}
			&\text{Ker} \, \mathbb{Q}_{\sigma} = \Im \, \mathbb{Q}_{\perp} = H^{-1/2}_{\perp}(\Gamma^{\re}), \qquad  \text{Ker} \, \mathbb{Q}_{\perp} = \Im \, \mathbb{Q}_{\sigma} = \mathcal{V}^{\re}_G.
		\end{split}
	\end{align}
	We will also need the following, immediate from  \eqref{eq:KerQsigma}, properties:
	\begin{align}
		\label{eq:qperp_pperp}
		(a)\quad	\mathbb{Q}_{\perp}\mathbb{P}_{\perp}^*=(\mathbb{I}-\mathbb{Q}_{\sigma})\mathbb{P}_{\perp}^*=\mathbb{P}_{\perp}^*,\qquad 
		(b)\quad \mathbb{Q}_{\sigma}\mathbb{P}_{0}^*=\mathbb{Q}_{\sigma}.
	\end{align}
	Remark that the above projection operators have been defined on the dual space $H^{-\frac{1}{2}}(\Gamma^{\re})$. We will also need to define the adjoint projectors on the space $H^{1/2}(\Gamma^{\re})$: for all  $\vec{\varphi}\in H^{1/2}(\Gamma^{\re}), \, \vec{\psi}\in H^{-1/2}(\Gamma^{\re})$, 
	\begin{align*}%
		&\langle \vec{\psi}, \overline{\mathbb{Q}^{*}_{\sigma} \vec{\varphi}} \rangle_{-1/2, \, 1/2} := \langle \mathbb{Q}_{\sigma} \vec{\psi},  \overline{\vec{\varphi}} \rangle_{-1/2, \, 1/2}, \\
		&\langle \vec{\psi}, \overline{\mathbb{Q}^{*}_{\perp} \vec{\varphi}} \rangle_{-1/2, \, 1/2} := \langle \mathbb{Q}_{\perp} \vec{\psi},  \overline{\vec{\varphi}} \rangle_{-1/2, \, 1/2}.
	\end{align*}
	By \cite[p. 23, Lemma 2.10]{mclean}, from the above definition and \eqref{eq:KerQsigma} it follows that 
	\begin{align}
		\label{eq:kerqperp}
		\begin{split}
			&\Im \, \mathbb{Q}^*_{\sigma} = \text{Ker} \, \mathbb{Q}^*_{\perp} = \mathbb{S}_0, \qquad \Im \, \mathbb{Q}^*_{\perp} = \text{Ker} \, \mathbb{Q}^*_{\sigma} =  ( \mathcal{V}^{\re}_G )^a,
		\end{split}
	\end{align}
	where $( \mathcal{V}^{\re}_G )^a=\{\vec{\varphi}\in H^{1/2}(\Gamma^{\re}): \langle \vec{\psi}, \vec{\varphi}\rangle=0, \text{ for all }\vec{\psi}\in \mathcal{V}^{\re}_G\}$. The above implies that 
	\begin{align}
		\label{eq:qperp_adj}
		\mathbb{Q}_{\perp}^*=\mathbb{Q}_{\perp}^*\mathbb{P}_{\perp}.
	\end{align}
	We also have an explicit characterization of the adjoint operators, see \eqref{eq:remark_cap_positive}.  
	\begin{lemma}
		\label{lem:adjoint}
		For $\vec{\varphi} \in H^{1/2}(\Gamma^{\re})$, 
		$
		\mathbb{Q}_{\sigma}^*\vec{\varphi}=\sum\limits_{k=1}^N \langle \vec{\sigma}^{\re}_k, \vec{\varphi}\rangle (c_k^{\re})^{-1}\vec{1}^{\re}_k.
		$
	\end{lemma}
	\begin{proof}
		Let $\vec{\varphi}_{\sigma}:=	\mathbb{Q}_{\sigma}^*\vec{\varphi}$, $\vec{\psi}\in H^{-1/2}(\Gamma^{\re})$. By definition of $\mathbb{Q}_{\sigma}^*$ and by Lemma \ref{lem:expl_char}:
		\begin{align*}%
			\langle \vec{\psi},\overline{\vec{\varphi}_{\sigma}}\rangle =\langle \mathbb{Q}_{\sigma} \vec{\psi},\overline{\vec{\varphi}}\rangle = \langle \sum_{k=1}^N \langle \vec{\psi}, \vec{1}^{\re}_k \rangle (c^{\re}_k)^{-1} \vec{\sigma}^{\re}_k, \overline{\vec{\varphi}} \rangle = \langle \vec{\psi}, \overline{\sum_{k=1}^{N} \langle \vec{\sigma}^{\re}_k, \vec{\varphi}  \rangle (c^{\re}_k)^{-1} \vec{1}^{\re}_k} \rangle.
		\end{align*}	
	\end{proof}
	With the above lemma we see that for real-valued $\vec{\varphi}$, $\mathbb{Q}_{\sigma}^*\vec{\varphi}$ is real-valued, and therefore, for all $\vec{\varphi}\in H^{1/2}(\Gamma^{\re})$, $\vec{\psi}\in H^{-1/2}(\Gamma^{\re})$,
	\begin{align*}%
		&\langle \vec{\psi}, \mathbb{Q}^{*}_{\sigma} \vec{\varphi} \rangle_{-1/2, \, 1/2} = \langle \mathbb{Q}_{\sigma} \vec{\psi},  \vec{\varphi} \rangle_{-1/2, \, 1/2}, \qquad 
		\langle \vec{\psi}, \mathbb{Q}^{*}_{\perp} \vec{\varphi} \rangle_{-1/2, \, 1/2} = \langle \mathbb{Q}_{\perp} \vec{\psi},  \vec{\varphi} \rangle_{-1/2, \, 1/2}.	
	\end{align*}
	We will use these observations without referring to them explicitly. 
	\subsubsection{Properties and relations between projectors}
	So far we have looked only at the algebraic properties of the subspaces and projectors. In this section we will establish relations between them and consider some topological (i.e. norm-related) properties. We start with an important fact about the space $H^{1/2}_{\perp}(\Gamma^{\re})$ that follows from Corollary \ref{cor:equiv_norm}.
	\begin{proposition}
		\label{prop:norm_equiv_perp}
		There exists $C>0$, s.t. for all $\re\in (0,1]$,  $\|\mathbb{P}_{\perp}\vec{\varphi}\|_{H^{1/2}(\Gamma^{\re})}\leq C |\vec{\varphi}|_{H^{1/2}(\Gamma^{\re})}$, for all $\vec{\varphi}\in H^{1/2}(\Gamma^{\re})$. 
	\end{proposition}
	The rest of the section will be dedicated to the interplay of the above property and the projectors introduced in Section \ref{sec:decompVg}. 
	In what follows, it will be convenient to remark that from an elementary scaling argument, for $\bx \in \Gamma^{\re}_k$, it follows that 
	\begin{align}
		\label{eq:sigmakeps}
		\sigma_k^{\re}(\vec{x})=\re^{-1}\sigma_k^1\left(\frac{\vec{x}}{\re}\right), \quad \text{thus},\quad \|\sigma_k^{\re}\|_{L^2(\Gamma_k^{\re})}^2=\|\sigma_k^1\|_{L^2(\Gamma_k)}\quad \text{and} \quad c^{\re}_k=\re c_k^1.
	\end{align}
	It appears that the norms $H^{-1/2}(\Gamma^{\re})$ and $L^2(\Gamma^{\re})$ are equivalent on the space $\mathcal{V}^{\re}_G$ with equivalence constants independent of $\re$.  
	\begin{proposition}
		\label{prop:L^2_H_minus_half_bound}
		There exists $C>0$, s.t. for all $\re \in (0, 1]$, it holds that
		\begin{align*}%
			\|\vec{\psi} \|_{H^{-1/2}(\Gamma^{\re})}\leq 	\|\vec{\psi}  \|_{L^2(\Gamma^{\re})} \leq C \|\vec{\psi} \|_{H^{-1/2}(\Gamma^{\re})}, \quad \forall \vec{\psi}\in \mathcal{V}_G^{\re}.
		\end{align*}
	\end{proposition}
	\begin{proof}
		See Proposition \ref{prop:L^2_H_minus_half_mbound_proof}.	
	\end{proof}
	\paragraph{Relations between projectors and bounds on their norms}
	The property below shows that the norms of $\mathbb{Q}_{\sigma}$, $\mathbb{Q}_{\perp}$ are uniformly bounded in $\re>0$. 
	\begin{proposition}
		\label{prop:norm_projectors_Q}
		There exists $C>0$, s.t. for all $0<\re\leq 1$, 
		\begin{align*}%
			\|\mathbb{Q}_{\sigma}\|_{\mathcal{L}(H^{-1/2}(\Gamma^{\re}),H^{-1/2}(\Gamma^{\re}))}+	\|\mathbb{Q}_{\perp}\|_{\mathcal{L}(H^{-1/2}(\Gamma^{\re}),H^{-1/2}(\Gamma^{\re}))}\leq C.
		\end{align*}
	\end{proposition}
	\begin{proof}
		Let $\boldsymbol{\psi} \in H^{-1/2}(\Gamma^{\re})$. By Lemma \ref{lem:expl_char}, we have
		\begin{align*}%
			\| \mathbb{Q}_{\sigma} \boldsymbol{\psi} \|^2_{-1/2,\Gamma^{\re}} &\leq \sum_{k=1}^N \frac{|\langle \boldsymbol{\psi}, \boldsymbol{1}^{\re}_k \rangle|^2}{(c_k^{\re})^2}  \| {\sigma}^{\re}_k \|^2_{L^2(\Gamma^{\re}_k)} \leq \sum_{k=1}^N C_k^{\re}\| {\psi}_k \|^2_{H^{-\frac{1}{2}}(\Gamma^{\re}_k)},
		\end{align*}
		where $C_k^{\re}:=\| {1}^{\re}_{k,k}\|^2_{L^2(\Gamma^{\re}_k)}\| {\sigma}^{\re}_k \|^2_{L^2(\Gamma^{\re}_k)}(c_k^{\re})^{-2}$. 
		Using $\| {1}^{\re}_{k,k}\|^2_{L^2(\Gamma^{\re}_k)}\lesssim \re^2$ and \eqref{eq:sigmakeps},  we conclude that $C_k^{\re}\lesssim 1$, and this yields the stated bound for $\mathbb{Q}_{\sigma}$. The uniform bound on $\mathbb{Q}_{\perp}$ follows from  $\mathbb{Q}_{\perp}=\mathbb{I}-\mathbb{Q}_{\sigma}$, the triangle inequality and the bound on $\mathbb{Q}_{\sigma}$. 
	\end{proof}
	The following result allows us to extract information about $\mathbb{P}_0^*\vec{\varphi}, \, \mathbb{P}_{\perp}^*\vec{\varphi}$ from $\mathbb{Q}_{\sigma}\vec{\varphi}$, $\mathbb{Q}_{\perp}\vec{\varphi}$. Its proof can be found in Corollary \ref{proof:bound_control}. 
	\begin{proposition}
		\label{prop:bound_control}
		There exists $C>0$, s.t. for all $0<\re\leq 1$, $\vec{\varphi}\in H^{-1/2}(\Gamma^{\re})$, 
		\begin{align*}%
			&\|\mathbb{Q}_{\sigma}\boldsymbol{\varphi}\|_{H^{-1/2}(\Gamma^{\re})} \leq C \|\mathbb{P}_0^*\boldsymbol{\varphi}\|_{L^2(\Gamma^{\re})} , \\
			&\|\mathbb{Q}_{\perp}\boldsymbol{\varphi}\|_{H^{-1/2}(\Gamma^{\re})} \leq C(\re^{1/2}\|\mathbb{P}_0^*\boldsymbol{\varphi}\|_{L^2(\Gamma^{\re})}+
			\|\mathbb{P}_{\perp}^*\boldsymbol{\varphi}\|_{H^{-1/2}(\Gamma^{\re})}).
		\end{align*}
	\end{proposition}
	The next result enables us to work with the projector $\mathbb{P}_{\perp}$ instead of $\mathbb{Q}_{\perp}^*$. 
	\begin{proposition}
		\label{prop:norm_dual_Q_perp}
		There exist $C,C'>0$, s.t., for all $0<\re\leq 1$, 
		\begin{align*}%
			\|\mathbb{Q}_{\perp}^*\vec{\varphi}\|_{H^{1/2}(\Gamma^{\re})}\leq C 	\|\mathbb{P}_{\perp}\vec{\varphi}\|_{H^{1/2}(\Gamma^{\re})}\leq C' |\vec{\varphi}|_{H^{1/2}(\Gamma^{\re})}, \qquad \forall \vec{\varphi}\in H^{1/2}(\Gamma^{\re}).
		\end{align*}
	\end{proposition}
	\begin{proof}
		The first bound follows from \eqref{eq:qperp_adj}, namely  $\mathbb{Q}_{\perp}^*=\mathbb{Q}_{\perp}^*\mathbb{P}_{\perp}$, the identity $\|\mathbb{Q}_{\perp}^*\|_{\mathcal{L}(H^{1/2}(\Gamma^{\re}),H^{1/2}(\Gamma^{\re}))}=\|\mathbb{Q}_{\perp}\|_{\mathcal{L}(H^{-1/2}(\Gamma^{\re}),H^{-1/2}(\Gamma^{\re}))}$ and Proposition \ref{prop:norm_projectors_Q}. The second bound is immediate from Proposition \ref{prop:norm_equiv_perp}.
	\end{proof}
	Similarly, we will also need a simplified stability bound for $\mathbb{Q}_{\sigma}^*$. 
	\begin{proposition}
		\label{prop:qsigmaP0}
		There exists $C>0$, s.t. for all $\re \in (0,1)$, all $\vec{\varphi}\in H^{1/2}(\Gamma^{\re})$, 
		\begin{align*}%
			\|\mathbb{Q}_{\sigma}^*\vec{\varphi}\|_{H^{1/2}(\Gamma^{\re})} = \|\mathbb{Q}_{\sigma}^*\vec{\varphi}\|_{L^2(\Gamma^{\re})} \leq  C \|\vec{\varphi}\|_{L^2(\Gamma^{\re})}.
		\end{align*}
	\end{proposition}
	\begin{proof}
		The first identity follows from $\Im \mathbb{Q}_{\sigma}^*=\mathbb{S}_0$, cf. \eqref{eq:kerqperp}. The stated inequality follows by using Lemma \ref{lem:adjoint} and the Cauchy-Schwarz inequality:
		\begin{align*}%
			\|\mathbb{Q}_{\sigma}^*\vec{\varphi}\|_{L^2(\Gamma^{\re})}^2\leq \sum\limits_{k=1}^N C_k^{\re}\|{\varphi}_k\|_{L^2(\Gamma_k^{\re})}^2,\, C_k^{\re}:=\|{\sigma}_k^{\re}\|^2_{L^2(\Gamma_k^{\re})}(c^{\re}_k)^{-2} \|1_{k,k}^{\re}\|_{L^2(\Gamma^{\re}_k)}^2,
		\end{align*}
		and we conclude like in the proof of Proposition \ref{prop:norm_projectors_Q}. 
	\end{proof}

	\paragraph{A remark on notation} Sometimes we will need to apply $\mathbb{P}_{0}$, $\mathbb{P}_{\perp}$ etc. to forms or functions defined on a single boundary $\Gamma_k^{\re}$. The distinction between the projectors on spaces on  $\Gamma_k^{\re}$ and on $\Gamma^{\re}$ is not reflected in the notation.
	\subsection{Density error}
	\label{sec:analysis_freq_domain}
	The goal of this section is the error analysis for the scattered field, namely the proof of Theorem \ref{th:conv}. The well-posedness result of Theorem \ref{th:wp} will be a by-product of the construction. 
	The starting point of the proof of Theorem \ref{th:conv} is obtaining an appropriate bound on the error of the density, which will be further translated into the error of the scattered field, by applying the single-layer potential to this error, cf. \eqref{eq:fdsys} and \eqref{eq:SL_potential}. 
	
	We study the density error $\hat{\boldsymbol{e}}^{\re}$, namely the difference between the exact density $\hat{\vec{\lambda}}^{\re}$ satisfying  \eqref{eq:SL_BIE_FD} and approximate density $\hat{\vec{\lambda}}^{\re}_G$, see \eqref{eq:approximate_density_frequency_domain}, i.e. 
	\begin{align*}%
		\hat{\vec{e}}^{\re} := \hat{\vec{\lambda}}^{\re} - \hat{\vec{\lambda}}^{\re}_G.
	\end{align*}
	Our first goal is to obtain a bound on the density error $\hat{\vec{e}}^{\re}$ that is explicit in terms of $\re$, $\omega$, $N$ and $d_*^{\re}$. From now on we fix the frequency $\omega\in \mathbb{C}^+$. 
	
	We remark that $\hat{\vec{e}}^{\re}$ satisfies the following SL BIE for the Helmholtz problem
	\begin{align}
		\label{eq:BIE_density_error}
		S^{\re}_{\omega} \hat{\vec{e}}^{\re} = \hat{\vec{g}}^{\re}  - \hat{\vec{g}}^{\re}_G, \quad \text{where} \quad \hat{\vec{g}}^{\re}_G := S^{\re}_{\omega} \hat{\vec{\lambda}}^{\re}_G.
	\end{align}
	At this point, one would be tempted to write 
	\begin{align}
		\label{eq:ere}
		\hat{\vec{e}}^{\re}=\left(S_{\omega}^{\re}\right)^{-1}(\hat{\vec{g}}^{\re}-\hat{\vec{g}}_G^{\re}),
	\end{align}
	and next use the usual operator-norm bound
	$%
	\|\hat{\vec{e}}^{\re}\| \leq \|(S_{\omega}^{\re})^{-1}\|\|\hat{\vec{g}}^{\re}-\hat{\vec{g}}_G^{\re}\|.
	$
	It can be shown, cf. \cite{MK}, that this strategy would lead to non-optimal error estimates. Instead, we proceed as follows. We split the error into two components, using the projectors  $\mathbb{Q}_{\sigma}$ and $\mathbb{Q}_{\perp}$, defined in Section \ref{sec:orthogonal_projectors}:
	\begin{align}
		\label{eq:splitting_density_errorr}
		\hat{\vec{e}}^{\re} =  \hat{\vec{e}}^{\re}_{\sigma} + \hat{\vec{e}}^{\re}_{\perp}, \quad \hat{\vec{e}}^{\re}_{\sigma}:=\mathbb{Q}_{\sigma} \hat{\vec{e}}^{\re} \in \mathcal{V}^{\re}_G, \quad  \hat{\vec{e}}^{\re}_{\perp}:=  \mathbb{Q}_{\perp} \hat{\vec{e}}^{\re}  \in H^{-1/2}_{\perp}(\Gamma^{\re}),
	\end{align}
	and estimate these two components separately. We will show that they have a different scaling with respect to $\re$, which has an incidence on the upper bound of the error between the approximate and the exact scattered fields. The choice of the above decomposition for the error is discussed in more detail in Section \ref{sec:continuity_bound}, see Remark \ref{remark:choice_of_decomposition}.
	
	\textit{Decomposition of \eqref{eq:BIE_density_error}. }
	Using the properties of the orthogonal projectors $\mathbb{Q}_{\sigma}$, $\mathbb{Q}_{\perp}$ and their adjoints $\mathbb{Q}^*_{\sigma}$, $\mathbb{Q}^*_{\perp}$, see \eqref{eq:kerqperp}, we introduce the following notations for the restrictions of $S^{\re}_{\omega}$ to the spaces $\mathcal{V}^{\re}_G$ and $H^{-1/2}_{\perp}(\Gamma^{\re})$:
	\begin{align}
		\label{eq:restrictions_single_layer}
		&\mathbb{S}^{\re}_{\sigma \sigma} := \mathbb{Q}^*_{\sigma} S^{\re}_{\omega}|_{\mathcal{V}^{\re}_G} : \mathcal{V}^{\re}_G \rightarrow  \mathbb{S}_0,  &&\mathbb{S}^{\re}_{\sigma \perp} := \mathbb{Q}^*_{\sigma} S^{\re}_{\omega}|_{H^{-1/2}_{\perp}(\Gamma^{\re})} : H^{-1/2}_{\perp}(\Gamma^{\re}) \rightarrow \mathbb{S}_0, \nonumber \\                                                                                                         
		&\mathbb{S}^{\re}_{\perp \sigma} := \mathbb{Q}^*_{\perp} S^{\re}_{\omega}|_{\mathcal{V}^{\re}_G} : \mathcal{V}^{\re}_G \rightarrow  (\mathcal{V}^{\re}_G)^a, 
		&&\mathbb{S}^{\re}_{\perp \perp} := \mathbb{Q}^*_{\perp} S^{\re}_{\omega}|_{H^{-1/2}_{\perp}(\Gamma^{\re})}  : H^{-1/2}_{\perp}(\Gamma^{\re}) \rightarrow  (\mathcal{V}^{\re}_G)^a.
	\end{align}
	We equip the corresponding domain spaces with the $H^{-1/2}(\Gamma^{\re})$-norm, while the range spaces with the $H^{1/2}(\Gamma^{\re})$-norm. 
	\begin{remark}
		Using \eqref{eq:restrictions_single_layer}, in particular, the asymptotic Galerkin model rewrites:
		\begin{align}
			\label{eq:sresigma}
			\mathbb{S}^{\re}_{\sigma\sigma}\hat{\blambda}^{\re}_G=\mathbb{Q}^{*}_{\sigma}\hat{\bg}^{\re}.	
		\end{align}
		To see this, we recall the definition of the asymptotic Galerkin method: 
		\begin{align*}%
			\text{ find } \hat{\blambda}^{\re}_G\in \mathcal{V}_{G}^{\re}, \text{ s.t.	}\langle \vec{\varphi}, \overline{\mathbb{S}^{\re}_{\omega}\hat{\blambda}^{\re}_G}\rangle_{-1/2, 1/2}=\langle \vec{\varphi}, \overline{\hat{\vec{g}}^{\re}}\rangle_{-1/2, 1/2}, \quad \forall\vec{\varphi} \in \mathcal{V}^{\re}_G.
		\end{align*}
		In particular, since  $\mathcal{V}^{\re}_G=\left\{\mathbb{Q}_{\sigma}\vec{\vec{\psi}},\, \vec{\psi}\in H^{-1/2}(\Gamma^{\re})\right\}$, taking in the above  $\vec{\varphi}=\mathbb{Q}_{\sigma}\vec{\psi}$, $\vec{\psi}\in H^{-1/2}(\Gamma^{\re})$ yields
		\begin{align*}%
			\langle \vec{\psi}, \overline{\mathbb{Q}_{\sigma}^*\mathbb{S}^{\re}_{\omega}\hat{\blambda}^{\re}_G}\rangle_{H^{-1/2}(\Gamma^{\re}), H^{1/2}(\Gamma^{\re})}=\langle \vec{\psi}, \overline{\mathbb{Q}_{\sigma}^*\hat{\vec{g}}^{\re}}\rangle_{H^{-1/2}(\Gamma^{\re}), H^{1/2}(\Gamma^{\re})}, \quad \vec{\forall }\vec{\psi}\in H^{-1/2}(\Gamma^{\re}),
		\end{align*}
		and as $\hat{\blambda}^{\re}_G\in \mathcal{V}_G^{\re}$,  \eqref{eq:sresigma} follows.
	\end{remark}	
	Applying the splitting \eqref{eq:splitting_density_errorr} and the notation \eqref{eq:restrictions_single_layer} to \eqref{eq:BIE_density_error} yields
	\begin{align}
		&\mathbb{S}^{\re}_{\sigma\sigma} \hat{\vec{e}}^{\re}_{\sigma} + \mathbb{S}^{\re}_{\sigma \perp} \hat{\vec{e}}^{\re}_{\perp} =0, \label{first_equation} \\
		&\mathbb{S}^{\re}_{\perp \sigma} \hat{\vec{e}}^{\re}_{\sigma} + \mathbb{S}^{\re}_{\perp \perp} \hat{\vec{e}}^{\re}_{\perp} = \mathbb{Q}^{*}_{\perp} (\hat{\vec{g}}^{\re} - \hat{\vec{g}}^{\re}_{G}), \label{second_equation}
	\end{align}
	where in the first equation we used that  $\mathbb{Q}^*_{\sigma}(\hat{\vec{g}}^{\re} - \hat{\vec{g}}^{\re}_G) = 0$. Indeed, $$\mathbb{Q}^*_{\sigma}(\hat{\vec{g}}^{\re} - \hat{\vec{g}}^{\re}_G)=\mathbb{Q}^*_{\sigma}(\hat{\vec{g}}^{\re} - S^{\re}_{\omega}\hat{\vec{\lambda}}_{G}^{\re})=\mathbb{Q}^*_{\sigma}\hat{\vec{g}}^{\re} - S^{\re}_{\sigma\sigma}\hat{\vec{\lambda}}_{G}^{\re}=0,$$
	as follows from \eqref{eq:sresigma}. 
	\begin{remark}
		The projectors $\mathbb{Q}_{\sigma}$ and $\mathbb{Q}_{\perp}$ are chosen so that the matrix in \eqref{first_equation}, \eqref{second_equation} becomes close to diagonal, see the estimates of Theorem \ref{theorem:estimates_norms}.	
	\end{remark}

	\textit{An expression for $\hat{\vec{e}}^{\re}_{\sigma}$. }From  \eqref{first_equation} we have that
	\begin{align*}
		\hat{\vec{e}}^{\re}_{\sigma} = - (\mathbb{S}^{\re}_{\sigma \sigma})^{-1} \mathbb{S}^{\re}_{ \sigma \perp } \hat{\vec{e}}^{\re}_{\perp}.
	\end{align*}
	Thus, obtaining an upper bound on $\hat{\vec{e}}_{\sigma}^{\re}$ amounts to bounding the operators involved in the above expression, as well as finding an upper bound for $\hat{\vec{e}}^{\re}_{\perp}$. To obtain this latter bound, it would be natural to use \eqref{second_equation}; however, a more efficient way is to follow the path of \cite{MK}, which allows to have better final bounds in terms of $\omega$. \\
	\textit{An expression of $\hat{\vec{e}}^{\re}_{\perp}$. }
	To obtain a close to optimal bound for the component orthogonal to the constants $\hat{\vec{e}}^{\re}_{\perp}$, we recall the original identity \eqref{eq:ere},
	\begin{align*}%
		\hat{\vec{e}}^{\re} = ({S}^{\re}_{\omega})^{-1} (\hat{\vec{g}}^{\re} - \hat{\vec{g}}^{\re}_G ),
	\end{align*}
	and next apply the orthogonal projector $\mathbb{Q}_{\perp}$ to the both sides of the above identity:
	\begin{align*}%
		\hat{\vec{e}}^{\re}_{\perp}\equiv\mathbb{Q}_{\perp} \hat{\vec{e}}^{\re} = \mathbb{Q}_{\perp} (S^{\re}_{\omega})^{-1} (\hat{\vec{g}}^{\re} - \hat{\vec{g}}^{\re}_G ) &= \mathbb{Q}_{\perp} (S^{\re}_{\omega})^{-1} ( \mathbb{Q}^*_{\sigma} (\hat{\vec{g}}^{\re} - \hat{\vec{g}}^{\re}_G) + \mathbb{Q}^*_{\perp} (\hat{\vec{g}}^{\re} - \hat{\vec{g}}^{\re}_G))  \\
		&= \mathbb{Q}_{\perp} (S^{\re}_{\omega})^{-1} \mathbb{Q}^*_{\perp} (\hat{\vec{g}}^{\re} - \hat{\vec{g}}^{\re}_G),
	\end{align*}
	where we used again the Galerkin orthogonality to argue that $\mathbb{Q}^*_{\sigma} (\hat{\vec{g}}^{\re} - \hat{\vec{g}}^{\re}_G) = 0$. \\
	\textit{A strategy. }	
	We thus end up with the two identities:
	\begin{align*}%
		&\hat{\vec{e}}^{\re}_{\sigma} = - (\mathbb{S}^{\re}_{\sigma \sigma})^{-1} \mathbb{S}^{\re}_{ \sigma \perp } \hat{\vec{e}}^{\re}_{\perp}, \qquad
		\hat{\vec{e}}^{\re}_{\perp} = \mathbb{Q}_{\perp} (S^{\re}_{\omega})^{-1} \mathbb{Q}^*_{\perp} (\hat{\vec{g}}^{\re} - \hat{\vec{g}}^{\re}_G),
	\end{align*}
	and therefore
	\begin{align}
		\label{eq:error_bounds}
		\begin{split}
			&\| \hat{\vec{e}}^{\re}_{\sigma} \|_{H^{-1/2}(\Gamma^{\re})} \leq \| (\mathbb{S}^{\re}_{\sigma \sigma})^{-1} \| \| \mathbb{S}^{\re}_{ \sigma \perp } \| \| \hat{\vec{e}}^{\re}_{\perp} \|_{H^{-1/2}(\Gamma^{\re})}, \\
			&\| \hat{\vec{e}}^{\re}_{\perp} \|_{H^{-1/2}(\Gamma^{\re})} \leq \| \mathbb{Q}_{\perp} (S^{\re}_{\omega})^{-1} \mathbb{Q}^*_{\perp} \| \| \mathbb{Q}^*_{\perp} (\hat{\vec{g}}^{\re} - \hat{\vec{g}}^{\re}_G) \|_{H^{1/2}(\Gamma^{\re})}. 
		\end{split}
	\end{align}
	
	Therefore, we see that estimating the error components in terms of $\re$ amounts to obtaining estimates on the operator norms and the data $\|\mathbb{Q}_{\perp}^*(\hat{\vec{g}}^{\re}-\hat{\vec{g}}_G^{\re})\|_{H^{1/2}(\Gamma^{\re})}$. 	
	
	\textit{Bounds on operator norms and data. }
	The following theorem summarizes the upper bounds on the norms of the operators that appear in \eqref{eq:error_bounds}. 
	\begin{theorem}[Estimates on operator norms]
		\label{theorem:estimates_norms}
		There exists $C>0$, which depends only on $\Gamma_j$, $j=1,\ldots,N$, s.t. for all $\re \in (0,1)$ and $\omega \in \mathbb{C}^+$,
		\begin{align}
			\|(\mathbb{S}^{\re}_{\sigma \sigma})^{-1}\| &\leq \re^{-1}\times C (\underline{d}^{\re}_*)^{-2} (1+|\omega|)^2 \max(1, (\operatorname{Im}\omega)^{-3}), \label{eq:coercivity_bound_1} \\
			\|\mathbb{Q}_{\perp} (S^{\re}_{\omega})^{-1} \mathbb{Q}^*_{\perp} \| &\leq C(\underline{d}^{\re}_*)^{-2} (1+ |\omega|)^2 \max(1, (\operatorname{Im}\omega)^{-3} ) , \label{eq:Q_perp_S_inv_Q_perp}\\
			\label{eq:continuity_bound}
			\|\mathbb{S}^{\re}_{\perp\sigma}\|,\; \| \mathbb{S}^{\re}_{\sigma \perp} \| &\leq \re^{5/2}\times C N(\underline{d}^{\re}_*)^{-2} (1+|\omega|)^2, 
		\end{align}
		where $
		\underline{d}^{\re}_* = \min(1, d^{\re}_*).$
	\end{theorem}
	\begin{proof}
		The stability bounds \eqref{eq:coercivity_bound_1} and \eqref{eq:Q_perp_S_inv_Q_perp} are proven in Section \ref{sec:coercivity_bound}, while the continuity bound \eqref{eq:continuity_bound} is proven in Section \ref{sec:continuity_bound}.
	\end{proof}
	Let us now present an auxiliary bound on the data, which will be used in the sequel. We first derive an intermediate result in the frequency domain, which will then be improved after obtaining the corresponding bounds in the time domain.
	\begin{theorem}[Bound on the data]
		\label{theorem:bounds_data} 
		There exists $C>0$, which depends only on $\Gamma_j$, $j=1,\ldots,N$, s.t. for all $\re \in (0,1)$ and $\omega \in \mathbb{C}^+$, the following bound holds true:
		\begin{align*}%
			\|\mathbb{Q}^{*}_{\perp}( \hat{\vec{g}}^{\re}-\hat{\vec{g}}^{\re}_G)\|_{H^{1/2}(\Gamma^{\re})}  
			\leq C_{\omega, N, d^{\re}_*} (\re^{3/2} \| \hat{\vec{g}}^{\re} \|_{L^2(\Gamma^{\re})} + |\hat{\vec{g}}^{\re}|_{H^{1/2}(\Gamma^{\re})}),
		\end{align*}
		with  $	C_{\omega, N, d^{\re}_*} := C N (\underline{d}^{\re}_{*})^{-4} (1+|\omega|)^4 \max(1, (\operatorname{Im} \omega)^{-3})$.
	\end{theorem}
	\begin{proof}
		See Section \ref{sec:data_bound}.
	\end{proof}
	Inserting the bounds of Theorems \ref{theorem:estimates_norms} and \ref{theorem:bounds_data} into \eqref{eq:error_bounds} yields the following bounds:
	%
	\begin{equation}
		\label{eq:bound_on_error_components_v0}
		\begin{aligned}
			&\| \hat{\vec{e}}^{\re}_{\sigma} \|_{H^{-1/2}(\Gamma^{\re})} \leq \re^{3/2} C N^2 (\underline{d}^{\re}_{*})^{-10} (1+|\omega|)^{10} \max(1, (\Im \omega)^{-9}) \trinorm{\hat{\vec{g}}^{\re}}_{3/2},\\
			&\| \hat{\vec{e}}^{\re}_{\perp} \|_{H^{-1/2}(\Gamma^{\re})} \leq C N (\underline{d}^{\re}_{*})^{-6} (1+|\omega|)^6 \max(1, (\Im \omega)^{-6})\trinorm{\hat{\vec{g}}^{\re}}_{3/2},
		\end{aligned}
	\end{equation}
	where we used a temporary notation $\trinorm{\hat{\vec{g}}^{\re}}_{\alpha}:=\re^{\alpha} \|  \hat{\vec{g}}^{\re} \|_{L^2(\Gamma^{\re})} +  |  \hat{\vec{g}}^{\re} |_{H^{1/2}(\Gamma^{\re})}$. The following result makes the bounds \eqref{eq:bound_on_error_components_v0} fully explicit in  $\re>0$.
	\begin{proposition}
		\label{prop:data}
		Let $\hat{u}^{\inc}\in H^3(B)$. Then there exists $C>0$, s.t. for all $0<\re\leq 1$,  $\trinorm{\hat{\vec{g}}^{\re}}_{3/2}\leq C \re^{3/2}\|\hat{u}^{\inc}\|_{H^3(B)}$. 
	\end{proposition}
	\begin{proof}
		See Section \ref{sec:data_bound2}.
	\end{proof}
	Thus, our goal in this section is to prove the following bounds:
	\begin{equation}
		\label{eq:bound_on_error_components_v1}
		\begin{aligned}
			&\| \hat{\vec{e}}^{\re}_{\sigma} \|_{H^{-1/2}(\Gamma^{\re})} \leq \re^{3}\times C N^2 (\underline{d}^{\re}_{*})^{-10} (1+|\omega|)^{10} \max(1, (\Im \omega)^{-9}) \|\hat{u}^{\inc}\|_{H^3(B)},\\
			&\| \hat{\vec{e}}^{\re}_{\perp} \|_{H^{-1/2}(\Gamma^{\re})} \leq \re^{3/2}\times C N (\underline{d}^{\re}_{*})^{-6} (1+|\omega|)^6 \max(1, (\Im \omega)^{-6})\|\hat{u}^{\inc}\|_{H^3(B)}.
		\end{aligned}
	\end{equation}
	We see that $\hat{\vec{e}}^{\re}_{\sigma}$ and $\hat{\vec{e}}_{\perp}^{\re}$ have a different scaling with respect to $\re$ as $\re\rightarrow 0$. It is this property that will allow to obtain an optimal error bound for the scattered field, see the discussion before \eqref{eq:ere}. 
	\subsection{Proofs of the technical results in Section \ref{sec:analysis_freq_domain}}	
	\subsubsection{Proof of the stability bound on the operator norms $\| (\mathbb{S}^{\re}_{\sigma \sigma})^{-1} \|$, $\| \mathbb{Q}_{\perp} (S^{\re}_{\omega})^{-1} \mathbb{Q}^*_{\perp} \|$}
	\label{sec:coercivity_bound}
	In this section, we prove stability bounds \eqref{eq:coercivity_bound_1} and \eqref{eq:Q_perp_S_inv_Q_perp}. Our strategy follows the strategy of similar estimates for circles in \cite{MK}, and relies on deriving a coercivity bound for the single layer boundary integral operator $S^{\re}_{\omega}$ with an explicit coercivity constant. To do so, we follow a classical approach based on the lifting lemma adapted to the many-particle case (an idea due to \cite{hassan_stamm}), which yields bounds explicit in $\re$. Let us define the following energy norm, for $\vec{v}\in H^1(\mathbb{R}^3)$:
	\begin{align}%
		\label{eq:energy_norm}
		\|{v}\|_{a, \mathbb{R}^3}^2:=a^2\|v\|^2_{L^2(\mathbb{R}^3)}+\|\nabla v\|_{L^2(\mathbb{R}^3)}^2.
	\end{align}
	For $V\in \{\mathbb{S}_0, \, H^{1/2}_{\perp}(\Gamma^{\re})\}$ we define the quantity (well-defined by the same argument as in proof of Proposition 2.5.1 in \cite{sayas})
	\begin{align}
		\label{eq:lva}
		L_V^{\re}(a):=\sup\limits_{\vec{\lambda}\in V\setminus\{0\}}\inf_{\substack{\Lambda^{\re} \in H^1(\mathbb{R}^3): \\  \gamma_0\Lambda^{\re}=\vec{\lambda}}}  \frac{\| \Lambda^{\re} \|_{a,\mathbb{R}^3}}{\| \vec{\lambda} \|_{H^{1/2}(\Gamma^{\re})}}\geq 0,
	\end{align}
	The proof of the result below resembles the proof of Proposition 4.4 in \cite{MK}. See Appendix \ref{app:lifting_lemma} for details.
	\begin{proposition}
		\label{prop:lifting_lemma}
		There exists $C>0$, s.t. for all $a>0$,  $0<\re\leq 1$,
		\begin{align*}%
			L_{\mathbb{S}_0}^{\re}(a)\leq C_a \re^{-1/2} ,\quad 
			L_{H^{1/2}_{\perp}(\Gamma^{\re})}^{\re}(a)\leq C_a,
		\end{align*}
		where $C_a=C(\underline{d}^{\re}_*)^{-1} (1+a)^{1/2} \max(1, a^{-1})$.
	\end{proposition}
	This result enables us to prove the proposition below.
	
	\begin{proposition}
		\label{prop:coercivity_bound}
		For all $0<\re\leq 1$, $\omega\in \mathbb{C}^+$, the operator $S^{\re}_{\omega}:$ $ H^{-1/2}(\Gamma^{\re}) \rightarrow$ $H^{1/2}(\Gamma^{\re})$
		is continuous and invertible. Moreover, for all $ \vec{\eta} \in H^{-1/2}(\Gamma^{\re})$,
		\begin{align*}
			-\Im\langle {\vec{\vec{\eta}}}, \overline{ \omega S^{\re}_{\omega} \vec{\vec{\eta}}}\rangle_{-1/2,1/2,\Gamma^{\re}} 
			&\gtrsim (\underline{d}^{\re}_*)^{2} (1+|\omega|)^{-1} \min(1, (\Im \omega)^3)\\
			&\times \left(\re \| \mathbb{P}^*_0 \vec{\eta} \|^2_{H^{-1/2}(\Gamma^{\re})} +  \| \mathbb{P}^*_{\perp} \vec{\eta} \|^2_{H^{-1/2}(\Gamma^{\re})} \right).
		\end{align*} 
	\end{proposition} 
	\begin{proof}
		\mypar{Step 1. Proof of coercivity estimate.}
		Let us define $v^{\re}_{\vec{\eta}} := \mathcal{S}^{\re}_{\omega} \vec{\eta} \in H^1(\mathbb{R}^3)$.
		We will use the following identity that follows from the Green's formula and the definition of the layer potentials: for any $\Lambda\in H^1(\mathbb{R}^3)$, s.t. $\gamma_0 \Lambda=\vec{\lambda}$, 
		\begin{align}
			\label{eq:identity}
			\begin{split}
				\langle \vec{\eta}, \vec{\lambda}\rangle_{-1/2,1/2,\Gamma^{\re}}&=\int_{\mathbb{R}^3\setminus\Gamma^{\re}} \Delta v^{\re}_{\vec{\eta}}\, \Lambda\, d\vec{x}-\int_{\mathbb{R}^3\setminus\Gamma^{\re}} \nabla  v^{\re}_{\vec{\eta}}\,\nabla \Lambda d\vec{x}\\
				&=-\omega^2\int_{\mathbb{R}^3} v^{\re}_{\vec{\eta}}\, \Lambda d\vec{x}-\int_{\mathbb{R}^3} \nabla  v^{\re}_{\vec{\eta}}\,\nabla \Lambda d\vec{x}=:a_{\omega}(v^{\re}_{\eta}, {\Lambda}),
			\end{split}
		\end{align}
		see the proof of \cite[Proposition 2.6.1, after (2.17)]{sayas}. The above with ${\Lambda}= \overline{\omega \mathcal{S}^{\re}_{\omega} \vec{\eta}}$ yields
		\begin{align}
			\label{eq:etaSeta}
			\begin{split}
				-\Im\langle \vec{\eta}, \overline{\omega S^{\re}_{\omega} \vec{\eta}}\rangle 
				&=-\Im \left(\overline{\omega} \int_{\mathbb{R}^3}  |\nabla v^{\re}_{\vec{\eta}} (\boldsymbol{x})|^2  -  |\omega|^2\omega \int_{\mathbb{R}^3} |v^{\re}_{\vec{\eta}}(\boldsymbol{x})|^2\right) \\
				&= \Im \omega  \| v^{\re}_{\vec{\eta}} \|^2_{|\omega|, \mathbb{R}^3}.
			\end{split}
		\end{align}
		It remains to relate the norms of $v^{\re}_{\vec{\eta}}$ and $\vec{\eta}$. With Proposition \ref{ortho_decompositionHm12}, $\vec{\eta} = \mathbb{P}^*_{0} \vec{\eta} + \mathbb{P}^*_{\perp} \vec{\eta}$ and we estimate these quantities separately. With Proposition \ref{prop:norm_equivalence}, 
		\begin{align*}%
			\|\mathbb{P}_0^*\vec{\eta} \|_{H^{-1/2}(\Gamma^{\re})} &=\sup_{\vec{\lambda}\in H^{-1/2}(\Gamma^{\re})\setminus\{0\}}\frac{|\langle  \vec{\eta},\mathbb{P}_0\vec{\lambda}\rangle_{-1/2,1/2}|}{\|\vec{\lambda}\|_{H^{-1/2}(\Gamma^{\re})}}= \sup_{\vec{\lambda}\in \mathbb{S}_0\setminus\{0\}}\frac{|\langle  \vec{\eta},\mathbb{P}_0\vec{\lambda}\rangle_{-1/2,1/2}|}{\|\mathbb{P}_0\vec{\lambda}\|_{L^2(\Gamma^{\re})}}, 
		\end{align*}
		and, using \eqref{eq:identity} and next the Cauchy-Schwarz inequality, 
		\begin{align}%
			\nonumber
			\|\mathbb{P}_0^*\vec{\eta} \|_{H^{-1/2}(\Gamma^{\re})} &=\sup_{\vec{\lambda}\in \mathbb{S}_0\setminus\{0\}}\inf_{\substack{\Lambda^{\re} \in H^1(\mathbb{R}^3): \\  \gamma_0\Lambda^{\re}=\vec{\lambda}}}\frac{|a_{\omega}(v^{\re}_{\vec{\eta}}, \Lambda^{\re})|}{\|\vec{\lambda}\|_{L^2(\Gamma^{\re})}}\\
			\label{eq:b0}
			&\leq \| v^{\re}_{\vec{\eta}} \|_{|\omega|, \mathbb{R}^3} \sup_{\vec{\lambda}\in \mathbb{S}_0\setminus\{0\}}\inf_{\substack{\Lambda^{\re} \in H^1(\mathbb{R}^3): \\  \gamma_0\Lambda^{\re}=\vec{\lambda}}}\frac{\|\Lambda^{\re}\|_{|\omega|,\mathbb{R}^3}}{\|\vec{\lambda}\|_{L^2(\Gamma^{\re})}}=\| v^{\re}_{\vec{\eta}} \|_{|\omega|, \mathbb{R}^3} L^{\re}_{\mathbb{S}_0}(|\omega|),
		\end{align}		
		see \eqref{eq:lva}. Proceeding in a similar manner for $\mathbb{P}_{\perp}^*\vec{\eta}$, we conclude that 
		\begin{align}%
			\label{eq:b1}
			\|\mathbb{P}_{\perp}^*\vec{\eta} \|_{H^{-1/2}(\Gamma^{\re})}\leq \| v^{\re}_{\vec{\eta}} \|_{|\omega|, \mathbb{R}^3} L_{H^{1/2}_{\perp}(\Gamma^{\re})}^{\re}(|\omega|).
		\end{align}
		Summing up the bounds (\eqref{eq:b0}), (\eqref{eq:b1}), using the result in \eqref{eq:etaSeta}, and next employing Proposition \ref{prop:lifting_lemma} yields 
		\begin{align*}
			-\Im\langle \vec{\eta}, \overline{\omega S^{\re}_{\omega} \vec{\eta}}\rangle \gtrsim (\underline{d}^{\re}_*)^{2} (1+|\omega|)^{-1} \min(1,  |\omega|^{2})\Im \omega (\re	\|\mathbb{P}_0^*\vec{\eta} \|^2+	\|\mathbb{P}_{\perp}^*\vec{\eta} \|^2).
		\end{align*}
		The stated coercivity estimate is obtained with $\min(1,  |\omega|^{2})\Im \omega\gtrsim \min(1, (\Im\omega)^3)$. 

		\mypar{Step 2. Proof that $S^{\re}_{\omega}$ is continuous and invertible. }
		Once we argue that $S^{\re}_{\omega}: \, H^{-1/2}(\Gamma^{\re})\rightarrow H^{1/2}(\Gamma^{\re})$ is continuous and coercive, the Lax-Milgram lemma, cf. \cite[Lemma 2.32]{mclean} implies that it is invertible. 
		The continuity and coercivity of the single layer boundary integral operator were established in \cite{bamberger_ha_duong} for a single obstacle, based on the energy-type argument easily extendable to the multiple-obstacle case. Since an optimal continuity estimate is not important in what follows, we omit the proof.
	\end{proof}
	We will use the coercivity bound of Proposition \ref{prop:coercivity_bound} in a modified form, see below.
	\begin{theorem}
		\label{theorem:coercivity_bound_principal}
		For all $\vec{\eta} \in H^{-1/2}(\Gamma^{\re})$, all $\re \in (0,1]$ and $\omega\in \mathbb{C}^+$,
		\begin{align*}%
			-\Im \langle {\vec{\vec{\eta}}},e^{-i \operatorname{Arg}(\omega)} \overline{  S^{\re}_{\omega} \vec{\vec{\eta}}}\rangle &\gtrsim  c_S \left(\re \| \mathbb{Q}_{\sigma}\vec{\eta} \|^2_{H^{-1/2}(\Gamma^{\re})} +  \|  \mathbb{Q}_{\perp}\vec{\eta} \|^2_{H^{-1/2}(\Gamma^{\re})} \right),\\
			c_S&=(\underline{d}^{\re}_*)^{2} (1+|\omega|)^{-2} \min(1, (\Im \omega)^3).
		\end{align*}
	\end{theorem} 
	\begin{proof}
		The result is immediate by combining Proposition \ref{prop:coercivity_bound}, Proposition \ref{prop:bound_control} and the bound  $|\omega|^{-1}\geq (1+|\omega|)^{-1}$. 
	\end{proof}
	The above results yield a stability bound for the asymptotic Galerkin method \eqref{eq:coercivity_bound_1}.
	\begin{corollary}
		\label{corollary:stability_bound}
		For all $\re\in (0,1]$ and $\omega \in \mathbb{C}^+$, the operator $\mathbb{S}^{\re}_{\sigma\sigma}: \, \mathcal{V}^{\re}_G\rightarrow \mathbb{S}_0$ is invertible, and satisfies
		\begin{align*}%
			&\| (\mathbb{S}^{\re}_{\sigma \sigma})^{-1} \|_{H^{1/2}(\Gamma^{\re})\rightarrow H^{-1/2}(\Gamma^{\re})} \lesssim \re^{-1}  (\underline{d}^{\re}_{*})^{-2} (1+  |\omega|  )^2 \max(1, (\Im \, \omega)^{-3}).
		\end{align*}
	\end{corollary}
	\begin{proof}
		Recall that $\mathbb{S}^{\re}_{\sigma \sigma} = \left.\mathbb{Q}^*_{\sigma} S^{\re}_{\omega}\right|_{\mathcal{V}^{\re}_G}: \mathcal{V}^{\re}_G \to \mathbb{S}_0$. Since  $\dim(\mathcal{V}^{\re}_G) = \dim(\mathbb{S}_0) < +\infty$, it is sufficient to show that $a(\vec{\varphi}, \, \vec{\eta}):=\langle \vec{\varphi}, \overline{\mathbb{S}^{\re}_{\sigma \sigma}\vec{\eta}}\rangle_{-1/2,1/2}$ satisfies an $\operatorname{inf}$-$\operatorname{sup}$-condition on $\mathbb{S}_0\times \mathcal{V}_G^{\re}$, see \cite[Section 3.4.2]{boffi}. In particular, for $\vec{\eta}_{\sigma}\in \mathcal{V}^{\re}_G$, we write
		\begin{align}
			\label{eq:ets}
			\langle \mathbb{P}_0^*\vec{\eta}_{\sigma}, \overline{\mathbb{S}^{\re}_{\sigma\sigma}\vec{\eta}_{\sigma}}\rangle_{-1/2,1/2} = \langle \mathbb{Q}_{\sigma}\mathbb{P}_0^*\vec{\eta}_{\sigma}, \overline{\mathbb{S}^{\re}\vec{\eta}_{\sigma}}\rangle_{-1/2,1/2} =\langle \vec{\eta}_{\sigma}, \overline{\mathbb{S}^{\re}\vec{\eta}_{\sigma}}\rangle_{-1/2,1/2},
		\end{align}
		where we used \eqref{eq:qperp_pperp}(b) and $\mathbb{Q}_{\sigma}\vec{\eta}_{\sigma}=\vec{\eta}_{\sigma}$. By Theorem \ref{theorem:coercivity_bound_principal},
		\begin{align*}
			\left|	\langle \mathbb{P}_0^*\vec{\eta}_{\sigma}, \overline{\mathbb{S}^{\re}_{\sigma\sigma}\vec{\eta}_{\sigma}}\rangle_{-1/2,1/2} \right|\gtrsim c_s\re \|\vec{\eta}_{\sigma}\|^2_{H^{-1/2}(\Gamma^{\re})},
		\end{align*}
		which implies that $\mathbb{S}^{\re}_{\sigma\sigma}$ is invertible. The stability bound follows by remarking that $|\langle \vec{\eta}_{\sigma}, \overline{\mathbb{S}^{\re}_{\sigma\sigma}\vec{\eta}_{\sigma}}\rangle_{-1/2,1/2, \Gamma^{\re}}|$ satisfies the above coercivity bound, see \eqref{eq:ets}. 
	\end{proof}
	Similarly, we obtain the stability bound \eqref{eq:Q_perp_S_inv_Q_perp}.  
	\begin{corollary}
		There exists a constant $C>0$, s.t. for all $\re \in (0,1]$ and $\omega \in \mathbb{C}^+$ the following bound holds true:
		\begin{align*}%
			\| \mathbb{Q}_{\perp} (S^{\re}_{\omega})^{-1} \mathbb{Q}_{\perp}^* \|_{\mathcal{L}(H^{1/2}(\Gamma^{\re}), H^{-1/2} (\Gamma^{\re}))} \leq C(\underline{d}^{\re}_{*})^{-2} (1+ |\omega| )^2 \max(1, (\Im \, \omega)^{-3}).
		\end{align*}
	\end{corollary}
	\begin{proof}
		Let $\vec{\varphi}\in H^{1/2}(\Gamma^{\re})$, then, by Proposition \ref{prop:coercivity_bound}, there exists $\vec{\eta} \in H^{-1/2}(\Gamma^{\re})$ such that $ \vec{\eta} =  (S^{\re}_{\omega})^{-1} \mathbb{Q}^*_{\perp} \vec{\varphi}$.
		By Theorem \ref{theorem:coercivity_bound_principal}, 
		\begin{align}
			\label{eq:qperp_bound}
			\| \mathbb{Q}_{\perp} \vec{\eta} \|^2
			_{H^{-1/2}(\Gamma^{\re})}\leq C^{-1} (\underline{d}^{\re}_*)^{-2} (1+|\omega|)^{2} \max(1, (\Im \omega)^{-3}) \left|\langle \vec{\eta}, \overline{S^{\re}_{\omega}\vec{\eta}}\rangle_{-1/2, 1/2} \right|.
		\end{align}
		With $S^{\re}_{\omega}\vec{\eta}=\mathbb{Q}^*_{\perp} \vec{\varphi}$, we replace in the above
		\begin{align*}%
			\langle \vec{\eta}, \overline{S^{\re}_{\omega}\vec{\eta}}\rangle_{-1/2, 1/2} =	\langle \vec{\eta}, \overline{\mathbb{Q}^*_{\perp} \vec{\varphi}}\rangle_{-1/2, 1/2}=\langle \mathbb{Q}_{\perp}\vec{\eta}, \overline{\vec{\varphi}}\rangle_{-1/2, 1/2},
		\end{align*}
		which implies, by recalling that $\vec{\eta}=(S^{\re}_{\omega})^{-1} \mathbb{Q}^*_{\perp}\vec{\varphi}$, the desired bound
		\begin{align*}%
			\| \mathbb{Q}_{\perp} (S^{\re}_{\omega})^{-1} \mathbb{Q}^*_{\perp}\vec{\varphi}\|_{H^{-1/2}(\Gamma^{\re})}\leq  C^{-1} (\underline{d}^{\re}_*)^{-2} (1+|\omega|)^{2} \max(1, (\Im \omega)^{-3})\|\vec{\varphi}\|_{H^{1/2}(\Gamma^{\re})}.
		\end{align*}	
	\end{proof}
	\subsubsection{Proof of the continuity bound on the operator norms \eqref{eq:continuity_bound}}
	\label{sec:continuity_bound} 
	In this section, we obtain a bound on the operator norm of the 'off-diagonal' terms \eqref{eq:continuity_bound}. We start with the operator $\mathbb{S}_{\sigma\perp}^{\re}$. First, remark that
	\begin{align*}%
		\mathbb{S}_{\sigma\perp}^{\re}=\mathbb{Q}_{\sigma}^*\left.S^{\re}_{\omega}\right|_{H^{-1/2}_{\perp}(\Gamma^{\re})}:\, H^{-1/2}_{\perp}(\Gamma^{\re})\rightarrow \mathbb{S}_0
	\end{align*}
	satisfies: for all $\vec{\varphi}\in H^{-1/2}_{\perp}(\Gamma^{\re})$, $
	\mathbb{S}_{\sigma\perp}^{\re}\vec{\varphi}=\mathbb{Q}_{\sigma}^* S^{\re}_{\omega}\mathbb{Q}_{\perp}\vec{\varphi},$ 
	see \eqref{eq:qperp_pperp}(a). Therefore, $$
	\|\mathbb{S}_{\sigma\perp}^{\re}\|_{\mathcal{L}(H^{-1/2}_{\perp}(\Gamma^{\re}), H^{1/2}(\Gamma^{\re}))}\leq \|\mathbb{Q}_{\sigma}^* S^{\re}_{\omega}\mathbb{Q}_{\perp}\|_{\mathcal{L}(H^{-1/2}(\Gamma^{\re}), H^{1/2}(\Gamma^{\re}))},$$
	and we will estimate this latter norm, recalling the following identity:
	\begin{align}
		\label{eq:split0}
		\| \mathbb{Q}_{\sigma}^*\mathbb{S}^{\re}_{\omega}\mathbb{Q}_{\perp} \|_{\mathcal{L}( H^{-1/2}(\Gamma^{\re}), H^{1/2}(\Gamma^{\re}))} &= \sup_{\vec{\lambda}  \in H^{-1/2}(\Gamma^{\re}) \setminus \{0\}} \frac{\| \mathbb{Q}^*_{\sigma} S^{\re}_{\omega} \mathbb{Q}_{\perp} \vec{\lambda} \|_{H^{1/2}(\Gamma^{\re})}}{\| \vec{\lambda} \|_{H^{-1/2}(\Gamma^{\re})}} \nonumber \\
		&= \sup_{\vec{\lambda}, \vec{v}  \in H^{-1/2}(\Gamma^{\re}) \setminus \{0\}} \frac{|\langle  S^{\re}_{\omega} \mathbb{Q}_{\perp}\vec{\lambda}, \mathbb{Q}_{\sigma}\vec{v}\rangle_{1/2, -1/2}|}{\| \vec{\lambda} \|_{H^{-1/2}(\Gamma^{\re})} \| \vec{v} \|_{H^{-1/2}(\Gamma^{\re})}}. 
	\end{align}
	
	In what follows, let us set $\vec{\lambda}_{\perp}:=\mathbb{Q}_{\perp}\vec{\lambda}$ and $\vec{v}_{\sigma}:=\mathbb{Q}_{\sigma}\vec{v}$. 
	The duality product in the above can be split into two terms, which correspond to  self- and cross-interactions:
	\begin{align}
		\label{eq:split}
		\langle  S^{\re}_{\omega} \mathbb{Q}_{\perp}\vec{\lambda}, \mathbb{Q}_{\sigma}\vec{v}\rangle_{1/2, -1/2}=\underbrace{\sum\limits_{k=1}^N\langle S^{\re,kk}_{\omega}\lambda_{\perp,k}, v_{\sigma,k}\rangle_{\Gamma^{\re}_k}}_{C_{si}}+\underbrace{\sum\limits_{k=1}^N\sum\limits_{\ell\neq k}\langle S^{\re,k\ell}_{\omega}\lambda_{\perp,\ell}, v_{\sigma,k}\rangle_{\Gamma^{\re}_k}}_{C_{ci}}.
	\end{align}
	Recall that $S_{\omega}^{\re,k\ell}$ is defined in \eqref{eq:operator_S_kl}. The above terms $C_{si}$ and $C_{ci}$ will be handled in different manners. 	
	
	\paragraph{Estimating self-interacting terms $C_{si}$}
	By density, it suffices to consider $C_{si}$ only for functions $\vec{v}\in L^2(\Gamma^{\re})$. In this case, the duality product writes
	
	\begin{align}
		\label{eq:skk}
		\langle  S^{\re,kk}_{\omega} \lambda_{\perp,k}, v_{\sigma,k}\rangle_{1/2, -1/2}=\int_{\Gamma^{\re}_k}\int_{\Gamma^{\re}_k}G_{\omega}(|\bx-\by|){\lambda}_{\perp,k}(\vec{y})d\Gamma_{\vec{y}}\; v_{\sigma,k}(\vec{x})d\Gamma_{\vec{x}}.
	\end{align}
	Let us now introduce the following new integral kernel (stems from the Taylor expansion of $\mathrm{e}^{i\omega r}$ around $\omega=0$): 
	\begin{align*}%
		\mathcal{K}_{\omega}(r) &:= G_{\omega}(r) - G_0(r) - \frac{i \omega}{4 \pi} = \frac{e^{i\omega r} - 1 - i \omega r}{4\pi r}, \quad r>0,
	\end{align*}
	and the associated operator in $\mathcal{L}(H^{-1/2}(\Gamma^{\re}_k), H^{1}(\mathbb{R}^{3}))$:
	\begin{align}
		\label{eq:Edefvol}
		&\mathcal{E}_{\omega}^{\re,k}\varphi=\int_{\Gamma^{\re}_k}\mathcal{K}_{\omega}(|\vec{x}-\vec{y}|)\varphi(\vec{y})d\Gamma_{\vec{y}},\quad \vec{x}\in \mathbb{R}^3,\\
		\nonumber
		&E_{\omega}^{\re,kk}=\gamma_0 \mathcal{E}_{\omega}^{\re,k}, \quad \text{so that}\quad E_{\omega}^{\re,kk}\varphi(\vec{x})=\int_{\Gamma^{\re}_k}\mathcal{K}_{\omega}(|\vec{x}-\vec{y}|)\varphi(\vec{y})d\Gamma_{\vec{y}},\quad 
		\vec{x}\in \Gamma^{\re}_k.
	\end{align}
	With this new notation, we can rewrite
	\begin{align*}%
		(S^{\re,kk}_{\omega}\varphi)(\bx) = (S_{0}^{\re,kk}\varphi)(\bx)+\frac{i\omega}{4\pi}\int_{\Gamma^{\re}_k}\varphi(\by) d\Gamma_{\by} + (E^{\re,kk}_{\omega}\varphi)(\bx), \quad \bx \in \Gamma^{\re}_k, 
	\end{align*}
	and inserting the above expression with $\varphi=\lambda_{\perp,k}$ into \eqref{eq:skk} yields 
	\begin{align*}%
		\langle  S^{\re,kk}_{\omega} \lambda_{\perp,k}, v_{\sigma,k}\rangle&=\langle S_{0}^{\re,kk}\lambda_{\perp,k}, v_{\sigma,k}\rangle+\frac{i\omega}{4\pi}\int_{\Gamma_k^{\re}}\lambda_{\perp,k}(\vec{y})d\Gamma_{\vec{y}}  \langle 1, v_{\sigma,k}\rangle+\langle E_{\omega}^{\re,kk} \lambda_{\perp,k}, v_{\sigma,k}\rangle\\
		&=\langle E_{\omega}^{\re,kk} \mathbb{P}_{\perp}^*\lambda_{\perp,k}, v_{\sigma,k}\rangle.
	\end{align*}
	In the above we used that $\langle S_0^{\re,kk}\lambda_{\perp,k}, v_{\sigma,k}\rangle=\langle \lambda_{\perp,k}, S_0^{\re,kk}v_{\sigma,k}\rangle=\langle \lambda_{\perp,k},1\rangle=0$ (recall that $\vec{v}_{\sigma}\in \mathcal{V}_{G}^{\re}$) and $\vec{\lambda}_{\perp}\in H^{-1/2}(\Gamma^{\re})$. Next, we rewrite the above by using the symmetry of the integral kernel of $E_{\omega}^{\re,kk}$ and the density argument: 
	\begin{align}
		\label{eq:key_ident}
		\langle  S^{\re,kk}_{\omega} \lambda_{\perp,k}, v_{\sigma,k}\rangle_{1/2, -1/2}=	\langle \lambda_{\perp,k}, \mathbb{P}_{\perp}E_{\omega}^{\re,kk}v_{\sigma,k}\rangle_{-1/2,1/2}.	
	\end{align}
	It remains thus to find a suitable estimate on the operator norm $E^{\re,kk}_{\omega}$. While right away we make use of $\mathbb{P}_{\perp}E_{\omega}^{\re,kk}$, for future computations we will need an estimate on the operator norm of $E_{\omega}^{\re,kk}$ as well. 	
	\begin{proposition}
		\label{prop:Ekk}
		There exists $C>0$ which depends only on $\Gamma_j, \, j=1,\ldots,N$, s.t. for all $\re\in (0,1)$, $\omega \in \mathbb{C}^+$ and $k=1,\ldots, N$, it holds that 
		\begin{align}
			\label{eq:pe}
			&\|\mathbb{P}_{\perp}E^{\re,kk}_{\omega} \|_{\mathcal{L}(L^2(\Gamma^{\re}_k),H^{1/2}(\Gamma^{\re}_k))} \leq C \re^{5/2} |\omega|^2,\\
			\label{eq:pel2}
			&\|E^{\re,kk}_{\omega} \|_{\mathcal{L}(L^2(\Gamma^{\re}_k),L^2(\Gamma^{\re}_k))} \leq C \re^{3} |\omega|^2.
		\end{align}
	\end{proposition}
	The proof of this proposition is based on the following auxiliary result. 
	\begin{lemma}
		\label{lemma:kernel_bound}
		The function $(\vec{x}, \vec{y})\mapsto \mathcal{K}_{\omega}(|\bx-\by|)\in C^1(\mathbb{R}^3\times\mathbb{R}^3)$. 
		Moreover, there exist $C_k, C_k'>0$, s.t. for all $\re \in (0,1)$ and $\omega\in \mathbb{C}^+$, it holds 
		\begin{align}
			\label{eq:Kbound}	
			(a)~\max_{\vec{x}, \, \vec{y}\in \Gamma^{\re}_k}|\mathcal{K}_{\omega}(|\vec{x}-\vec{y}|)|\leq C_k\varepsilon|\omega|^2,\quad
			(b)	~ \max_{\vec{x}, \vec{y}\in \mathbb{R}^3}|\nabla_{\vec{x}}\mathcal{K}_{\omega}(|\vec{x}-\vec{y}|)|\leq C_k'|\omega|^2.
		\end{align}
	\end{lemma}
	\begin{proof}
		From the definition of $r\mapsto\mathcal{K}_{\omega}(r)$ it is immediate that it is $C^{\infty}(0, \infty)$, therefore, in our case, it is sufficient to check that $r\mapsto \mathcal{K}_{\omega}'(r)$ admits a limit as $r\rightarrow 0$. The desired result is obtained by a direct computation for
		\begin{align}
			\label{eq:komega}
			\mathcal{K}_{\omega}'(r)=\omega^2\frac{i\omega r \mathrm{e}^{i\omega r}-\mathrm{e}^{i\omega r}+1}{4\pi \omega^2 r^2}.
		\end{align}
		\mypar{Proof of \eqref{eq:Kbound}(a).} 
		By the mean-value theorem, we have 
		\begin{align}
			\label{eq:uppbound}
			\max_{\vec{x}, \, \vec{y}\in \Gamma^{\re}_k}|\mathcal{K}_{\omega}(|\vec{x}-\vec{y}|)|\leq \max_{r\in [0, \operatorname{diam}\Gamma_k^{\re}]}|\mathcal{K}_{\omega}'(r)|\operatorname{diam}\Gamma_k^{\re},
		\end{align}
		where we used $
		\max_{\vec{x}, \, \vec{y}\in \Gamma^{\re}_k}|\mathcal{K}_{\omega}(|\vec{x}-\vec{y}|)|=\max_{r\in [0, \operatorname{diam}\Gamma_k^{\re}]}|\mathcal{K}_{\omega}(r)|$ and  $\mathcal{K}_{\omega}(0)=0$.
		Next, a direct computation \eqref{eq:komega} shows that $\mathcal{K}'_{\omega}=\omega^2 f(\omega r)$,  
		with $f(z)=\frac{z\mathrm{e}^{iz}-\mathrm{e}^{iz}+1}{4\pi z^2}$. The function $z\mapsto f(z)$ is analytic in $\mathbb{C}^+$,  continuous up to the boundary of $\mathbb{C}^+$ (i.e. well-defined on $\mathbb{R}$), is uniformly bounded by $C/|z|$, $C>0$, for $z\in \mathbb{C}^+\cup\mathbb{R}$ s.t. $|z|\rightarrow+\infty$. By the maximum principle, $z\mapsto|f(z)|$ is bounded by a constant. 	Thus, 
		\begin{align}
			\label{eq:uniform_bound}
			|\mathcal{K}'_{\omega}(r)|\leq C|\omega|^2, \quad \forall \omega\in \mathbb{C}^+, \quad  r\geq 0.
		\end{align}
		With \eqref{eq:uppbound}, this yields the desired bound.
		
		\mypar{Proof of \eqref{eq:Kbound}(b).}	
		We start by remarking that 
		\begin{align}
			\label{eq:grad_K}
			\nabla_{\vec{x}}\mathcal{K}_{\omega}(|\vec{x}-\vec{y}|)=\mathcal{K}'_{\omega}(|\vec{x}-\vec{y}|)\frac{(\vec{x}-\vec{y})}{|\vec{x}-\vec{y}|}.
		\end{align}
		Using \eqref{eq:uniform_bound} yields the desired bound \eqref{eq:Kbound}(b). 
		
	\end{proof}
	\begin{proof}[Proof of Proposition \ref{prop:Ekk}]
		\mypar{Proof of \eqref{eq:pe}. }
		Let $\varphi \in L^2(\Gamma^{\re}_k)$. We use the norm equivalence result of Proposition \ref{prop:norm_equivalence}:
		\begin{align}
			\label{eq:Pk}
			\|\mathbb{P}_{\perp}E^{\re,kk}_{\omega}\varphi\|_{H^{1/2}(\Gamma^{\re})}\lesssim |E^{\re,{kk}}_{\omega}\varphi|_{H^{1/2}(\Gamma^{\re}_k)}.
		\end{align}
		To estimate the $H^{1/2}(\Gamma^{\re}_k)$-seminorm, we make use of Lemma \ref{lem:normh12}. To justify its use, let us remark that for $\varphi\in L^2(\Gamma^{\re}_k)$,  $\mathcal{E}^{\re,k}_{\omega}\varphi\in C^1(\overline{B^{\re}_k})$, cf. Lemma \ref{lemma:kernel_bound}:
		\begin{align}
			\label{eq:Erek}
			|E^{\re,kk}_\omega \varphi |_{H^{1/2}(\Gamma^{\re}_k)}\lesssim \varepsilon^{3/2} \|\nabla \mathcal{E}^{\re,k}_{\omega}\varphi\|_{L^{\infty}(B^{\re}_k)}.
		\end{align}
		Using \eqref{eq:uniform_bound}(b) allows us to interchange $\nabla$ and integration, and we obtain
		\begin{align*}%
			(\nabla \mathcal{E}^{\re,k}_{\omega}\varphi)(\vec{x})=\int_{\Gamma^{\re}_k} \nabla_{\vec{x}}\mathcal{K}_{\omega}(|\bx-\by|)\varphi(\vec{y})d\Gamma_{\vec{y}} \Rightarrow \|\nabla \mathcal{E}^{\re,k}_{\omega}\varphi\|_{L^{\infty}(B_k^{\re})}\leq C_k \re |\omega|^2\|\varphi\|_{L^2(\Gamma^{\re}_k)},
		\end{align*}
		where the last bound follows from \eqref{eq:Kbound}(b), the Cauchy-Schwarz inequality and the bound $\|  1 \|_{L^2(\Gamma^{\re}_k)} \lesssim \re$. Finally, combining the above bound with \eqref{eq:Erek} yields 
		\begin{align*}%
			|E^{\re,kk}_{\omega}\varphi|_{H^{1/2}(\Gamma^{\re}_k)}\leq C_k \varepsilon^{5/2}|\omega|^2\|\varphi\|_{L^2(\Gamma^{\re}_k)},
		\end{align*}
		which, with \eqref{eq:Pk} results in the desired statement.\\
		\mypar{Proof of \eqref{eq:pel2}.} Let $\varphi\in L^2(\Gamma^{\re}_k)$. We make use of Lemma \ref{lem:norml2}  which gives
		\begin{align*}%
			\|E_{\omega}^{\re,kk}\varphi\|_{L^2(\Gamma^{\re}_k)}\lesssim \re	\|\mathcal{E}_{\omega}^{\re,k}\varphi\|_{L^{\infty}(B^{\re}_k)}.
		\end{align*}
		We proceed like in the previous step but use now  \eqref{eq:uniform_bound}(a). 
	\end{proof}
	Now we can plug the bound of Proposition \ref{prop:Ekk} into the identity \eqref{eq:key_ident}. Applying Proposition \ref{prop:L^2_H_minus_half_bound}, in order to control the $L^2$ norm of $v_{\sigma, k}$ via the $H^{-1/2}$ norm, we obtain 
	\begin{align*}%
		|\langle S^{\re,kk}_{\omega} \lambda_{\perp, k}, v_{\sigma, k} \rangle_{\Gamma^{\re}_k}| &= |\langle \lambda_{\perp,k}, \mathbb{P}_{\perp} E^{\re,k}_{\omega} v_{\sigma, k} \rangle_{\Gamma^{\re}_k}|  \\
		&\leq \| \mathbb{P}_{\perp} E^{\re,kk}_{\omega} \|_{L^2(\Gamma^{\re}_k) \rightarrow H^{1/2}(\Gamma^{\re}_k)} \| \lambda_{\perp, k} \|_{H^{-1/2}(\Gamma^{\re}_k)} \| v_{\sigma, k} \|_{L^2(\Gamma^{\re}_k)} \\
		&\leq C_k \re^{5/2} |\omega|^2  \| \lambda_{\perp, k} \|_{H^{-1/2}(\Gamma^{\re}_k)} \| v_{\sigma, k} \|_{H^{-1/2}(\Gamma^{\re}_k)}.
	\end{align*}
	Combining this bound with the definition of $C_{si}$ in \eqref{eq:split} and using the $\ell_2$ Cauchy-Schwarz inequality yields the following key bound: 
	\begin{align}
		\label{eq:Csi}
		C_{si} &\leq C\re^{5/2}|\omega|^2 \|\vec{\lambda}_{\perp}\|_{H^{-1/2}(\Gamma^{\re})}\|\vec{v}_{\sigma}\|_{H^{-1/2}(\Gamma^{\re})}.
	\end{align}
	\paragraph{Estimating cross-interacting terms $C_{ci}$}
	
	We start with one of the terms, for $\ell\neq k$. By the symmetry of the integral kernel and density argumentm
	\begin{align*}%
		\langle S_{\omega}^{\re,\ell k}\lambda_{\perp,k}, v_{\sigma,\ell}\rangle_{\Gamma^{\re}_{\ell}}=\langle \lambda_{\perp,k}, S_{\omega}^{\re,k \ell}v_{\sigma,\ell}\rangle_{\Gamma^{\re}_k}=\langle \lambda_{\perp,k}, \mathbb{P}_{\perp}S_{\omega}^{\re,k \ell}v_{\sigma,\ell}\rangle_{\Gamma^{\re}_k}.
	\end{align*}
	Applying the norm equivalence of Proposition \ref{prop:norm_equivalence} yields
	\begin{equation}
		\label{eq:somega0}
		\begin{aligned}
			|\langle \mathbb{P}_{\perp}S_{\omega}^{\re,k \ell}v_{\sigma,\ell},  \lambda_{\perp,k} \rangle_{\Gamma^{\re}_k}| \lesssim 
			|\mathbb{P}_{\perp}S_{\omega}^{\re,k\ell}v_{\sigma,\ell}|_{H^{1/2}(\Gamma^{\re}_k)}\|\lambda_{\perp,k}\|_{H^{-1/2}(\Gamma^{\re}_k)}.
		\end{aligned}
	\end{equation}
	Since $|\mathbb{P}_{\perp}v|_{H^{1/2}(\Gamma^{\re}_k)}=|v-\mathbb{P}_0v|_{H^{1/2}(\Gamma^{\re}_k)}=|v|_{H^{1/2}(\Gamma^{\re}_k)}$, we also have that 
	\begin{align*}%
		|\langle \mathbb{P}_{\perp}S_{\omega}^{\re,k \ell}v_{\sigma,\ell},  \lambda_{\perp,k} \rangle_{\Gamma^{\re}_k}| \lesssim 
		|S_{\omega}^{\re,k\ell}v_{\sigma,\ell}|_{H^{1/2}(\Gamma^{\re}_k)}\|\lambda_{\perp,k}\|_{H^{-1/2}(\Gamma^{\re}_k)}.
	\end{align*}
	Finally, with Lemma \ref{lem:normh12}, 
	\begin{align}
		\label{eq:somega}	
		\begin{split}
			|S_{\omega}^{\re,k\ell}v_{\sigma,\ell}|_{H^{1/2}(\Gamma^{\re}_k)} &\leq C \varepsilon^{3/2}\left\|\nabla_{\vec{x}}\int_{\Gamma_{\ell}^{\re}}G_{\omega}(|\bx-\by|)v_{\sigma,\ell}(\vec{y})d\Gamma_{\vec{y}}\right\|_{L^{\infty}(B_k^{\re})} \\
			&\leq C \varepsilon^{5/2}\max_{\vec{x}\in B_{\ell}^{\re}, \vec{y}\in B_k^{\re}}|\nabla_{\vec{x}}G_{\omega}(|\vec{x}-\vec{y}|)|\|v_{\sigma, {\ell}}\|_{L^2(\Gamma^{\re}_{\ell})},
		\end{split}
	\end{align}
	where in the last inequality we used the fact that $\vec{x}\rightarrow G_{\omega}(|\vec{x}-\vec{y}|)$ is a $C^{\infty}$ function, the Cauchy-Schwarz inequality and the bound $\|1\|_{L^2(\Gamma^{\re}_{\ell})}\lesssim \varepsilon$. Since 
	\begin{align*}%
		\nabla_{\bx} G_{\omega}(|\bx-\by|)=\frac{i\omega \mathrm{e}^{i\omega|\bx-\by|}}{4\pi|\bx-\by|}\frac{(\bx-\by)}{|\bx-\by|}-\frac{\mathrm{e}^{i\omega|\bx-\by|}}{4\pi|\bx-\by|^3}(\vec{x}-\vec{y}),
	\end{align*}
	it is straightforward to see that for $\omega\in \mathbb{C}^+$,
	\begin{align}
		\label{eq:bound_grad_G}
		|\nabla_{\vec{x}}G_{\omega}(|\bx-\by|)|\lesssim \max(1,|\omega|) |\vec{x}-\vec{y}|^{-1}\left(1+|\vec{x}-\vec{y}|^{-1}\right)\lesssim (1+|\omega|)(\underline{d}^{\re}_*)^{-2}.
	\end{align}
	Combining the above and \eqref{eq:somega} into \eqref{eq:somega0}, we conclude that 
	\begin{align*}%
		|\langle S_{\omega}^{\re,k\ell}\lambda_{\perp,\ell}, v_{\sigma,k}\rangle| &\lesssim \re^{5/2} (\underline{d}^{\re}_*)^{-2}\max(1,|\omega|) \|v_{\sigma,\ell}\|_{L^2(\Gamma^{\re}_{\ell})}\|\lambda_{\perp, k}\|_{H^{-1/2}(\Gamma^{\re}_k)}.
	\end{align*}
	The above, Proposition \ref{prop:L^2_H_minus_half_bound},  \eqref{eq:split} and the $\ell_2$ Cauchy-Schwarz inequality yield:
	\begin{align}
		\nonumber
		C_{ci} &\leq C \varepsilon^{5/2} \max(1,|\omega|) (\underline{d}^{\re}_{*})^{-2} \sum_{\ell=1}^N \|\lambda_{\perp,\ell}\|_{H^{-1/2}(\Gamma^{\re}_{\ell})} \sum_{k=1}^N \| v_{\sigma, k} \|_{H^{-1/2}(\Gamma^{\re}_k)} \\
		\label{eq:cci}
		&\leq  C \varepsilon^{5/2} N \max(1,|\omega|) (\underline{d}^{\re}_{*})^{-2} \| \vec{v}_{\sigma} \|_{H^{-1/2}(\Gamma^{\re})} \| \vec{\lambda}_{\perp} \|_{H^{-1/2}(\Gamma^{\re})}.
	\end{align}
	\paragraph{Proof of the bound \eqref{eq:continuity_bound}}
	We combine the bounds \eqref{eq:Csi} and \eqref{eq:cci} into \eqref{eq:split0}: 
	\begin{align*}%
		\| \mathbb{Q}_{\sigma}^*\mathbb{S}^{\re}_{\omega}\mathbb{Q}_{\perp} \|_{\mathcal{L}( H^{-1/2}(\Gamma^{\re}), H^{1/2}(\Gamma^{\re}))} &\leq
		C\re^{5/2}N (1+|\omega|)^2(\underline{d}^{\re}_{*})^{-2} \times \\
		& \times\sup_{\vec{\lambda}, \vec{v}  \in H^{-1/2}(\Gamma^{\re}) \setminus \{0\}} \frac{ \|\mathbb{Q}_{\perp}\vec{\lambda}\|_{H^{-1/2}(\Gamma^{\re})} \|\mathbb{Q}_{\sigma}\vec{v}\|_{H^{-1/2}(\Gamma^{\re})}}{\| \vec{\lambda} \|_{H^{-1/2}(\Gamma^{\re})} \| \vec{v} \|_{H^{-1/2}(\Gamma^{\re})}},
	\end{align*}
	and the desired bound is immediate from Proposition \ref{prop:norm_projectors_Q}.
	
	The bound for $\mathbb{S}_{\perp\sigma}^{\re}$ can be obtained similarly, and thus we omit its proof.  
	\begin{remark}
		\label{remark:choice_of_decomposition}
		From the above proof, it follows that that the behaviour of the self-interacting terms $\langle S^{\re, kk}_{\omega} \lambda_{\perp,k}, v_{\sigma,k} \rangle_{\Gamma^{\re}_k}$  is very sensitive to the proper choice of the space decomposition of the solution space $H^{-1/2}(\Gamma^{\re})$; at the same time, it can be shown that for the cross-interacting terms $\langle S^{\re, k \ell}_{\omega} \lambda_{\perp, \ell}, v_{\sigma,k} \rangle_{\Gamma^{\re}_k}$, replacing $v_{\sigma,k}\in \mathcal{V}^{\re}_G$ by elements of $\mathbb{S}_0$ would  still yield an estimate of order $\re^{5/2}$.
	\end{remark}
	
	\subsubsection{Proof of Theorem \ref{theorem:bounds_data}}
	\label{sec:data_bound}
	The goal of this section is to prove that
	\begin{align*}%
		\|\mathbb{Q}^{*}_{\perp}( \hat{\vec{g}}^{\re}-\hat{\vec{g}}^{\re}_G)\|_{H^{1/2}(\Gamma^{\re})}  
		\lesssim C_{\omega, N, d^{\re}_*} (\re^{3/2} \|\hat{\vec{g}}^{\re} \|_{L^2(\Gamma^{\re})} + |  \hat{\vec{g}}^{\re} |_{H^{1/2}(\Gamma^{\re})} ),
	\end{align*}
	with $C_{\omega, N, d^{\re}_*} := N (\underline{d}^{\re}_{*})^{-4} (1+|\omega|)^4 \max(1, (\operatorname{Im} \omega)^{-3})$.
	The proof relies on the following auxiliary result. 
	\begin{lemma}[A stability bound on $\hat{\vec{\lambda}}^{\re}_G$]
		\label{lemma:stability_bound_lambda_G}
		There exists a constant $C>0$, s.t. for all $\re \in (0,1)$ and $\omega \in \mathbb{C}^+$, the Galerkin density $\hat{\vec{\lambda}}^{\re}_G$, defined in \eqref{eq:approximate_density_frequency_domain}, satisfies
		\begin{align*}%
			\| \hat{\vec{\lambda}}^{\re}_G \|_{H^{-1/2}(\Gamma^{\re})} \leq\re^{-1} \times C (\underline{d}^{\re}_{*})^{-2} (1+  |\omega|)^2 \max(1, (\Im  \omega)^{-3}) \| \hat{\vec{g}}^\re \|_{L^2(\Gamma^{\re})}. 
		\end{align*}
	\end{lemma}
	\begin{proof}
		From \eqref{eq:sresigma} it follows that  $\|\blambda^{\re}_G\|_{H^{-1/2}(\Gamma^{\re})}\leq \|\left(\mathbb{S}_{\sigma\sigma}^{\re}\right)^{-1}\|\|\mathbb{Q}_{\sigma}^*\vec{\hat{g}}^{\re}\|_{H^{1/2}(\Gamma^{\re})}$. It remains to combine Corollary \ref{corollary:stability_bound} and Proposition \ref{prop:qsigmaP0}. 
	\end{proof}
	\begin{proof}[Proof of Theorem \ref{theorem:bounds_data}]
		We begin by applying the triangle inequality 
		\begin{align*}%
			\| \mathbb{Q}^*_{\perp} (\hat{\vec{g}}^{\re} - \hat{\vec{g}}^{\re}_G) \|_{H^{1/2}(\Gamma^{\re})} \leq \| \mathbb{Q}^*_{\perp} \hat{\vec{g}}^{\re} \|_{H^{1/2}(\Gamma^{\re})} + \| \mathbb{Q}^*_{\perp} \hat{\vec{g}}^{\re}_G \|_{H^{1/2}(\Gamma^{\re})}.
		\end{align*}
		By Proposition \ref{prop:norm_dual_Q_perp}, for the first term in the rhs of the above, it holds that
		\begin{align*}%
			\| \mathbb{Q}^*_{\perp} \hat{\vec{g}}^{\re} \|_{H^{1/2}(\Gamma^{\re})} \lesssim  | \hat{\vec{\vec{g}}}^{\re}|_{H^{1/2}(\Gamma^{\re})}.
		\end{align*}
		To estimate the second term, we use that  $\hat{\vec{g}}^{\re}_G={S}^{\re}_{\omega}\hat{\blambda}^{\re}_G$, see \eqref{eq:BIE_density_error}, and $\hat{\blambda}_G^{\re}\in \mathcal{V}^{\re}_G$: 
		\begin{align*}%
			&\mathbb{Q}_{\perp}^*\hat{\vec{g}}^{\re}_G=\mathbb{Q}_{\perp}^*S^{\re}_{\omega}\hat{\blambda}_G^{\re}=\mathbb{S}_{\perp\sigma}^{\re}\hat{\blambda}_G^{\re} \implies \|	\mathbb{Q}_{\perp}^*\hat{\vec{g}}^{\re}_G\|_{H^{1/2}(\Gamma^{\re})}\leq \|\mathbb{S}_{\perp\sigma}^{\re}\|\|\hat{\blambda}_G^{\re}\|_{H^{-1/2}(\Gamma^{\re})}.
		\end{align*}
		The desired inequality then follows by combining the bound \eqref{eq:continuity_bound} and Lemma \ref{lemma:stability_bound_lambda_G}.
	\end{proof}
	\subsubsection{Proof of Proposition \ref{prop:data}}
	\label{sec:data_bound2}
	Since $\hat{u}^{\inc}\in H^3(B)$, by the Sobolev embedding \cite[Theorem 4.12, part 2]{adams}, $\hat{u}^{\inc}\in C^{1,\alpha}(\overline{B})$, $\alpha<1/2$. By Lemma \ref{lem:norml2},
	\begin{align*}%
		\|\hat{\vec{g}}^{\re} \|^2_{L^2(\Gamma^{\re})} &\leq C_0\re \sum_{k=1}^N  \|\hat{u}^{\operatorname{inc}} \|^2_{L^{\infty}(B^{\re}_k)} \leq C_0 \re N \|\hat{u}^{\operatorname{inc}} \|^2_{H^{2}(B)}.
	\end{align*}
	By Proposition \ref{prop:norm_equiv_perp} and Lemma \ref{lem:normh12},  
	\begin{align*}%
		| \mathbb{P}_{\perp}\hat {\vec{g}}^{\re} |^2_{H^{1/2}(\Gamma^{\re})} 
		\leq C_{\perp} \sum_{k=1}^N \re^{3} \| \nabla \hat{u}^{\operatorname{inc}} \|_{L^{\infty}(B^{\re}_k)} ^2\leq \re^{3} \times C_{\perp} N \| \hat{u}^{\operatorname{inc}}\|^2_{H^3(B)}.
	\end{align*}
	\subsection{Proof of Theorem \ref{th:conv}}
	\label{sec:fd_bounds}
	Using the splitting of the density error \eqref{eq:splitting_density_errorr}, we introduce the field error in the Fourier-Laplace domain via the single layer potential operator, see \eqref{eq:SL_potential},  
	\begin{align}%
		\label{eq:urex}
		\hat{e}^{\re}(\boldsymbol{x}):=	\hat{u}^{\re}(\boldsymbol{x}) - \hat{u}^{\re}_{\operatorname{app}}(\boldsymbol{x}) =  (\mathcal{S}^{\re}_{\omega} \hat{\vec{e}}^{\re}_{\sigma})(\bx) + (\mathcal{S}^{\re}_{\omega} \hat{\vec{e}}^{\re}_{\perp})(\bx), \quad \bx \in \Omega^{\re,c}.
	\end{align}
	Now that we have bounds on $\hat{\vec{e}}^{\re}_{\sigma}$, $\hat{\vec{e}}^{\re}_{\perp}$ derived in  \eqref{eq:bound_on_error_components_v1}, it remains to understand how the single-layer potential behaves on the corresponding subspaces. 
	\begin{lemma}
		\label{lem:slp}
		Let $K\subset B^c$ be a compact set. Then there exists a constant $C_K>0$, s.t. for all $\re \in (0,1)$ and $\omega \in \mathbb{C}^+$, $\mathcal{S}^{\re}_{\omega}: \, H^{-1/2}(\Gamma^{\re})\rightarrow H^{1}(\Omega^{\re,c})$ satisfies
		\begin{align}%
			\label{eq:s1}
			&\|\mathcal{S}^{\re}_{\omega}\mathbb{Q}_{\sigma}\|_{H^{-1/2}(\Gamma^{\re})\rightarrow L^{\infty}(K)}\leq C_KN^{1/2}\re,\\
			\label{eq:s2}
			&\|\mathcal{S}^{\re}_{\omega}\mathbb{Q}_{\perp}\|_{H^{-1/2}(\Gamma^{\re})\rightarrow L^{\infty}(K)}\leq C_KN^{1/2}\re^{3/2}(1+|\omega|).
		\end{align}
	\end{lemma}
	\begin{proof}
		Let $\bx\in K$. We fix $\vec{v}\in L^2(\Gamma^{\re})$ and next use the density argument. Since, in this case,  $\mathbb{Q}_{\sigma}\vec{v}\in L^2(\Gamma^{\re})$, and $|G_{\omega}(\vec{x}-\vec{y})|\leq (4\pi | \vec{x}-\vec{y} |)^{-1}$, it holds that 
		\begin{align*}%
			\left|	(\mathcal{S}^{\re}_{\omega}\mathbb{Q}_{\sigma}\vec{v})(\vec{x})\right|&=\left|(G_{\omega}(|\bx - \cdot|), \mathbb{Q}_{\sigma}\vec{v})_{L^2(\Gamma^{\re})}\right|  \leq \| G_{\omega} (|\bx - \cdot|) \|_{L^2(\Gamma^{\re})} \| \mathbb{Q}_{\sigma} \vec{v} \|_{L^2(\Gamma^{\re})} \\&\lesssim\frac{\re N^{1/2} }{4\pi\operatorname{dist}(\vec{x},\Gamma^{\re})}\|\mathbb{Q}_{\sigma}\vec{v}\|_{L^2(\Gamma^{\re})}
			,
		\end{align*}
		where we used $\|f\|_{L^2(\Gamma^{\re}_k)}\leq C_k \re \|f\|_{L^{\infty}(\Gamma^{\re}_k)}$, by Lemma \ref{lem:norml2}. Propositions \ref{prop:L^2_H_minus_half_bound} and  \ref{prop:norm_projectors_Q} yield the desired estimate \eqref{eq:s1}. \\
		To prove \eqref{eq:s2}, we proceed similarly. Since $\Im \mathbb{Q}_{\perp}=H^{-1/2}_{\perp}(\Gamma^{\re})$, cf. \eqref{eq:KerQsigma}, we have
		\begin{align*}%
			\left| (\mathcal{S}^{\re}_{\omega}\mathbb{Q}_{\perp}\vec{v})(\vec{x}) \right|
			&= \left|  \left\langle \mathbb{P}_{\perp} G_{\omega}(|\vec{x} - \cdot|), \mathbb{Q}_{\perp}\vec{v} \right\rangle_{H^{1/2}(\Gamma^{\re}), H^{-1/2}(\Gamma^{\re})} \right|
			\\&\lesssim |G_{\omega}(|\vec{x} - \cdot|)|_{H^{1/2}(\Gamma^{\re})}\|\mathbb{Q}_{\perp}\vec{v}\|_{H^{-1/2}(\Gamma^{\re})}, 
		\end{align*}
		cf. Proposition \ref{prop:norm_equiv_perp}.  
		With Lemma \ref{lem:normh12}, we have that
		\begin{align*}%
			|G_{\omega}(|\vec{x} - \cdot|)|_{H^{1/2}(\Gamma^{\re}_k)} &\leq C \varepsilon^{3/2}\|\nabla G_{\omega}(|\vec{x}-\cdot|)\|_{L^{\infty}(B^{\re}_k)} \\ 
			&\leq C\re^{3/2}(1+|\omega|)\max(1, \operatorname{dist}(\vec{x},\Gamma^{\re})^{-2}),
		\end{align*}
		where in the last inequality we made use of \eqref{eq:bound_grad_G}.
		This yields the final bound.
	\end{proof}
	Combining Lemma \ref{lem:slp} and \eqref{eq:bound_on_error_components_v1} into \eqref{eq:urex} yields the desired bound:
	\begin{align}
		\label{eq:bound_field_FD}
		\begin{split}
			|\hat{{e}}^{\re}(\boldsymbol{x})| &\leq C_K \re^{3}N^{3/2}(\underline{d}^{\re}_*)^{-6}
			\left(1+\re N(\underline{d}^{\re}_*)^{-4}\right)(1+|\omega|)^{10}\max(1, (\operatorname{Im}\omega)^{-9})\\
			&\times  \|\hat{u}^{\inc}\|_{H^3(B)},
		\end{split}
	\end{align}
	where $C_K>0$ depends on $K$ and the shape of the scatterers $\Gamma_j,\, j=1,\ldots, N$.
	\begin{remark} 
		Splitting the error into two parts allows to prove a higher convergence order w.r.t. $\re$  for the approximated scattered field, compared to the approach without the splitting. Indeed, using the worst of the two bounds for the error components yields $\| \hat{\vec{e}}^{\re}  \|_{H^{-1/2}(\Gamma^{\re})}= O(\re^{3/2})$, which implies that \( |(\mathcal{S}^{\re}_{\omega} \hat{\vec{e}}^{\re})(\bx)| = O(\re^{5/2}) \) as \( \re \to 0 \). On the other hand, when applying the above splitting into two terms, we obtain an improved estimate of order $O(\re^3)$.
	\end{remark} 
	Now we have necessary frequency-domain ingredients to prove the main statement of the paper for the model \eqref{eq:tdsys} in the time domain. 
	\section{Proof of Theorem \ref{theorem:convergence_of_models} for the model (\ref{eq:tdsys})}
	\label{sec:error_analysis}
	In the section that follows we aim at proving the following two results, which are refinement of the statement of Theorem \ref{theorem:convergence_of_models}, for the model \eqref{eq:tdsys}. 
	\subsection{Well-posedness}
	Since our well-posedness result is stated for the coefficients  $\lambda_{G,k}^{\re}$ of the expansion of the approximate density into the basis $\mathcal{V}^{\re}_G$ rather than the density itself, it will be convenient to introduce an auxiliary operator, which establishes correspondance between both. More precisely, for $\vv{v}\in \mathbb{C}^N$,   
	\begin{align}
		\label{eq:defIe}
		\mathcal{I}^{\re}: \, \mathbb{C}^N\rightarrow H^{-1/2}(\Gamma^{\re}), \quad \mathcal{I}^{\re}\vv{v}=\sum\limits_{j=1}^Nv_j\vec{\sigma}_j^{\re},
	\end{align}
	so that \eqref{eq:approximate_density} rewrites  $\vec{\lambda}^{\re}_G(t)=\mathcal{I}^{\re}\vv{\lambda^{\re}_{G}(t)}$. By Lemma \ref{lemma:Galerkin_space} $\mathcal{I}^{\re}: \,\mathbb{C}^N\rightarrow \mathcal{V}^{\re}_G$ is invertible. Moreover, since $ \|\vec{\sigma}^{\re}_j\|_{L^2(\Gamma^{\re})} = \|\vec{\sigma}^1_j\|_{L^2(\Gamma)}$, see \eqref{eq:sigmakeps}, there exist $c, C>0$, s.t.
	\begin{align}
		\label{eq:vecV0}
		\|\vv{v}\|_{\mathbb{C}^N} \lesssim	\|\mathcal{I}^{\re}\vv{v}\|_{L^2(\Gamma^{\re})} \lesssim  \|\vv{v}\|_{\mathbb{C}^N}.
	\end{align} 
	With Proposition \ref{prop:L^2_H_minus_half_bound}, the above rewrites, for $c, C>0$ independent of $\re$,
	\begin{align}
		\label{eq:vecV2}
		c\|\vv{v}\|_{\mathbb{C}^N} \leq	\|\mathcal{I}^{\re}\vv{v}\|_{H^{-1/2}(\Gamma^{\re})} \leq C\|\vv{v}\|_{\mathbb{C}^N}.
	\end{align}
	\begin{theorem}[Well-posedness of \eqref{eq:tdsys}]
		\label{theorem:well_posedness}
		Assume that $u^{\operatorname{inc}} \in TD(H^3(B))\cap C^{3}(\mathbb{R}_{\geq 0};H^3(B))$. Then the asymptotic model \eqref{eq:tdsys} admits a unique solution in $L'_+(\mathbb{R}^N)$. Moreover, the coefficients $\vv{\lambda}^{\re}_G\in \operatorname{TD}(\mathbb{R}^N)\cap C(\mathbb{R}_{\geq 0};\mathbb{R}^N)$ and satisfy the following stability bound:
		\begin{align*}%
			\|\vv{\lambda}^{\re}_G \|_{L^{\infty}(0,T;\mathbb{R}^N)} \leq \re^{-1} \times C(T, N, (\underline{d}^{\re}_{*})^{-1})\| u^{\operatorname{inc}} \|_{H^3(0,T; H^3(B))},\quad \text{ for all }T>0,
		\end{align*}
		where $C=C(T, N, (\underline{d}^{\re}_{*})^{-1})$ depends  polynomially on $N$, $(\underline{d}^{\re}_{*})^{-1}$, $T$.
	\end{theorem}
	
	\begin{proof}
		Remark that $\vec{g}^{\re}\in \operatorname{TD}(H^{1/2}(\Gamma^{\re}))\cap C^3(\mathbb{R}_{\geq 0}; H^{1/2}(\Gamma^{\re}))$, by assumption on $u^{\operatorname{inc}}$. To prove the existence and uniqueness, we apply the Fourier-Laplace transform to \eqref{eq:tdsys}, which yields \eqref{eq:operator_S_kl}. 
		By Theorem \ref{th:wp}, the problem \eqref{eq:operator_S_kl} admits a unique solution for all $\omega \in \mathbb{C}^{+}$. By assumption on $u^{\operatorname{inc}}$, $\omega \mapsto \hat{\vec{g}}^{\re}$ is analytic on $\mathbb{C}^+$. Moreover, the mapping $\omega \mapsto (S^{\re}_{\omega})^{-1}$ is analytic on $\mathbb{C}^{+}$, cf. \cite[Lemma 13, p. 592]{dunford_schwartz}. Therefore, the density $\omega\mapsto \hat{\vec{\lambda}}^{\re}_G(\omega)$ is analytic on $\mathbb{C}^+$.
		
		
		
		Moreover, since $\hat{\blambda}^{\re}$ satisfies polynomial stability bounds in the Laplace domain, cf. Theorem \ref{th:wp} (because $\hat{u}^{\inc}$ satisfies them as well), it is a Fourier-Laplace transform of a distribution in $\operatorname{TD}(H^{-1/2}(\Gamma^{\re}))$. See Section \ref{sec:ft} for more details.
		
		Since $\vv{\lambda}^{\re}_G=\left(\mathcal{I}^{\re}\right)^{-1}\blambda_G^{\re}$, the desired result follows for the coefficients of expansion of $\blambda^{\re}_G$, cf. \eqref{eq:defIe}. 
		
		To prove the time-domain stability result and a more precise regularity on $\vec{\lambda}^{\re}_G$ (and thus on $\vv{\lambda}^{\re}_G$), we invoke the Plancherel identity combined with the causality argument.   A complete proof can be found in \cite[Section 4.5.2]{MK} or \cite[Proof of Proposition 3.2.2]{sayas}. 
		Note that the Plancherel identity provides an $L^2$ bound on the scattered field in the time domain. To obtain a bound in the $L^{\infty}$ norm, we use the fact that for an $X$-valued function $t \mapsto v(t)$, of a  Banach space $X$,  with $v(0)=0$, it holds true that $\| v \|_{L^{\infty}(0,T; X)} \leq T^{1/2} \| \partial_t v \|_{L^2(0,T; X)}$. This implies that, in the final estimate, the data must possess one additional order of time regularity; in this case the solution belongs to $C([0,T];X)$ for all $T>0$, because it is $H^1(0,T;X)$ by the Plancherel identity, and because of the embedding $H^1(0,T; X)$ into $C([0, T]; X)$.	
	\end{proof}
	\subsection{Convergence}
	\begin{theorem}[Pointwise estimate on the field error]
		\label{theorem:point_wise_estimate}
		Let $K\subset B^c $ be a compact set, and $ u^{\operatorname{inc}} \in C^{11}(\mathbb{R}_{\geq 0}; H^3(B)) $. Then, for all $T>0$, 
		\begin{align*}%
			\|u^{\re}- u^{\re}_{G, \operatorname{\operatorname{app}}}\|_{L^{\infty}(0,T; L^{\infty}(K))} &\leq C(T, N, (\underline{d}^{\re}_{*})^{-1}) \re^3 \| u^{\operatorname{inc}} \|_{H^{11}(0, T;H^3(B))},
		\end{align*}
		where $C=C(T, N, (\underline{d}^{\re}_{*})^{-1})$ depends  polynomially on $N$, $(\underline{d}^{\re}_{*})^{-1}$, $T$.
	\end{theorem}
	\begin{proof}
		Follows by combining \eqref{eq:bound_field_FD} together with the Plancherel identity and the causality argument, a complete proof can be bound in \cite[Section 4.5.2]{MK} or \cite[Proof of Proposition 3.2.2]{sayas}.
	\end{proof}
	
	\section{Derivation and analysis of the simplified model (\ref{eq:tdsys2})}
	\label{sec:simplified_models}
	\subsection{Motivation}
	While the model (\ref{eq:tdsys}) provides an approximation of the far-field with an absolute error of $O(\re^3)$, its numerical realization appears to be quite involved. It can be done with the help of e.g. the boundary element method in space combined with an appropriate discretization scheme in time, e.g. Convolution Quadrature method \cite{lubich88_1, lubich88_2, lubich94}. In this context, such a numerical realization relies on 
	\begin{itemize}[leftmargin=*]
		\item meshing the domains $\Omega^{\varepsilon}$ and defining an appropriate ('simple') boundary element basis with functions supported on a few elements of the meshes;
		\item computation of the densities $\sigma_k^{\varepsilon}$ expressed in such a basis;
		\item evaluation of the double surface integrals in $K^{\varepsilon}_{G, \ell k}$;
		\item evaluation of the single surface integrals occurring in the expressions for the scattered field and in the right-hand side of the system \eqref{eq:final_asymtotic_model}.
	\end{itemize}
	The obvious bottleneck in the above is the evaluation of the entries of the double surface integrals in $K^{\varepsilon}_{G, \ell k}$. The goal of this section is to suggest an alternative model, which does not rely on such computations. 
	\subsection{Derivation and analysis of the model (\ref{eq:tdsys2})}
	\label{sec:simplified}
	To derive the model under consideration, let us re-state it and explain the intuition behind it. Given
	\begin{align}
		\label{eq:ckre2}
		c_k^{\varepsilon}:=\int_{\Gamma^{\varepsilon}_k}\sigma^{\varepsilon}_k(\vec{y})d\Gamma_{\vec{y}}, \quad \vec{p}_k^{\varepsilon}:=\int_{\Gamma^{\varepsilon}_k}(\vec{y}-\bc_k)\sigma_k^{\varepsilon}(\vec{y})d\Gamma_{\vec{y}}=\int_{\Gamma^{\varepsilon}}\vec{y}\sigma_k^{\varepsilon}(\vec{y})d\Gamma_{\vec{y}}-\vec{c}_k c_k^{\varepsilon},
	\end{align}
	we approximate the scattered field by
	\begin{align}
		\label{eq:scatap}
		\begin{split}
			{u}^{\varepsilon}_{s, \operatorname{app}}(\vec{x},t)&=\sum\limits_{k=1}^N\frac{\lambda^{\re}_{s, k}(t-|\vec{x}-\bc_k|)}{4\pi|\vec{x}-\bc_k|}\left(c_k^{\varepsilon}+\frac{(\vec{x}-\bc_k)}{4\pi|\vec{x}-\bc_k|^3}\cdot \vec{p}_k^{\varepsilon}\right)\\
			&+\sum\limits_{k=1}^N\frac{\partial_t\lambda^{\re}_{s, k}(t-|\vec{x}-\bc_k|)}{4\pi|\vec{x}-\bc_k|^2}(\vec{x}-\bc_k)\cdot\vec{p}_k^{\varepsilon},
		\end{split}
	\end{align}
	where the time-dependent functions $\lambda^{\re}_{s, k}: \, \mathbb{R}_{>0}\rightarrow \mathbb{R}$ solve 
	\begin{align}
		\label{eq:tdsys2p}
		& \sum_{k=1}^N (K^{\re}_{s, \ell k}*\lambda^{\re}_{s, k})(t)= - u^{\operatorname{inc}}(\bc_{\ell}, t)c_{\ell}^{\re}-\nabla u^{\operatorname{inc}}(\bc_{\ell}, t)\cdot\vec{p}_{\ell}^{\re}, \text{ with }\\
		\nonumber
		&K^{\re}_{s, \ell k }(t) = \mathcal{F}^{-1} (\mathbb{M}^{\re}_{s, \ell k})(t).
	\end{align}
	and the matrix $\mathbb{M}^{\re}_{s} $ is given by, with $\rho^{\re}_k:= c^{\re}_k/4 \pi$,  
	\begin{align*}%
		\mathbb{M}^{\re}_{s, \ell k}(\omega)=
		\left\{
		\begin{array}{ll}
			4 \pi i \omega (\rho^{\re}_k)^2 j_0(\omega \rho^{\re}_k) h^{(1)}_0(\omega \rho^{\re}_k), &k = \ell, \\
			4 \pi i \omega \rho^{\re}_k \rho^{\re}_{\ell} j_0(\omega \rho^{\re}_k) j_0(\omega \rho^{\re}_{\ell}) h^{(1)}_0(\omega |\bc_k - \bc_{\ell}| ), &k \neq \ell.
		\end{array}
		\right.
	\end{align*}
	To explain the idea behind the above, we resort to \cite{sini_frequency_domain}. As argued in this paper, to obtain the $O(\varepsilon^2)$-accurate model in the far-field, it is sufficient to know capacitance of the obstacles only. In other words, the fields scattered by Lipschitz obstacles and by the spheres of equivalent capacitance would be almost the same. 
	
	The first idea is thus the following: one could replace obstacles by balls of equivalent capacitance, and use the asymptotic Galerkin model \eqref{eq:tdsys} for balls, with the basis functions being now obstacle-wise constant (since $\sigma_j^{\re}=1$). Importantly, the corresponding model will preserve its stability properties, since it is still the Galerkin discretization of the single-layer boundary integral operator, see \cite{banjai_sauter}. Moreover, in this case all the Galerkin integrals are known analytically, cf. e.g. \cite{MK} for the circles. Later on, we will argue that this procedure indeed yields \eqref{eq:tdsys2p}, see Proposition \ref{prop:difM1}. 
	
	However, this may potentially lead to a loss of accuracy compared to \eqref{eq:tdsys}: indeed, the model \eqref{eq:tdsys} yields $O(\varepsilon^3)$ error in the far-field, rather than $O(\re^2)$ obtained in the frequency-domain \cite{sini_frequency_domain}. As we will see later, this loss of accuracy can be avoided by adding more terms in the approximation of the far-field and in the data, which requires a rather minor effort from the computational viewpoint. 
	
	\subsubsection{Derivation of the simplified model (\ref{eq:tdsys2}) from (\ref{eq:tdsys})}
	The model \eqref{eq:tdsys2} can be derived also starting with the model \eqref{eq:tdsys}. To see this, we compare the two models, starting with simplest differences: 
	\begin{itemize}[leftmargin=*]
		\item the data in the right-hand side of \eqref{eq:tdsys} is changed to
		\begin{align*}%
			-\int_{\Gamma^{\re}_{\ell}}u^{\operatorname{inc}}(\vec{x},t)\sigma^{\re}_{\ell}(\vec{x})d\Gamma_{\vec{x}}\rightarrow -\int_{\Gamma^{\re}_{\ell}}(u^{\operatorname{inc}}(\vec{c}_{\ell},t)+\nabla u^{\operatorname{inc}}(\vec{c}_{\ell}, t)\cdot (\vec{x}-\vec{c}_{\ell}))\sigma^{\re}_{\ell}(\vec{x})d\Gamma_{\vec{x}}.
		\end{align*}
		This replacement comes from the Taylor expansion of $u^{\operatorname{inc}}(\vec{x},t)$ around $\vec{x}=\vec{c}_{\ell}$.
		\item the convolution operator in the left-hand side of \eqref{eq:tdsys} is approximated via (here we write the corresponding expression in a formal manner \footnote{These expressions should be understood as $\langle K_{G,\ell k}, \varphi\rangle_{\mathcal{D}', \mathcal{D}}=\iint_{\Gamma_{\ell}\times\Gamma_k}\frac{\varphi(|\bx-\by|)}{4\pi|\bx-\by|}d\Gamma_{\bx}d\Gamma_{\by}$})
		\begin{align}
			\label{eq:K_s_lk}
			K^{\re}_{G, \ell k}(t) = \iint_{\Gamma_{\ell}^{\re}\times \Gamma_k^{\re}}&\frac{\delta(t-|\vec{x}-\vec{y}|)}{4\pi|\vec{x}-\vec{y}|}\sigma^{\re}_{\ell}(\vec{x})\sigma^{\re}_k(\vec{y})d\Gamma_{\vec{x}}d\Gamma_{\vec{y}} \nonumber \\	
			&\rightarrow K^{\re}_{s, \ell k}(t) = 4\pi  (c^{\re}_{\ell} c^{\re}_k)^{-1} \iint_{\widetilde{\Gamma}_{\ell}^{\re}\times \widetilde{\Gamma}_k^{\re}}\frac{\delta(t-|\vec{x}-\vec{y}|)}{|\vec{x}-\vec{y}|}d\Gamma_{\vec{x}}d\Gamma_{\vec{y}},
		\end{align} 
		where $\tilde{\Gamma}_{\ell}^{\re}$ is a sphere centered in $\vec{c}_{\ell}$ of a capacitance $c_{\ell}^{\re}$. The last equality is a subject of Proposition \ref{prop:difM1}. 
		
		Remark that while $K_G$ is nothing else but a semi-discretization of  ${S}^{\re}$ on $\Gamma^{\re}$ with the 'asymptotic' Galerkin space $\mathcal{V}_{G}^{\re}$ (see \eqref{eq:SL_BIE_TD}), $K_s$ is the semi-discretization of the single-layer $S^{\re}$ on the spheres centered at $\{\vec{c}_{\ell}\}_{\ell=1}^N$, with capacitances $c_{\ell}^{\re}$, with the asymptotic Galerkin space for the spheres (elements of $\mathbb{S}_0$).
		
		\item finally, the convolution solution operator in \eqref{eq:tdsys}, cf. \eqref{eq:ureg_sc}, is replaced by
		\begin{align*}	\int_{\Gamma_k^{\re}}\frac{\delta(t-|\vec{x}-\vec{y}|)}{4\pi|\vec{x}-\vec{y}|}\sigma^{\re}_k(\vec{x})d\Gamma_{\vec{y}}\rightarrow &\frac{\delta(t-|\vec{x}-\vec{c}_k|)}{4\pi|\vec{x}-\vec{c}_k|}\left(c_k^{\re}+\frac{(\vec{x}-\vec{c}_k)\cdot\vec{p}_k^{\re}}{4\pi|\vec{x}-\vec{c}_k|^3}\right) \\
			&-\frac{\frac{d}{dt} \delta(t-|\vec{x}-\vec{c}_k|)}{4\pi|\vec{x}-\vec{c}_k|^2}(\vec{x}-\vec{c}_k)\cdot \vec{p}_k^{\re}.
		\end{align*}
		The above comes by writing \eqref{eq:ureg_sc} in the Fourier-Laplace domain, with the use of
		\begin{align*}%
			(\mathcal{S}^{\re}_{\omega}\vec{\varphi})(\vec{x})=\sum\limits_{k=1}^N\int_{\Gamma^{\re}_k}\frac{\mathrm{e}^{i\omega|\bx-\by|}}{4\pi|\bx-\by|}\varphi_k(\vec{y})d\Gamma_{\vec{y}},
		\end{align*}
		expanding $\vec{y}\mapsto \frac{\mathrm{e}^{i\omega|\bx-\by|}}{4\pi|\bx-\by|}$ around $\vec{c}_k$ into the Taylor series, neglecting higher-order terms and going back to the time domain. 
	\end{itemize}
	
	\subsubsection{A roadmap of the analysis}
	Like in previous sections, we will derive our estimates in the frequency domain. We will concentrate our efforts on deriving an appropriate error estimate, while the corresponding well-posedness/stability estimate will appear as a by-product of the intermediate results, cf. e.g. Lemma \ref{lemma:stability_bound_lambda_G} which proves Theorem \ref{th:wp}, that, in turn, implies by standard arguments Theorem \ref{theorem:well_posedness}. As before, we will work with boundary densities, and next estimate the error of the scattered field. From the discussion in the previous paragraph, it may seem that this comparison of the boundary densities makes no sense, since, to derive the model, we replace obstacles by balls, and thus the densities are defined on different surfaces! Such a comparison will be done through the coefficients of the expansion of these densities into their respective bases. Recall that, for the 'standard' Galerkin Foldy-Lax model \eqref{eq:tdsys}, we define $\vec{\lambda}_G^{\re}(\vec{x},t)=\sum\limits_{k=1}^N\lambda_{G,k}^{\re}(t)\vec{\sigma}^{\re}_k(\vec{x})$, see \eqref{eq:approximate_density}; in a similar manner, we define a new density, associated to the model \eqref{eq:tdsys2p}, but living on $\Gamma^{\re}$, 
	\begin{align*}%
		\vec{\lambda}_s^{\re}(\vec{x},t):=\sum\limits_{k=1}^N\lambda_{s,k}^{\re}(t)\vec{\sigma}^{\re}_k(\vec{x}), \quad \vec{x}\in \Gamma^{\re}, \, t>0.
	\end{align*}
	Let us introduce the quantities to estimate. As discussed in the previous section, we will take into account the effects of different approximations separately.  
	For this we introduce the following intermediate quantities. First of all, for $(\bx, t) \in \Gamma^{\re} \times \mathbb{R}$,
	\begin{align}
		\label{eq:dens_pert}
		\vec{\lambda}^{\re,(1)}\in \operatorname{TD}(H^{-1/2}(\Gamma^{\re})), \quad \text{s.t.}\quad ({S}^{\re}*\vec{\lambda}^{\re,(1)})(\bx, t)=	\vec{g}^{\re,(1)}(\vec{x},t), \\
		\nonumber
		g^{\re,(1)}_{\ell}(\vec{x},t):=-u^{\operatorname{inc}}(\vec{c}_{\ell},t)-\nabla u^{\operatorname{inc}}(\vec{c}_{\ell},t)\cdot (\vec{x}-\vec{c}_{\ell}),\, (\vec{x}, t) \in \Gamma^{\re}_{\ell} \times \mathbb{R},
	\end{align}
	which is the solution to the original problem \eqref{eq:SL_BIE_TD} with a perturbed right-hand side. 
	Next, we approximate \eqref{eq:dens_pert} with the help of the asymptotic Galerkin method \eqref{eq:tdsys}, which yields another density:
	\begin{align}
		\label{eq:fdsys_corr}
			{\vec{\lambda}}^{\re,(1)}_G\in \operatorname{TD}(\mathcal{V}_{G}^{\re}) \quad \text{s.t.} \quad &\sum\limits_{k=1}^N (K^{\re}_{G, \ell k}*	{\lambda}^{\re,(1)}_{G, k})(t)=\int_{\Gamma^{\re}_{\ell}}{\vec{g}}^{\re,(1)}(\vec{x},t)\sigma^{\re}_{\ell}(\vec{x})d\Gamma_{\vec{x}},\\	
			\nonumber	&{\vec{\lambda}}^{\re,(1)}_G(\vec{x},t)=\sum\limits_{k=1}^N \lambda^{\re,(1)}_{G, k }(t)\vec{\sigma}_k^{\re}(\vec{x}).
	\end{align} 
	Then the error between the exact density $\boldsymbol{\lambda}^{\re}$ and the density $\boldsymbol{\lambda}^{\re}_s$ will be measured by means of comparing them to the above newly introduced quantities. In other words, we define three error terms $\vec{e}_j^{\re,(1)}$, which are simpler to estimate and whose sum allows to bound the sought error:
	\begin{align}
		\label{eq:three_errors}
		\begin{split}
			&\|\vec{\lambda}^{\re}(\cdot,t)-\boldsymbol{\lambda}^{\re}_s(\cdot,t)\|_{H^{-1/2}(\Gamma^{\re})}\leq 	\sum\limits_{j=1}^3\|\vec{e}_j^{\re,(1)}(\cdot,t)\|_{H^{-1/2}(\Gamma^{\re})},\\
			&\vec{e}_1^{\re,(1)}=\boldsymbol{\lambda}^{\re}-\boldsymbol{\lambda}^{\re,(1)},\quad
			\vec{e}_2^{\re,(1)}=\boldsymbol{\lambda}^{\re,(1)}-\boldsymbol{\lambda}^{\re,(1)}_G,\quad
			\vec{e}_3^{\re,(1)}=\boldsymbol{\lambda}^{\re,(1)}_G-\boldsymbol{\lambda}^{\re}_s.
		\end{split}
	\end{align}
	Let us remark that it is the error $\vec{e}_3^{\re,(1)}$ that carries an information about the change in the operators $K$ when passing from (\ref{eq:tdsys}) to (\ref{eq:tdsys2}). 
	
	As soon as the density error is controlled, we will show how they allow us to bound the error in the scattered field, see Section \ref{sec:recovering}. As discussed in the beginning of this section, we will concentrate on the estimates in the frequency domain, which we further translate into the time domain. 
	
	Since most of our analysis relies on the existing estimates, in most of the arguments we will no longer keep track of exact powers of $|\omega|$ and $\Im \omega$, but rather hide them in the notation, by remarking that the dependence is polynomial in $|\omega|$ and inverse-polynomial in $\min(1,\Im \omega)$ where needed. Same refers to the dependence on the quantities $d_*^{\re}$ and $N$.  More precisely, we will write that 
	\begin{align*}%
		a\leq C_{\operatorname{geom}}(N, \underline{d}_*^{\re}) p(|\omega|,\Im \omega)b,
	\end{align*}
	if there exists $m_1,m_2,n_1,n_2\geq 0$ and a constant $C>0$ (independent of $\re\in (0,1]$), s.t. for all $\omega\in \mathbb{C}^+$, 
	\begin{align*}%
		a\leq C N^{n_1}(\underline{d}_*^{\re})^{-n_2} (1+|\omega|)^{m_1}\max(1, (\Im \omega)^{-m_2})b.
	\end{align*}
	\subsubsection{Error stemming from the simplification of the right-hand side $\hat{\vec{e}}_1^{\re, (1)}$}
	\label{sec:ere1}
	Remark that the error $\hat{\vec{e}}_1^{\re, (1)}$ is the error of the replacing the right-hand side in the original problem \eqref{eq:SL_BIE_TD} by a simplified quantity, and thus
	\begin{align}
		\label{eq:crv}
		{S}^{\varepsilon}_{\omega}\hat{\vec{e}}_1^{\re,(1)}=	{S}^{\varepsilon}_{\omega}(\hat{{\blambda}}^{\varepsilon}-\hat{\blambda}^{\varepsilon,(1)})=\hat{{\vec{g}}}^{\re}-\hat{\vec{g}}^{\re,(1)}.
	\end{align}
	To analyze the corresponding errors, we proceed similarly to how it was done in Section \ref{sec:analysis_freq_domain}, by decomposing the error into different components $\hat{\vec{e}}^{\re,(1)}_{1, \sigma}=\mathbb{Q}_{\sigma}\hat{\vec{e}}^{\re,(1)}_1\in \mathcal{V}_G^{\re}$ and $\hat{\vec{e}}^{\re,(1)}_{1, \perp}=\mathbb{Q}_{\perp}\hat{\vec{e}}^{\re,(1)}_1\in H^{-1/2}_{\perp}(\Gamma^{\re})$. Both terms are bounded in a rather direct manner. Testing \eqref{eq:crv} with $\hat{\vec{e}}^{\re,(1)}_{1}$ and using the coercivity bound of Theorem \ref{theorem:coercivity_bound_principal} yields
	\begin{align*}%
		c_S(\varepsilon\|\hat{\vec{e}}^{\re,(1)}_{1,\sigma}\|^2_{H^{-1/2}(\Gamma^{\re})}+\|\hat{\vec{e}}^{\re,(1)}_{1,\perp}\|^2_{H^{-1/2}(\Gamma^{\re})})\lesssim |\langle \,	\hat{{\vec{g}}}^{\re,(1)}-\hat{\vec{g}}^{\re}, \hat{\vec{e}}^{\re,(1)}_{1,\sigma}\rangle| + |\langle \,	{\hat{\vec{g}}}^{\re,(1)}-\hat{\vec{g}}^{\re}, \hat{\vec{e}}^{\re,(1)}_{1,\perp}\rangle|.
	\end{align*}
	Using Propositions \ref{prop:L^2_H_minus_half_bound} and \ref{prop:norm_equiv_perp}, we conclude that
	\begin{align*}%
		c_S	\left(\varepsilon\|\hat{\vec{e}}^{\re,(1)}_{1,\sigma}\|^2_{L^2(\Gamma^{\re})}+\|\hat{\vec{e}}^{\re,(1)}_{1,\perp}\|^2_{H^{-1/2}(\Gamma^{\re})}\right)&\lesssim \|\hat{\vec{g}}^{\re,(1)}-\hat{\vec{g}}^{\re}\|_{L^2(\Gamma^{\re})}\|\hat{\vec{e}}^{\re,(1)}_{1,\sigma}\|_{L^2(\Gamma^{\re})}\\
		&+|\hat{\vec{g}}^{\re,(1)}-\hat{\vec{g}}^{\re}|_{H^{1/2}(\Gamma^{\re})}\|\hat{\vec{e}}^{\re,(1)}_{1,\perp}\|_{H^{-1/2}(\Gamma^{\re})}.
	\end{align*}
	The inequality
	$|a_1b_1+a_2b_2|^2\leq (|a_1|^2/\re+|a_2|^2)(\re|b_1|^2+|b_2|^2)$ yields
	\begin{align*}%
		c_S(\varepsilon\|\hat{\vec{e}}^{\re,(1)}_{1,\sigma}\|^2_{L^2(\Gamma^{\re})}+\| \hat{\vec{e}}^{\re,(1)}_{1,\perp}
		&\|^2_{H^{-1/2}(\Gamma^{\re})})
		\lesssim
		(\varepsilon\|\hat{\vec{e}}^{\re,(1)}_{1,\sigma}\|^2_{L^2(\Gamma^{\re})}+\|\hat{\vec{e}}^{\re,(1)}_{1,\perp}\|^2_{H^{-1/2}(\Gamma^{\re})})^{1/2} \\
		&\times (\varepsilon^{-1}\|\hat{\vec{g}}^{\re,(1)}-\hat{\vec{g}}^{\re}\|_{L^2(\Gamma^{\re})}^2+|\hat{\vec{g}}^{\re,(1)}-\hat{\vec{g}}^{\re}|^2_{H^{1/2}(\Gamma^{\re})})^{1/2},
	\end{align*}
	so that, with the definition of $c_S$ from Theorem \ref{theorem:coercivity_bound_principal}, it holds that
	\begin{align}
		\label{eq:QperpE}
		&\| \hat{\vec{e}}^{\re,(1)}_{1, \perp}\|_{H^{-1/2}(\Gamma^{\re})}
		\leq C_{\operatorname{geom}}(N, \underline{d}_*^{\re})p(|\omega|,\Im\omega) \trinorm{\hat{\vec{g}}^{\re,(1)}-\hat{\vec{g}}^{\re}}_{-1/2}, \\
		\label{eq:ere1}
		&\| \hat{\vec{e}}^{\re,(1)}_{1, \sigma}\|_{L^2(\Gamma^{\re})} \leq \varepsilon^{-1/2}C_{\operatorname{geom}}(N, \underline{d}_*^{\re})p(|\omega|,\Im\omega)\trinorm{\hat{\vec{g}}^{\re,(1)}-\hat{\vec{g}}^{\re}}_{-1/2},
	\end{align}
	with \begin{align}
		\label{eq:trinorm}
		\trinorm{ \vec{\psi} }_{\alpha} := \re^{\alpha} \| \vec{\psi} \|_{L^2(\Gamma^{\re})} + | \vec{\psi} |_{H^{1/2}(\Gamma^{\re})}.
	\end{align}
	\subsubsection{Galerkin Foldy-Lax error $\hat{\vec{e}}_2^{\re, (1)}$}
	\label{sec:ere2}
	For the error term $\hat{\vec{e}}_2^{\re, (1)}$ we apply the same analysis as in Section \ref{sec:analysis_freq_domain}. The only difference in the estimates arises from the replacement of $\hat{\vec{g}}^{\re}$ by $\hat{\vec{g}}^{\re, (1)}$. The inequality \eqref{eq:bound_on_error_components_v0} rewrites, see \eqref{eq:trinorm}: 
	\begin{align*}%
		&\| \hat{\boldsymbol{e}}_{2, \perp}^{\re,(1)}\|_{H^{-1/2}(\Gamma^{\re})}\leq C_{\operatorname{geom},\perp,2}(N, \underline{d}_*^{\re})p_{\perp,2}(|\omega|,\Im \omega) \trinorm{\hat{\vec{g}}^{\re,(1)}}_{3/2},\\
		&\| \hat{\boldsymbol{e}}_{2, \sigma}^{\re,(1)}\|_{H^{-1/2}(\Gamma^{\re})}\leq \varepsilon^{3/2}\times	
		C_{\operatorname{geom},\sigma,2}(N, \underline{d}_*^{\re})p_{\sigma,2}(|\omega|,\Im \omega) \trinorm{\hat{\vec{g}}^{\re,(1)}}_{3/2}.
	\end{align*}
	\subsubsection{Error $\hat{\vec{e}}_3^{\re, (1)}$ stemming from replacing the operator}
	\label{sec:ere3}
	Written in an algebraic manner, the Galerkin Foldy-Lax model \eqref{eq:fdsys_corr} with the right-hand side $\hat{{\boldsymbol{g}}}^{\re,(1)}$ in the frequency domain reads, cf. \eqref{eq:as_model_fd} and  \eqref{eq:approximate_density}: find $\omega\mapsto \vv{\lambda}^{\re,(1)}_G(\omega)\in\mathbb{C}^N$, s.t.  
	\begin{equation}
		\label{eq:problem_M_G}
		\begin{aligned}
			&\mathbb{M}^{\varepsilon}(\omega) \vv{\lambda}^{\re, (1)}_G (\omega)=\vv{q}^{\varepsilon,(1)}(\omega) \quad \text{with} \quad q^{\varepsilon,(1)}_{\ell }(\omega)= \int_{\Gamma^{\varepsilon}_{\ell} }\hat{{g}}^{\varepsilon,(1)}_{\ell} (\bx, \omega){\sigma}_{\ell} ^{\varepsilon}(\bx)d\Gamma_{\bx},\\
			&\mathbb{M}^{\varepsilon}_{\ell k}(\omega)=\iint_{\Gamma^{\varepsilon}_{\ell}\times \Gamma^{\varepsilon}_k} G_{\omega}(|\bx-\by|)\sigma^{\varepsilon}_{\ell} (\bx)\sigma^{\varepsilon}_k(\by)d\Gamma_{\bx}d\Gamma_{\by}.
		\end{aligned} 
	\end{equation}
	We use the notation  $\vv{\lambda}_G^{\re,(1)}$ to underline that we are looking for the coefficients in the expansion ${\hat{\blambda}}_G^{\re,(1)}=\sum\limits_{k=1}^N \lambda_{G,k}^{\re,(1)}\vec{\sigma}_k^{\re}$, rather than densities defined on $\Gamma^{\re}$.  In order to keep the notation less cumbersome, we abuse it by omitting the $\hat{.}$ in $\vv{\lambda}$, keeping in mind that in this and following sections we work in the Fourier-Laplace domain. \\
	Our next step is to replace the matrix in the asymptotic Galerkin model by its counterpart for spherical particles. Proceeding in a similar manner as above, we rewrite the simplified model \eqref{eq:tdsys2} in the frequency domain. Find $\omega \mapsto \vv{\lambda}^{\re}_s(\omega) \in \mathbb{C}^N$, s.t. 
	\begin{equation}
		\label{eq:problem_M_s}
		\begin{aligned}
			&\mathbb{M}^{\re}_s(\omega) \vv{\lambda}^{\re}_s(\omega) = \vv{q}^{\re, (1)}(\omega) \quad \text{with} \quad q^{\re, (1)}_{\ell}(\omega) = \int_{\Gamma^{\varepsilon}_{\ell} }\hat{g}^{\varepsilon,(1)}_{\ell} (\bx, \omega){\sigma}_{\ell} ^{\varepsilon}(\bx)d\Gamma_{\bx},\\
			&\mathbb{M}^{\re}_{s, \ell k}(\omega) = 16  \pi^2 (c^{\re}_{\ell} c^{\re}_k)^{-1} \iint_{\tilde{\Gamma}^{\re}_{\ell} \times \tilde{\Gamma}^{\re}_k} G_{\omega}(|\bx-\by|) d\Gamma_{\bx}d\Gamma_{\by}, 
		\end{aligned}
	\end{equation}
	where $\tilde{\Gamma}^{\re}_k$ denotes a sphere centered at $\bc_k$ with the radius $ (4 \pi )^{-1} c^{\re}_k$.
	
	To state the key result of this section, let us define a new geometry. Let  $\widetilde{\Omega}^{\varepsilon}_k=B(\vec{c}_k, \rho^{\re}_k)$, $\rho^{\re}_k:=(4\pi)^{-1}c_k^{\varepsilon}=(4\pi)^{-1}c_k^1\re$ (see \eqref{eq:sigmakeps}). 
	\begin{proposition}
		\label{prop:difM1_geometry}
		The following holds true:  
		\begin{enumerate}
			\item the equilibrium densities $\widetilde{\sigma}^{\re}_k$ for $\widetilde{\Omega}_k^{\re}$ are constant and equal to $(\rho_k^{\re})^{-1}$;
			\item the capacitances of the obstacles $\Omega^{\varepsilon}_k$ and $\widetilde{\Omega}^{\varepsilon}_k$ are equal;
			\item if $(\Omega^{\re})_{0<\re\leq1}$ is admissible (see Definition \ref{def:admissible}), so is $(\widetilde{\Omega}^{\re})_{0<\re\leq1}$, and $$\operatorname{min}\operatorname{dist}(\widetilde{\Gamma}_j^{\re}, \widetilde{\Gamma}_k^{\re})\geq 
			\operatorname{min}\operatorname{dist}({\Gamma}_j^{\re}, {\Gamma}_k^{\re})=d^{\re}_*, \quad \text{ for all }0<\re\leq 1.$$
		\end{enumerate}
		
	\end{proposition}
	\begin{proof}
		The first two statements follow immediately once we recall that the capacitance of the ball of radius $r$ is $4\pi r$; this latter result follows from the expression 
		\begin{align*}%
			\int_{\partial B(0,r)}\frac{1}{4\pi | \bx-\by |}d\Gamma_{\bx}=r, \quad \vec{x}\in \partial B(0,r)
		\end{align*}
		and thus the associated equilibrium density equals $\sigma(\vec{x})=r^{-1}$. 
		
		To prove the third statement, by \cite[Exercise 8.12, p. 272]{mclean}, $c_k^{\re}\leq 4\pi \operatorname{diam}\Gamma_k^{\re}$, thus $\rho_k^{\re}\leq \operatorname{diam}\Gamma_k^{\re}$, which shows that $B(\vec{c}_k, \rho_k^{\re})\subset B(\vec{c}_k, \re r_k)$, cf. \eqref{eq:omega_ball}. Then $(\widetilde{\Omega}^{\re})_{0<\re\leq 1}$ is admissible since $(\Omega^{\re})_{0<\re\leq 1}$ is, see the discussion after Definition \ref{def:admissible}. 
	\end{proof}
	The next proposition shows that, on one hand $K^{\re}_{s}$ is essentially $K^{\re}_G$ written for $\widetilde{\Omega}^{\re}$, and, on the other hand, that these operators are close as $\re\rightarrow 0$ in a certain sense. 
	\begin{proposition}
		\label{prop:difM1}
		For $K^{\re}_{s,\ell k}$ defined in \eqref{eq:tdsys2}, it holds that 
		\begin{align}
			\label{eq:fm}
			\mathcal{F}\left(\frac{1}{\rho_{\ell}^{\re}\rho_k^{\re}}\int_{\widetilde{\Gamma}_{\ell}^{\re}}\int_{\widetilde{\Gamma}_{k}^{\re}}\frac{\delta(t-|\vec{x}-\vec{y}|)}{4\pi|\vec{x}-\vec{y}|}d\Gamma_{\vec{x}}d\Gamma_{\vec{y}}\right)=\mathcal{F}\left(K^{\re}_{s,\ell k}\right)(\omega)= \mathbb{M}_{s, \ell k}^{\re}(\omega)
		\end{align}
		with
		\begin{align}
			\label{eq:M_simplified}
			\mathbb{M}^{\re}_{s, \ell k}(\omega) = 
			\left\{
			\begin{array}{ll}
				4 \pi i \omega (\rho^{\re}_k)^2 j_0(\omega \rho^{\re}_k) h^{(1)}_0(\omega \rho^{\re}_k), &k = \ell, \\
				4 \pi i \omega \rho^{\re}_k \rho^{\re}_{\ell} j_0(\omega \rho^{\re}_k) j_0(\omega \rho^{\re}_{\ell}) h^{(1)}_0(\omega |\bc_k - \bc_{\ell}| ), &k \neq \ell,
			\end{array}
			\right.
		\end{align}
		where $j_0$, $h^{(1)}_0$ denote the spherical Bessel and Hankel functions of the first kind. 
		
		For all  $0<\re<1$, all $\omega\in \mathbb{C}^+$, the matrix $\mathbb{M}^{\re}_s (\omega)$ is invertible, and the following two estimates hold true: 
		\begin{align}
			\label{eq:widm}
			&\|(\mathbb{M}^{\re}_s)^{-1}\|_{F} \leq \varepsilon^{-1} C_{s, \operatorname{geom}}( N,{\underline{d}}^{\re}_*)p_s(|\omega|,\Im \omega),\\
			\label{eq:widm2}
			&\|\left(\mathbb{M}^{\varepsilon}\right)^{-1}-(\mathbb{M}^{\re}_s)^{-1}\|_{F} \leq \varepsilon {C}_{\operatorname{diff}, \operatorname{geom}}( N,\underline{{d}}_*^{\re}){p}_{\operatorname{diff}}(|\omega|,\Im \omega),
		\end{align}
		where $\|.\|_F$ is the Frobenius matrix norm. 
	\end{proposition}
	\begin{proof}
		\mypar{Proof of \eqref{eq:fm}-\eqref{eq:M_simplified}.}
		We start by rewriting 
		\begin{align}
			\nonumber
			\mathcal{F}\left(\frac{1}{\rho_{\ell}^{\re}\rho_k^{\re}}\int_{\widetilde{\Gamma}_{\ell}^{\re}}\int_{\widetilde{\Gamma}_{k}^{\re}}\frac{\delta(t-|\vec{x}-\vec{y}|)}{4\pi|\vec{x}-\vec{y}|}d\Gamma_{\vec{x}}d\Gamma_{\vec{y}}\right)&= (\rho^{\re}_{\ell} \rho^{\re}_{k})^{-1} \iint_{\tilde{\Gamma}^{\re}_{\ell} \times \tilde{\Gamma}^{\re}_k} G_{\omega} (|\bx - \by|) d\Gamma_{\bx} d\Gamma_{\by} \\
			\label{eq:M_s_int_form}	
			&= (\rho^{\re}_{\ell} \rho^{\re}_{k})^{-1} (\widetilde{S}^{\re, \ell k}_{\omega} 1, 1)_{L^2(\tilde{\Gamma}^{\re}_k)},
		\end{align}
		where we use the $\widetilde{.}$ to indicate that the corresponding operator is defined on $\widetilde{\Gamma}^{\re}$ rather than $\Gamma^{\re}$:   $$(\widetilde{S}^{\re}_{\omega}\vec{\varphi})(\vec{x})=\sum\limits_{k=1}^N\int_{\widetilde{\Gamma}_k^{\re}}G_{\omega}(|\bx-\by|)\varphi_k(\vec{y})d\Gamma_{\by}, \quad \vec{x}\in \widetilde{\Gamma}^{\re}.$$  
		
		\mypar{Case $\ell = k$.}
		For the unit sphere, the spherical harmonics $\{ Y^m_n: n \in \mathbb{N}, ~m = -n,\ldots, n \}$ are eigenfunctions of the single layer boundary operator for the Helmholtz problem, i.e.
		\begin{align*}%
			\widetilde{S}^{\re, kk}_{\omega} Y^m_n = \lambda_n Y^m_n \quad \text{with}\quad \lambda_n = i \omega (\rho^{\re}_k)^2 j_n(\omega \rho^{\re}_k) h^{(1)}_n(\omega \rho^{\re}_k),
		\end{align*} 
		where $j_n$ and $h^{(1)}_n$ are the spherical Bessel and Hankel functions of the first find, receptively, see \cite[p. 437, Section 10]{abramowitz} for details. Using this spectral representation and the orthogonality of spherical harmonics over the unit sphere, only $Y^0_0 = (4\pi)^{-1/2}$ contributes when integrating against a constant. Thus, inserting the above identity into \eqref{eq:M_s_int_form} yields the equiality of \eqref{eq:M_s_int_form} for $\ell=k$ with $\mathbb{M}^{\re}_{s, kk}$.
		
		\mypar{Case $\ell \neq k$.}
		We make use of the Addition Theorem (cf. \cite[\S 10.1, p. 440, Formulas 10.1.45 and 10.1.46]{abramowitz})
		
		\begin{align}
			\label{eq:add_theorem}
			\frac{ e^{i\omega| \boldsymbol{x} - \boldsymbol{y} |}}{4 \pi | \boldsymbol{x} - \boldsymbol{y} |} = \frac{i \omega}{4 \pi} \sum_{n=0}^{+\infty} (2n+1) j_n(\omega \rho^{\re}_{\ell}) h^{(1)}_n (\omega |\by-\bc_{\ell}|) P_n (\widehat{\bx - \bc_{\ell}} \cdot \widehat{\by - \bc_{\ell}}),
		\end{align}
		for $\bx \in \tilde{\Gamma}^{\re}_{\ell}$, $\by \in \tilde{\Gamma}^{\re}_k$, with $ \rho^{\re}_{\ell} = |\bx - \bc_{\ell}|$. Here $P_n$, $n\geq 0$ are the Legendre polynomials. 
		Using the Addition Theorem \eqref{eq:add_theorem}, we obtain 
		\begin{align}
			\label{eq:S_lk_1}
			(\tilde{S}^{\re, \ell k}_{\omega} 1)(\by) =  \int_{\tilde{\Gamma}^{\re}_{\ell}}  G_{\omega}(|\bx-\by|) d\Gamma_{\bx} = \frac{i \omega}{4 \pi} \sum_{m=0}^{\infty} (2m+1) j_m(\omega \rho^{\re}_{\ell}) h^{(1)}_m (\omega|\by-\bc_{\ell}|) \int_{\tilde{\Gamma}^{\re}_{\ell}} P_m(\widehat{\bx - \bc_{\ell}}\cdot \widehat{\by - \bc_{\ell}}) d\Gamma_{\bx}, \quad \by \in \tilde{\Gamma}^{\re}_k
		\end{align}
		In order to compute the integral in the right-hand side of the above identity, we employ the Funk-Hecke formula, see \cite{hecke}: for $f \in L^2(-1,1)$, it holds true the following identity
		\begin{align}
			\label{eq:funk_hecke}
			\int_{\mathbb{S}^2} f(\bx \cdot \by) Y^m_n(\by) d\Gamma_{\by} = 2\pi \int_{-1}^1 f(t) P_{n}(t) dt Y^m_n(\bx),\quad \bx \in \mathbb{S}^2, 
		\end{align}
		for all $n \in \mathbb{N}$ and $m=-n,\ldots, n$ where $\mathbb{S}^2$ denotes the unit sphere. 
		Then, changing the variables in the integral in \eqref{eq:S_lk_1}, namely, $\vec{\tau} = \widehat{\bx - \bc_{\ell}} = (\rho^{\re}_{\ell})^{-1}(\bx - \bc_{\ell})$,  and then applying the Funk-Hecke formula \eqref{eq:funk_hecke} yields 
		\begin{align}
			\label{eq:application_funk_hecke}
			\int_{\tilde{\Gamma}^{\re}_{\ell}} P_m(\widehat{\bx - \bc_{\ell}}\cdot \widehat{\by - \bc_{\ell}}) d\Gamma_{\bx} = ( \rho^{\re}_{\ell})^2 \int_{\mathbb{S}^2} P_m(\vec{\tau} \cdot \widehat{\by - \bc_{\ell}}) d\Gamma_{\vec{\tau}}  = 2\pi ( \rho^{\re}_{\ell})^2 \int_{-1}^1 P_m(t) dt = 4 \pi ( \rho^{\re}_{\ell})^2, 
		\end{align}
		where in the last equality we made use of the fact that the Legendre polynomials are pairwise orthogonal with respect to the standard $L^2(-1,1)$ inner product. Thus, we have
		\begin{align}
			\label{eq:S_lk_h_0}
			(\tilde{S}^{\re, \ell k}_{\omega} 1)(\by) =  i \omega ( \rho^{\re}_{\ell})^2 j_0 (\omega \rho^{\re}_{\ell}) h^{(1)}_0 (\omega|\by-\bc_{\ell}|),\quad \by \in \tilde{\Gamma}^{\re}_k.
		\end{align} 
		Next, using $h^{(1)}_{0} (z) = e^{iz}/z$ (see \cite[p. 438, 10.1.11-10.1.12]{abramowitz}), we apply the Addition Theorem again to $h^{(1)}_0(\omega|\by-\bc_{\ell}|)$:
		\begin{align*}%
			h^{(1)}_{0} (\omega| \by -\bc_{\ell}|) = \sum_{m=0}^{+\infty} (2m+1) j_m(\omega \rho^{\re}_k) h^{(1)}_m(\omega|\bc_k - \bc_{\ell}|) P_{m} (\widehat{\by-\bc_k} \cdot \widehat{\bc_k-\bc_{\ell}}). 
		\end{align*}
		Inserting the above identity into \eqref{eq:S_lk_h_0} and then testing it with a constant, it follows 
		\begin{align*}%
			(\tilde{S}^{\re, \ell k}_{\omega} 1, 1)_{L^2(\tilde{\Gamma}^{\re}_k)} =  i \omega ( \rho^{\re}_{\ell})^2  j_0 (\omega \rho^{\re}_{\ell}) \sum_{m=0}^{+\infty} (2m+1) j_m(\omega \rho^{\re}_k) h^{(1)}_0 (|\bc_{\ell}-\bc_k|) \int_{\tilde{\Gamma}^{\re}_k} P_{m} (\widehat{\by-\bc_k} \cdot \widehat{\bc_k-\bc_{\ell}}) d\Gamma_{\by}.
		\end{align*}
		Computing the above integral in a similar manner like in \eqref{eq:application_funk_hecke} leads
		\begin{align*}%
			(\rho^{\re}_{\ell} \rho^{\re}_{k})^{-1} (\tilde{S}^{\re, \ell k}_{\omega} 1, 1)_{L^2(\tilde{\Gamma}^{\re}_k)} = 4 \pi  i \omega \rho^{\re}_{\ell} \rho^{\re}_{k} j_0 (\omega \rho^{\re}_{\ell})  j_0(\omega \rho^{\re}_k) h^{(1)}_0 (|\bc_{\ell}-\bc_k|).
		\end{align*}
		Let us now prove \eqref{eq:widm}. 
		In virtue of Proposition \ref{prop:difM1_geometry}, the equilibrium density $\widetilde{\vec{\sigma}}^{\re}_k$ considered on $\widetilde{\Gamma}^{\re}$ equals $(\rho^{\re}_k)^{-1} \vec{1}^{\re}_k$, $k=1,\ldots,N$. Let us denote by $\widetilde{\mathcal{I}}^{\re}$ the operator $\mathcal{I}^{\re}$ considered on $\widetilde{\Gamma}^{\re}$, i.e. 	
		\begin{align*}%
			\widetilde{\mathcal{I}}^{\re}: \, \mathbb{C}^N\rightarrow H^{-1/2}(\widetilde{\Gamma}^{\re}), \quad 	\widetilde{\mathcal{I}}^{\re}\vv{v}=\sum\limits_{k=1}^N \frac{v_{k}}{\rho^{\re}_{k}} \vec{1}^{\re}_{k}.
		\end{align*}
		Let $\vv{v}\in \mathbb{C}^N$. Then, by definition of $\mathbb{M}_s^{\re}(\omega)$, it holds that 
		\begin{align*}%
			(\vv{v}, \,\mathbb{M}^{\varepsilon}_{s}(\omega)\vv{v})_{\mathbb{C}^N}=\langle \widetilde{\mathcal{I}}^{\re}\vv{v}, \tilde{S}^{\re}_{\omega} \widetilde{\mathcal{I}}^{\re}\vv{v}\rangle_{-1/2, 1/2, \tilde{\Gamma}^{\re}},
		\end{align*}
		where $\tilde{S}^{\re}_{\omega}$ is the single-layer boundary integral operator refined on $\tilde{\Gamma}^{\re}$, cf. \eqref{eq:fm}. 
		Next, we use the above identity, Theorem \ref{theorem:coercivity_bound_principal} and statement 3 of Proposition \ref{prop:difM1_geometry} to conclude that 
		\begin{align*}%
			\left|	 (\vv{v}, \,\mathbb{M}^{\re}_s(\omega)\vv{v})_{\mathbb{C}^N}\right|&\geq C \re (\underline{{d}}^{\re}_*)^2 (1+|\omega|)^{-2} \min(1, (\Im \omega)^3) \|	\widetilde{\mathcal{I}}^{\re}\vv{v}\|^2_{H^{-1/2}(\tilde{\Gamma}^{\re})}\\
			&\geq C \re (\underline{{d}}^{\re}_*)^2 (1+|\omega|)^{-2} \min(1, (\Im \omega)^3)\|	\vv{v}\|^2_{\mathbb{C}^N},
		\end{align*}
		where the last inequality follows from \eqref{eq:vecV2}. From the above and the equivalence of the induced $2-$matrix and Frobenius norm ($\|(\mathbb{M}^{\re}_s)^{-1}\|_F\leq	N^{1/2}\|(\mathbb{M}^{\re}_s)^{-1}\|_2$), we obtain  \eqref{eq:widm}:
		\begin{align}
			\label{eq:invM}
			\|(\mathbb{M}^{\re}_s)^{-1}\|_F\leq C^{-1} N^{1/2}\re^{-1} (d^{\re}_*)^{-2} (1+|\omega|)^{2} \max(1, (\Im \omega)^{-3}).
		\end{align} 
		\textit{Proof of \eqref{eq:widm2}. }We will prove this result by comparing the matrices $\mathbb{M}^{\re}$ and $\mathbb{M}_s^{\re}$. Recall that $\mathbb{M}^{\re}$ is defined in \eqref{eq:problem_M_G}. 
		
		\mypar{Expansion of the entries of $\mathbb{M}^{\varepsilon}$ as $\varepsilon\rightarrow 0$. }
		Let us expand the entries of the matrix $\mathbb{M}^{\varepsilon}$ in the powers of $\varepsilon$. We start with the diagonal entries, which we develop like in Section \ref{sec:continuity_bound}, recall \eqref{eq:Edefvol},
		\begin{align*}%
			\mathbb{M}^{\varepsilon}_{k k}(\omega)&=\langle S^{\varepsilon,k k}_0\sigma_k^{\varepsilon}+\frac{i\omega}{4\pi}\int_{\Gamma^{\re}_{k}}\sigma_k^{\varepsilon}(\vec{y})d\Gamma_{\vec{y}}, \sigma_k^{\varepsilon}\rangle_{1/2, -1/2, \Gamma^{\re}_k}+ \langle E^{\varepsilon,k k}(\omega)\sigma_k^{\varepsilon}, \sigma_k^{\varepsilon}\rangle_{1/2, -1/2, \Gamma^{\re}_k}.
		\end{align*}
		Since $S^{\varepsilon,kk}_0\sigma^{\re}_k=1$, and recalling that $\int_{\Gamma_k^{\re}}\sigma_k^{\re}(\vec{y})d\Gamma_{\vec{y}}=c_k^{\re}$, we obtain
		\begin{align*}%
			|\mathbb{M}^{\varepsilon}_{kk}(\omega)-c_k^{\varepsilon}-\frac{i\omega}{4\pi}(c_k^{\varepsilon})^2|\leq \|E^{\varepsilon,kk}(\omega)\sigma_k^{\varepsilon}\|_{L^2(\Gamma_k^{\varepsilon})}\|\sigma_k^{\varepsilon}\|_{L^2(\Gamma_k^{\varepsilon})}\leq \varepsilon^3|\omega|^2\|\sigma_k^{\varepsilon}\|^2_{L^2(\Gamma^{\varepsilon}_k)}.
		\end{align*}
		where the last bound follows from \eqref{eq:pel2} . With $\|\sigma_k^{\varepsilon}\|_{L^2(\Gamma_k^{\varepsilon})}\lesssim 1$, see \eqref{eq:sigmakeps}, 
		\begin{align}
			\label{eq:Me}
			|\mathbb{M}^{\varepsilon}_{kk}(\omega)-c_k^{\varepsilon}-\frac{i\omega}{4\pi}(c_k^{\varepsilon})^2|\leq C_{kk} \varepsilon^3|\omega|^2,
		\end{align}
		where $C_{kk}>0$ depends on $\Gamma_k$ only. 
		
		Next, let us consider the off-diagonal entries, which we expand by developing $(\bx,\by)\mapsto G_{\omega}(|\bx-\by|)$ into the Taylor series around $(\bc_{\ell}, \, \bc_k):$
		\begin{align*}%
			\mathbb{M}^{\varepsilon}_{\ell k}(\omega) = G_{\omega}(| \bc_{\ell} - \bc_k|) &\iint\limits_{\Gamma_{\ell}^{\varepsilon}\times \Gamma_k^{\varepsilon}} \sigma^{\re}_{\ell}(\bx) \sigma^{\re}_{k}(\by)d\Gamma_{\vec{x}}d\Gamma_{\vec{y}} 
			+\iint\limits_{\Gamma_{\ell}^{\varepsilon}\times \Gamma_k^{\varepsilon}} R^{\varepsilon}_{\ell k}(\vec{x},\vec{y}) \sigma^{\re}_{\ell}(\bx) \sigma^{\re}_{k}(\by) d\Gamma_{\vec{x}}d\Gamma_{\vec{y}},\\
			\text{ where }\quad	|R^{\varepsilon}_{\ell k}(\bx,\by)| &\leq 
			\sup_{(\bx,\by)\in B_{\ell}^{\varepsilon}\times B_{k}^{\varepsilon}} \sqrt{|\nabla_{\vec{x}} G_{\omega}(|\bx-\by|)|^2+|\nabla_{\vec{y}} G_{\omega}(|\bx-\by|)|^2} \\ 
			&\times \sqrt{| \vec{x}-\vec{c}_{\ell} |^2+| \vec{y}-\vec{c}_{k} |^2} \leq C \re (\underline{d}_*^{\varepsilon})^{-2}(1+|\omega|) ,
		\end{align*}
		where in the last inequality we used \eqref{eq:bound_grad_G} and $\sqrt{| \vec{x}-\vec{c}_{\ell} |^2+| \vec{y}-\vec{c}_{k} |^2}  \leq \re \sqrt{r_{\ell}^2 + r_{k}^2}$.
		
		With this we conclude that, with $C_{\ell k}$ only dependent of $\Gamma_{\ell}$ and $\Gamma_k$,  
		\begin{align}
			\label{eq:MeG}
			|\mathbb{M}^{\varepsilon}_{\ell k}-G_{\omega}(\bc_{\ell}, \bc_{k})c_{\ell}^{\varepsilon}c_k^{\varepsilon}|
			\leq C_{\ell k} \re^3 (\underline{d}^{\re}_*)^{-2} (1+|\omega|),
		\end{align}
		where we used $\|\sigma_k^{\varepsilon}\|_{L^1(\Gamma_k^{\varepsilon})}\leq |\Gamma_k^{\varepsilon}|^{\frac{1}{2}}\|\sigma_k^{\varepsilon}\|_{L^2(\Gamma_k^{\varepsilon})}$ and $\|\sigma_k^{\varepsilon}\|_{L^2(\Gamma_k^{\varepsilon})}\lesssim 1$, see \eqref{eq:sigmakeps}. 
		
		\mypar{Expansion of the entries of $\mathbb{M}^{\re}_{s}$ as $\varepsilon\rightarrow 0$.} We  apply the results of Step 4.1 to the matrix $\mathbb{M}^{\re}_{s, \ell k}$, which is a particular case of $\mathbb{M}^{\re}_{\ell k}$ for the case when all obstacles are balls of the same capacitance as $\Gamma^{\re}_k$; thus, using also Proposition \ref{prop:difM1_geometry}, item 3, 
		\begin{align}
			\label{eq:dev_entries}
			\begin{split}
				&	|\mathbb{M}^{\re}_{s, kk}(\omega)-c_k^{\varepsilon}-\frac{i\omega}{4\pi}(c_k^{\varepsilon})^2|\leq \widetilde{C}_{kk} \varepsilon^3|\omega|^2,\\
				&	|\mathbb{M}^{\re}_{s, \ell k}(\omega)-G_{\omega}(\bc_{\ell}, \bc_{k})c^{\re}_{\ell} c^{\re}_k | \leq \widetilde{C}_{\ell k} \varepsilon^3 ({\underline{d}}^{\re}_*)^{-2} (1+|\omega|).
			\end{split}
		\end{align}
		\mypar{Estimating $(\mathbb{M}^{\varepsilon})^{-1}-(\mathbb{M}^{\re}_{s})^{-1}$.} We start by rewriting 
		\begin{align*}%
			\mathbb{M}^{\re}&=\mathbb{{M}}^{\re}_s+\mathbb{E}^{\re}=\mathbb{M}^{\re}_s\left(\mathbb{I}+(\mathbb{M}^{\re}_{s})^{-1}\mathbb{E}^{\re}\right), \text{ so that }\\
			(\mathbb{M}^{\re})^{-1}&-(\mathbb{M}^{\re}_s)^{-1}=\left( \left(\mathbb{I}+(\mathbb{M}^{\re}_{s})^{-1}\mathbb{E}^{\re}\right)^{-1}-\mathbb{I}\right)(\mathbb{M}^{\re}_{s})^{-1}.
		\end{align*}
		If $\|(\mathbb{M}^{\re}_{s})^{-1}\mathbb{E}^{\re}\|_F<1$, the standard Neumann series argument yields the bound:
		\begin{align*}%
			\|  \left(\mathbb{I}+(\mathbb{M}^{\re}_{s})^{-1}\mathbb{E}^{\re}\right)^{-1}-\mathbb{I} \|_F  
			&\leq \| (\mathbb{M}^{\re}_{s})^{-1}\mathbb{E}^{\re} \|_F (1 - \| (\mathbb{M}^{\re}_{s})^{-1}\mathbb{E}^{\re} \|_F)^{-1},
		\end{align*}
		and it holds that 
		\begin{align}
			\label{eq:MtildeM}
			\|(\mathbb{M}^{\re})^{-1}-\left(\mathbb{M}^{\re}_s\right)^{-1}\|_F\leq (1-\|(\mathbb{M}^{\re}_{s})^{-1}\mathbb{E}^{\re}\|_F)^{-1}	\|(\mathbb{M}^{\re}_{s})^{-1}\|^2_F\|\mathbb{E}^{\re}\|_F.
		\end{align}
		This is the bound that we would like to use. However, as we will see later, it appears to be impossible to apply for high frequencies $\omega$, and thus we will have to resort to other types of bounds in this case. We start by examining how the condition $\|(\mathbb{M}^{\re}_{s})^{-1}\mathbb{E}^{\re}\|_F<1$ can be ensured. 
		
		\mypar{An estimate on $\|\mathbb{E}^{\re}\|_F$.}
		Comparing \eqref{eq:dev_entries}, \eqref{eq:MeG} and \eqref{eq:Me} yields, with $C>0$, 
		\begin{align*}%
			&|\mathbb{E}^{\re}_{kk}(\omega)|\leq C \varepsilon^3 |\omega|^2, \quad 	|\mathbb{E}^{\re}_{\ell k}(\omega)|\leq C \varepsilon^3 (\underline{d}^{\re}_*)^{-2} (1+|\omega|), \quad \ell \neq k, \text{ so that }\\
			&\|\mathbb{E}^{\re} \|_F \leq CN\varepsilon^3 	(\underline{d}_*^{\re})^{-2} (1+|\omega|)^2.
		\end{align*}
		\mypar{'Low-frequency' estimate. }In view of the above bound and \eqref{eq:invM}, under the condition 
		\begin{equation}
			\label{eq:ME}
			\begin{aligned}
				\|(\mathbb{M}^{\re}_s)^{-1} \mathbb{E}^{\re}\|_F &\leq \| (\mathbb{M}^{\re}_s)^{-1} \|_F \| \mathbb{E}^{\re} \|_F \\
				&\leq \re^2 \times CN^{3/2} (\underline{d}_*^{\re})^{-4} (1+|\omega|)^4 \max(1, (\Im \omega)^{-3})<c_{\operatorname{lf}}<1,
			\end{aligned}
		\end{equation}
		we have, by \eqref{eq:MtildeM}, 
		\begin{align}
			\label{eq:MtildeM2}
			\| (\mathbb{M}^{\re})^{-1}-\left(\mathbb{{M}}^{\re}_s\right)^{-1}\|_F \leq \re C (1-c_{\operatorname{lf}})^{-1}N^{2} \left(\underline{d}_*^{\re}\right)^{-6} (1+|\omega|)^6 \max(1, (\Im \omega)^{-6}).
		\end{align}
		
		Recall the condition  \eqref{eq:ME}, which ensures the validity of \eqref{eq:MtildeM}:
		\begin{align}
			\label{eq:MEalt}
			\re^2 \times CN^{3/2} (\underline{d}^{\re}_{*})^{-4} (1+|\omega|)^4 \max(1, (\Im \omega)^{-3})<c_{\operatorname{lf}}<1.
		\end{align}
		Let us now argue that the bound analogous to \eqref{eq:MtildeM} holds true also in the case when \eqref{eq:MEalt} does not hold, in other words, when 
		\begin{align}
			\label{eq:MEalt_hf}
			\re^{-2}\leq c^{-1}_{\operatorname{lf}}CN^{3/2}(\underline{d}^{\re}_{*})^{-4} (1+|\omega|)^4 \max(1, (\Im \omega)^{-3}),
		\end{align} 
		or, loosely speaking, $|\omega|$ is sufficiently large. \\	
		\mypar{'High-frequency' estimate. }
		We assume \eqref{eq:MEalt_hf} and use the triangle inequality
		\begin{align*}%
			\|(\mathbb{M}^{\re})^{-1}-\left(\mathbb{M}^{\re}_s \right)^{-1}\|_F \leq
			\|(\mathbb{M}^{\re})^{-1}\|_F+\|(\mathbb{M}^{\re}_s )^{-1}\|_F.
		\end{align*}
		Using \eqref{eq:invM} and remarking that by exactly same argument the bound analogous to \eqref{eq:invM} holds as well for $\left(\mathbb{{M}}^{\re}\right)^{-1}$, we obtain 
		\begin{align*}%
			\|(\mathbb{M}^{\re})^{-1}-\left(\mathbb{M}^{\re}_s \right)^{-1}\|_F&\leq \re^{-1} N^{1/2}\widetilde{C} (\underline{d}^{\re}_{*})^{-2} (1+|\omega|)^2 \max(1, (\Im \omega)^{-3}). 
		\end{align*}
		Next, we rewrite the above as 
		\begin{align}
			\nonumber
			\|(\mathbb{M}^{\re})^{-1}-\left(\mathbb{M}^{\re}_s \right)^{-1}\|_F&\leq \re \times \re^{-2}N^{1/2} \widetilde{C} (\underline{d}^{\re}_{*})^{-2} (1+|\omega|)^2 \max(1, (\Im \omega)^{-3}) \\
			\label{eq:hf_bd}
			&\leq \re  \widetilde{C}c^{-1}_{\operatorname{lf}}CN^{2}(\underline{d}^{\re}_{*})^{-6} (1+|\omega|)^6 \max(1, (\Im \omega)^{-6}),
		\end{align}
		where in the last bound we used the condition \eqref{eq:MEalt_hf}. 
		
		To deduce the desired bound \eqref{eq:widm2}, it is sufficient to combine the bound \eqref{eq:hf_bd}, valid under \eqref{eq:MEalt_hf}, and \eqref{eq:MtildeM2}, valid under \eqref{eq:MEalt}.   
	\end{proof}
	The above bounds allow us to quantify the error between the coefficients $ \vv{\lambda}^{\re, (1)}_G$ and $\vv{\lambda}^{\re}_s$, which we translate as the error between the interpolated densities.
	\begin{corollary}
		\label{cor:freq_fin}	
		Let $\omega\in \mathbb{C}^+$, $0<\re\leq 1$. 
		Let $\vv{\lambda}_G^{\re, (1)}: \, \mathbb{C}^+\rightarrow \mathbb{C}^N$ satisfy 
		\begin{align*}%
			{\mathbb{M}}^{\re}(\omega)\vv{\lambda}^{\re,(1)}_G(\omega)=\vv{q}^{\re,(1)}(\omega),
		\end{align*}
		with $\vv{q}^{\re,(1)}(\omega)$ defined in \eqref{eq:problem_M_G},	and $\vv{\lambda}^{\re}_s: \, \mathbb{C}^+\rightarrow \mathbb{C}^N$  a unique solution to
		\begin{align}%
			\label{eq:ds}
			\mathbb{M}^{\re}_s(\omega)\vv{\lambda}^{\re}_s(\omega)=\vv{q}^{\re, (1)}(\omega).
		\end{align}
		Then 
		\begin{align} 
			\label{eq:bound_density_simplified}
			\|\vv{\lambda}^{\re}_s\|_{\mathbb{C}^N} \leq\re^{-1} C_{\operatorname{geom},st}(N, \underline{d}^{\re}_*)p_{st}(|\omega|, \Im \omega) \|\hat{\vec{g}}^{\re,(1)}\|_{L^2(\Gamma^{\re})}, 
		\end{align}
		and the densities $\hat{\blambda}_G^{\re,(1)}=\mathcal{I}^{\re}\vv{\lambda}_G^{\re,(1)}$ and $\hat{\blambda}_s^{\re}=\mathcal{I}^{\re}\vv{\lambda}_G^{\re}$ satisfy the following error bound:
		\begin{align}
			\label{eq:lmb}
			\|\hat{\vec{\lambda}}^{\re, (1)}_G-\hat{\vec{\lambda}}^{\re}_s \|_{H^{-1/2}(\Gamma^{\re})}\leq \re  C_{\operatorname{geom},conv}(N, \underline{d}^{\re}_*)p_{conv}(|\omega|, \Im \omega) \|\hat{\vec{g}}^{\re,(1)}\|_{L^2(\Gamma^{\re})}.
		\end{align}
	\end{corollary}
	\begin{proof}
		We start by proving \eqref{eq:bound_density_simplified}. It holds that $$\|\vv{\lambda}^{\re}_s\|_{\mathbb{C}^N} \leq \| (\mathbb{M}^{\re}_s)^{-1} \|_F \| \vv{q}^{\re,(1)}\|_{\mathbb{C}^N},$$ and next estimate the right-hand side. By the Cauchy-Schwarz inequality and \eqref{eq:sigmakeps}, with $C_{\sigma} := \max_{k=1,\ldots,N}\|\sigma_k^{\re}\|^2_{L^2(\Gamma^{\re})}\lesssim 1$,
		\begin{align}%
			\label{eq:bound_q}
			\| \vv{q}^{\re,(1)}(\omega)\|^2_{\mathbb{C}^N}=\sum\limits_{k=1}^N\left|\int_{\Gamma_k^{\re}}\hat{g}_k^{\re,(1)}(\vec{x},\omega)\sigma_i^{\re}(\vec{x})d\Gamma_{\vec{x}}\right|^2&\leq C_{\sigma} \|\hat{\vec{g}}^{\re,(1)}(\omega)\|_{L^2(\Gamma^{\re})}^2.
		\end{align}	
		The desired bound follows from \eqref{eq:widm} along with \eqref{eq:bound_q}.
		
		\mypar{Step 2.2 Bound on $\hat{\vec{\lambda}}^{\re, (1)}_G -\hat{\vec{\lambda}}^{\re}_s$.}
		Using \eqref{eq:widm2} and \eqref{eq:bound_q}, we obtain 
		\begin{align}
			\label{eq:lambdas}
			\begin{split}
				\| \vv{\lambda}^{\re, (1)}_G - \vv{\lambda}^{\re}_s \|_{\mathbb{C}^N} &\leq \| (\mathbb{M}^{\re})^{-1} - (\mathbb{M}^{\re}_s)^{-1} \|_F \| \vv{q}^{\re,(1)} \|_{\mathbb{C}^N} \\
				&\leq \re C_{\sigma}^{1/2} {C}_{\operatorname{geom}, conv}( N,\underline{d}_*^{\re}) {p}_{\operatorname{diff}}(|\omega|,\Im \omega) \| \hat{\vec{g}}^{\re, (1)} \|_{L^2(\Gamma^{\re})}.
			\end{split}
		\end{align} 
		To conclude, we use the fact that $\mathcal{I}^{\re}$ behaves like an isometry, cf. \eqref{eq:vecV2}.
	\end{proof}
	
	\begin{remark}
		Using the stability bound \eqref{eq:bound_density_simplified} with the isomorphism $\mathcal{I}^{\re}$, the density $\hat{\vec{\lambda}}^{\re}_{s}$ satisfies a stability bound of the type $\eqref{eq:bound_density_simplified}$. This implies the time-domain well-posedness of the simplified model \eqref{eq:tdsys2}, see the reasoning in Section \ref{sec:error_analysis}. 
	\end{remark}
	\subsubsection{Frequency-domain error of the scattered field}
	\label{sec:recovering}
	Recall that the scattered field is obtained via \eqref{eq:scatap}, which we rewrite in the Fourier-Laplace domain as 
	\begin{align}
		\label{eq:fdsys2}
		\begin{split}
			\hat{u}^{\re}_{s,\operatorname{app}}(\vec{x},\omega)&= \sum\limits_{k=1}^N\frac{\mathrm{e}^{i\omega|\bx-\bc_k|}}{4\pi|\bx-\bc_k|}\hat{\lambda}_{s,k}^{\re} c^{\re}_k +\\
			&+\sum\limits_{k=1}^N (\vec{c}_k-\vec{x})\cdot \vec{p}_k^{\re}\frac{\mathrm{e}^{i\omega|\bx-\bc_k|}}{4\pi|\bx-\bc_k|^3}(i\omega|\bx-\bc_k|-1)\hat{\lambda}_{s,k}^{\re}, \quad \vec{x}\in \Omega^{\re,c}.
		\end{split}
	\end{align}
	Let the operator $\mathcal{S}^{\re}_{s}(\omega): \, \mathbb{C}^N\rightarrow L^2(\Omega^{\re,c})$ be defined via 
	\begin{align*}%
		\vv{\varphi} \mapsto 	\left(\mathcal{S}^{\re}_{s}(\omega) \vv{\varphi}\right)(\vec{x})&=\sum\limits_{k=1}^N\frac{\mathrm{e}^{i\omega|\bx-\bc_k|}}{4\pi|\bx-\bc_k|}\varphi_kc_k^{\re} + \\
		&+\sum\limits_{k=1}^N\varphi_k (\vec{c}_k-\vec{x})\cdot \vec{p}_k^{\re}\frac{\mathrm{e}^{i\omega|\bx-\bc_k|}}{4\pi|\bx-\bc_k|^3}(i\omega|\bx-\bc_k|-1), \quad \vec{x}\in \Omega^{\re,c}.
	\end{align*}
	The error between the exact and the scattered field then reads 
	\begin{align*}%
		\hat{u}^{\re}(\vec{x},\omega)-\hat{u}^{\re}_{s, \operatorname{app}}(\vec{x},\omega)&=  (\mathcal{S}^{\re}(\omega)\hat{\vec{\blambda}}^{\re})(\bx)-(\mathcal{S}^{\re}_{s}(\omega)\vv{\lambda}^{\re}_s)(\bx).
	\end{align*}
	We rewrite this error by using the previously introduced quantities \eqref{eq:three_errors}
	\begin{equation}
		\label{eq:four_split}
		\begin{aligned}
			\hat{u}^{\re}(\vec{x},\omega)-\hat{u}^{\re}_{s, \operatorname{app}}(\vec{x},\omega)&=\mathcal{S}^{\re}(\omega)(\hat{\vec{\lambda}}^{\re}-\hat{\vec{\lambda}}^{\re,(1)}) (\bx)
			+\mathcal{S}^{\re}(\omega)(\hat{\vec{\lambda}}^{\re,(1)}-\hat{\vec{\lambda}}^{\re,(1)}_G) (\bx)\\
			&	+\mathcal{S}^{\re}(\omega)(\hat{\vec{\lambda}}^{\re,(1)}_G-\hat{\vec{\lambda}}^{\re}_s)(\bx)+(\mathcal{S}^{\re}_G(\omega)-\mathcal{S}^{\re}_{s}(\omega)) \vv{\lambda}^{\re}_s(\bx),  
		\end{aligned}
	\end{equation}
	where in the latter expression we have introduced the notation 
	\begin{align*}%
		\mathcal{S}_G^{\re}(\omega):=\mathcal{S}^{\re}(\omega)\mathcal{I}^{\re}, 
	\end{align*}
	and have used that  $\hat{\blambda}^{\re}_s=\mathcal{I}^{\re}\vv{\lambda}_s^{\re}$. 
	
	The first three terms can be estimated by combining the stability result of Lemma \ref{lem:slp} used to bound $\mathcal{S}^{\re}(\omega)$ and of Sections \ref{sec:ere1}, \ref{sec:ere2}, \ref{sec:ere3}, recall also \eqref{eq:three_errors}. We thus need to bound the last term in the above. In what follows we will need a more explicit expression of $\mathcal{S}^{\re}_G(\omega): \, \mathbb{C}^N\rightarrow L^2(\Omega^{\re,c})$,
	\begin{align*}%
		\vv{\varphi} \mapsto 	\left(\mathcal{S}^{\re}_G(\omega)\vv{\varphi}\right)(\vec{x})=\sum\limits_{k=1}^N\int_{\Gamma^{\re}_k}\frac{\mathrm{e}^{i\omega|\bx-\by|}}{4\pi|\bx-\by|}\varphi_k\sigma_k^{\re}(\vec{y})d\Gamma_{\vec{y}}, \quad \vec{x}\in \Omega^{\re,c}.
	\end{align*}
	\begin{lemma}[Approximating the volume operator]
		\label{lem:appx_slp_vol}
		Let $K\subset B^c$ be a compact set. Then, there exists $C_K>0$, s.t. for all $\re\in(0,1]$ and $\omega\in \mathbb{C}^+$,
		\begin{align*}%
			\|\mathcal{S}^{\re}_G(\omega)-	\mathcal{S}^{\re}_{s}(\omega)\|_{\mathcal{L}(\mathbb{C}^N, L^{\infty}(K))}\leq \varepsilon^3 \times C_K N^{1/2} (1+|\omega|)^2 .
		\end{align*}
	\end{lemma}
	
	\begin{proof}
		Let $\bx \in K$ and $\by \in \Gamma^{\re}_k$. Expanding the integral kernel in the Taylor series around $\vec{c}_k$ yields
		\begin{align*}%
			\frac{\mathrm{e}^{i\omega|\bx-\by|}}{4\pi|\bx-\by|} =\frac{\mathrm{e}^{i\omega|\bx-\vec{c}_k|}}{4\pi|\bx-\vec{c}_k|}\left(1+(\vec{y}-\vec{c}_k)\cdot \frac{(\vec{c}_k-\vec{x})}{|\vec{x}-\vec{c}_k|} \left(i\omega-\frac{1}{|\vec{x}-\vec{c}_k|}\right)\right)+&\\
			+\sum\limits_{|\vec{\alpha}|=2}R_{\vec{\alpha}}(\vec{x}, \vec{y}, \vec{c}_k)(\vec{y}-\vec{c}_k)^{\vec{\alpha}},&
		\end{align*}
		with the Taylor's series remainder given by, for $|\vec{\alpha}| = 2$, 
		\begin{align*}%
			R_{\vec{\alpha}}(\bx, \by, \bc_k) = \frac{|\vec{\alpha}|}{\vec{\alpha}!} \int_0^1(1-t)^{|\vec{\alpha}|-1} D^{\vec{\alpha}}_{\by} G_{\omega}(|\bx - \bc_k + t(\by-\bc_k)|) dt, \quad D^{\vec{\alpha}}_{\by} := \frac{\partial^{|\vec{\alpha}|}}{\partial \by^{\vec{\alpha}}}.
		\end{align*}
		By the definitions \eqref{eq:ckre2} of $c_k^{\re}$ and $\vec{p}_k^{\re}$ and using the above expansion, we arrive at
		\begin{align}
			\label{eq:Sreg_SreGappa}
			( ( \mathcal{S}^{\re}_G(\omega)-\mathcal{S}^{\re}_{s}(\omega)) \vv{\varphi} )(\vec{x})=\sum\limits_{k=1}^N \varphi_k\int_{\Gamma_k^{\re}}\sum_{|\vec{\alpha}|=2}R_{\vec{\alpha}}(\vec{x}, \vec{y}, \vec{c}_k)(\vec{y}-\vec{c}_k)^{\vec{\alpha}}\sigma_k^{\re}(\vec{y})d\Gamma_{\vec{y}}.
		\end{align}
		Applying the $\ell^2$ Cauchy-Schwarz inequality to \eqref{eq:Sreg_SreGappa} yields 
		\begin{align}
			\label{eq:Sreg_SreGappa_bound}
			| ( ( \mathcal{S}^{\re}_G(\omega)-\mathcal{S}^{\re}_{s}(\omega)) \vv{\varphi} )(\vec{x}) |^2 \leq \| \vv{\varphi} \|^2_{\mathbb{C}^N} \sum_{k=1}^N |\mathbb{R}_k[\sigma^{\re}_k]|^2,
		\end{align}
		with $\mathbb{R}_k[\sigma^{\re}_k]:=\int_{\Gamma^{\re}_k} \sum_{|\vec{\alpha}|=2}R_{\vec{\alpha}}(\vec{x}, \vec{y}, \vec{c}_k)(\vec{y}-\vec{c}_k)^{\vec{\alpha}}\sigma_k^{\re}(\vec{y})d\Gamma_{\vec{y}}$. It is straightforward to see that, since for $|\vec{\alpha}|=2$, it holds that  $|(\vec{y}-\vec{c}_k)^{\vec{\alpha}}|\leq \re^2 r_k^2$, we have
		\begin{align}%
			\label{eq:Rksigma}
			\begin{split}
				|\mathbb{R}_k[\sigma^{\re}_k] |&\leq \| \sum_{|\vec{\alpha}| = 2} R_{\vec{\alpha}} (\bx, \cdot, \bc_k) (\cdot - \bc_k)^{\vec{\alpha}} \|_{L^2(\Gamma^{\re}_k)} \| \sigma^{\re}_k \|_{L^2(\Gamma^{\re}_k)} \\
				&\lesssim \re^3 \max_{|\vec{\alpha}|=2}\max_{\vec{y}\in \Gamma_k^{\re}}\left|R_{\vec{\alpha}}(\vec{x}, \vec{y}, \vec{c}_k)\right|   \| \sigma^{\re}_k \|_{L^2(\Gamma^{\re}_k)}.
			\end{split}
		\end{align}
		Next, we use that, with \eqref{eq:sigmakeps}, $\| \sigma^{\re}_k \|_{L^2(\Gamma^{\re}_k)} \lesssim 1$, and 
		\begin{align*}%
			\left|R_{\vec{\alpha}}(\vec{x}, \vec{y}, \vec{c}_k)\right|\lesssim 
			\min(1, \operatorname{dist}(\bx, B^{\re}_k))^{-3}(1+|\omega|)^2.
		\end{align*}
		Inserting the above into \eqref{eq:Rksigma}, and the result into \eqref{eq:Sreg_SreGappa_bound} yields the desired bound in the statement of the theorem. 
	\end{proof}
	With the above lemma, we obtain the following counterpart of Theorem \ref{th:conv}, modulo the bounds on the data. 
	\begin{theorem}
		\label{th:conv2}
		Let $K\subset B^c$ be compact; $\omega\in \mathbb{C}^+$ and $\hat{u}^{\inc}(\omega)\in H^5(B)$. The error between $\hat{u}^{\re}$ and $\hat{u}^{\re}_{s,\operatorname{app}}$, defined in \eqref{eq:SL_potential} and \eqref{eq:fdsys2}, is bounded by
		\begin{align*}
			\|\hat{u}^{\re} - \hat{u}^{\re}_{s, \operatorname{app}}\|_{L^{\infty}(K)}\leq C_KC_{\operatorname{geom}}(N, \underline{d}_*^{\re}) \, p(|\omega|, \Im \omega)\|\hat{u}^{\inc}\|_{H^5(B)},
		\end{align*}
		where $C_K>0$ depends on $K$.
	\end{theorem}
	\begin{proof}
		Let $\vec{x}\in K$. 
		We start with \eqref{eq:four_split}. With the help of Lemma \ref{lem:slp} used to bound $\mathcal{S}^{\re}(\omega)$ and Lemma \ref{lem:appx_slp_vol} to bound the last term in  \eqref{eq:four_split}, we obtain
		\begin{align*}%
			\|\hat{u}^{\re}(\vec{x},\omega) - \hat{u}^{\re}_{s, \operatorname{app}}(\vec{x},\omega)\|_{L^{\infty}(K)} &\leq 
			C_K C_{\operatorname{geom}}(N, \underline{d}_*^{\re}) \, p(|\omega|, \Im \omega)  \\
			&\times\big[ \re\big( \|\hat{\vec{e}}^{\re,(1)}_{1,\sigma}\|_{L^2(\Gamma^{\re})} + \|\hat{\vec{e}}^{\re,(1)}_{2,\sigma}\|_{L^2(\Gamma^{\re})} \big) \\
			&\quad + \re^{3/2} \big( \|\hat{\vec{e}}^{\re,(1)}_{1,\perp}\|_{H^{-1/2}(\Gamma^{\re})} + \|\hat{\vec{e}}^{\re,(1)}_{2,\perp}\|_{H^{-1/2}(\Gamma^{\re})} \big) \\
			&\quad + \re \|\hat{\vec{e}}^{\re,(1)}_3\|_{H^{-1/2}(\Gamma^{\re})} + \re^3 \| \vv{\lambda}^{\re}_s \|_{\mathbb{C}^N} \big].
		\end{align*}
		Then, by using  
		the bounds \eqref{eq:QperpE}, \eqref{eq:ere1}, the bounds of Section \ref{sec:ere2} and of Corollary \ref{cor:freq_fin}  (namely \eqref{eq:lmb}, \eqref{eq:bound_density_simplified}), we arrive at the following bound:
		\begin{align}
			\label{eq:bound_field_s_FD}
			\begin{split}
				|\hat{u}^{\re}(\vec{x},\omega) - &\hat{u}^{\re}_{s, \operatorname{app}}(\vec{x},\omega)| \leq 
				C_KC_{\operatorname{geom}}(N, \underline{d}_*^{\re}) \, p(|\omega|, \Im \omega) \times    \\
				&\times \left(\re^{1/2}\trinorm {\hat{\vec{g}}^{\re}-\hat{\vec{g}}^{\re,(1)}}_{-1/2}+\re^{3/2}\trinorm {\hat{\vec{g}}^{\re,(1)}}_{3/2}+\re^2\|\hat{\vec{g}}^{\re,(1)}\|_{L^2(\Gamma^{\re})}\right). 
			\end{split}
		\end{align}
		Next, we employ a counterpart of Proposition \ref{prop:data}, namely, Proposition \ref{prop:td2}, to bound the above in terms of $\hat{u}$. 
	\end{proof}
	
	\subsubsection{Proof of Theorem \ref{theorem:convergence_of_models} for the model \eqref{eq:tdsys2}}
	\label{sec:convmod}
	The well-posedness proof follows the same lines as the proof of Theorem \ref{theorem:well_posedness}, cf. the statement of Corollary \ref{cor:freq_fin} and Theorem \ref{th:wp}. 
	
	To prove convergence, we proceed like in the proof of Theorem \ref{theorem:convergence_of_models} for \eqref{eq:tdsys}, see Section \ref{sec:error_analysis}, however, this time work with the estimate of Theorem \ref{th:conv2}.

	\section{Numerical experiments}
	\label{sec:numerical_experiments}
	In all numerical experiments, we employ the convolution quadrature method for the time discretization (see \cite{lubich88_1, lubich88_2, lubich94} for details). As a reference solution, we use results obtained via the Galerkin BEM.
	\subsection{Convergence of the asymptotic models}
	This section provides numerical validation of the results established in Theorem \ref{theorem:convergence_of_models}. We consider a configuration of $N = 5$ obstacles with cavities, positioned in the origin and in the vertices of a regular tetrahedron; see Figure \ref{fig:geom_config} for the geometric setup. 
	\begin{figure}[h!]
		\centering
		\includegraphics[scale=0.45]{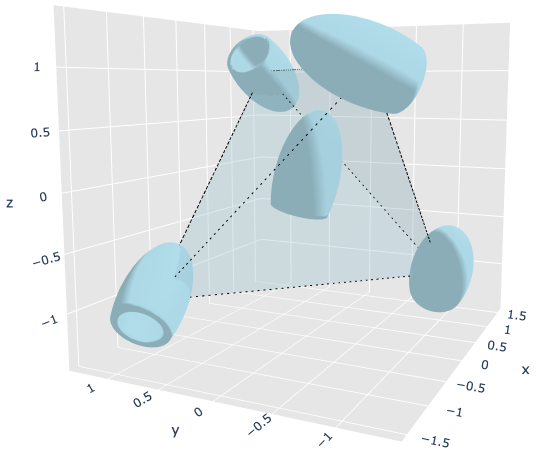}		
		\caption{Geometric configuration of the numerical experiments for $\re=0.5$. The obstacles are centred at $\bc_1 = (0,0,0), ~ \bc_2 = (-1,-1,1), ~\bc_3 = (-1,1,1), ~ \bc_4 = (1,-1,-1),~ \bc_5 = (1,1,1)$.}
		\label{fig:geom_config}
		\vspace{-15pt}
	\end{figure}
	For the incident field, we use a modulated Gaussian pulse (MGP) given by
	\begin{align}
		\label{eq:modulated_gaussian_pulse}
		u^{\operatorname{inc}}(\vec{x}, t) = \cos\left( \omega(t - \vec{x} \cdot \vec{d}) \right) \, e^{-\sigma(t - \vec{x} \cdot \vec{d} - \mu)^2}.
	\end{align}
	with $\omega = 2\pi, \sigma = 3,  \mu = 3, \vec{d} = \frac{1}{\sqrt{3}}(1, -1, 1)$. The simulations are carried out over the time interval $(0, T)$, with final time $T = 13$. We vary the parameter $\re$ and compare the reference solution $u^{\re}(\bx_0, t)$ with the approximate scattered fields $u^{\re}_{\operatorname{app}}(\bx_0, t)$ computed using the three models introduced in Section \ref{sec:principal_result}. The scattered fields are measured at the observation point $\bx_0 = (-1, -1, -1)$. Figure \ref{figure:error_sca_field} illustrates the absolute error of the scattered field for different values of $\re$ and the temporal evolution of the scattered field.
	The numerical results validate the theoretical result stated in Theorem \ref{theorem:convergence_of_models}. In particular, we observe that the absolute error behaves as $O(\re^3)$ for both the asymptotic model \eqref{eq:tdsys} and the simplified model \eqref{eq:tdsys2}, while for the Born approximation \eqref{eq:tdsysborn}, the observed absolute error is of order $O(\re^2)$, consistent with theoretical expectations.
	\begin{figure}[h!]
		\centering
		\includegraphics{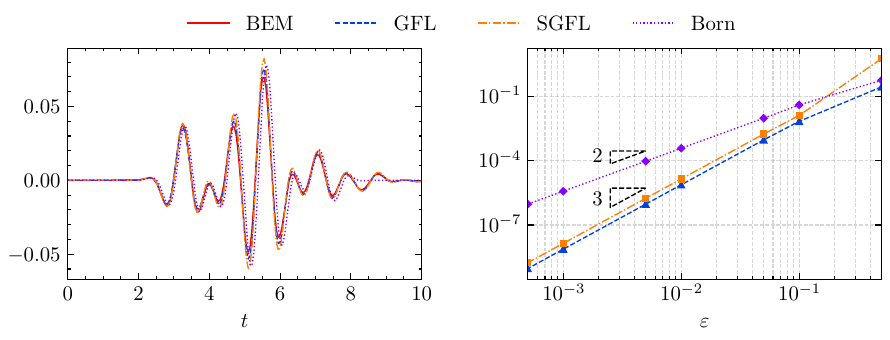}
		\caption{(Left) Time evolution of the scattered field for $\re = 0.1$. 
			(Right) Absolute error of the scattered field computed using the three models described in Section \ref{sec:principal_result}: the Galerkin-Foldy-Lax (GFL), Simplified Galerkin-Foldy-Lax (SGFL), and Born models.}
		\label{figure:error_sca_field}
		\vspace{-15pt}
	\end{figure}
	\subsection{Stability under periodic perturbations}
	\label{sec:long_time_exp}
	In the second experiment, we study the behavior of the scattered field  under periodic perturbations. To this end, we consider $N=4$ non-rotated ellipsoids with principal semi-axes $(0.5, 0.5, 1)$ positioned in the vertices of the regular tetrahedron, see Figure \ref{fig:geom_config}. For the incident field, we set
	\begin{align*}%
		u^{\operatorname{inc}}(\bx, t) = \sin(2\pi(t-\bx\cdot\vec{d}-3)) \sigma(t-\bx\cdot\vec{d}-3).
	\end{align*}
	with $ \sigma(x) = \frac{1}{1+e^{-2x}}, \vec{d} = \frac{1}{\sqrt{3}}(1, -1, 1)$.
	\begin{figure}[h!]
		\centering
		\includegraphics{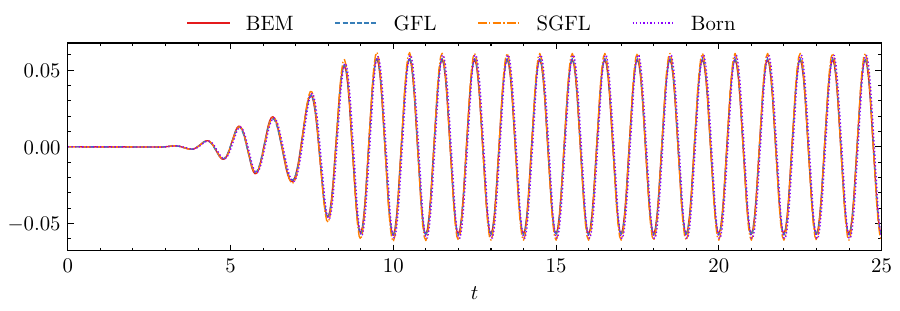}
		\label{figure:long_time}
		\caption{Dependence of the scattered field on time for $\re = 0.05$.
		}
		\label{fig:long_time}
	\end{figure}
	The simulations were made on the time interval $(0, T)$, for the final time $T=25$. The scattered fields are measured at the origin $\bx_0=(0,0,0)$. Figure \ref{fig:long_time} illustrates dependence of the scattered fields on time.
	\subsection{Simplified Model vs Born Approximation}
	\label{sec:many_non_convex_particles}
	In our third experiment, we compare the scattered field computed using the simplified Galerkin-Foldy-Lax model and the Born approximation for a large number of particles. Specifically, we consider a geometric configuration consisting of $N = 71$ non-convex obstacles distributed on the surface of a sphere with radius $r = 1.5$ using the Fibonacci lattice method \cite[Section 3.1]{gonzalez}; see Figure \ref{fig:many_non_convex_particles}.
	
	As the incident field, we again use the MGP \eqref{eq:modulated_gaussian_pulse} with the following parameters:
	\begin{align*}%
		\omega = 2\pi, \quad \sigma = 100, \quad \mu = 2, \quad \vec{d} = (1, -1, 1)/\sqrt{3}.
	\end{align*}
	
	The simulations are performed over the time interval $(0, T)$ with $T = 11$, and the observation point is set at $\bx_0 = (-1.5, -1.5, -1.5)$. The time evolution of the scattered field at this point is presented in Figure \ref{fig:many_non_convex_particles}. According to the theoretical estimates stated in Theorem \ref{theorem:convergence_of_models}, the difference between the scattered fields produced by the simplified model and the Born approximation is of $O(\re^2)$.
	This estimate is confirmed by the numerical results shown in Figure \ref{fig:many_non_convex_particles}. 
	
	
	\begin{figure}[h!]
		\centering
		\begin{minipage}{0.49\textwidth}
			\centering
			\includegraphics[scale=0.42]{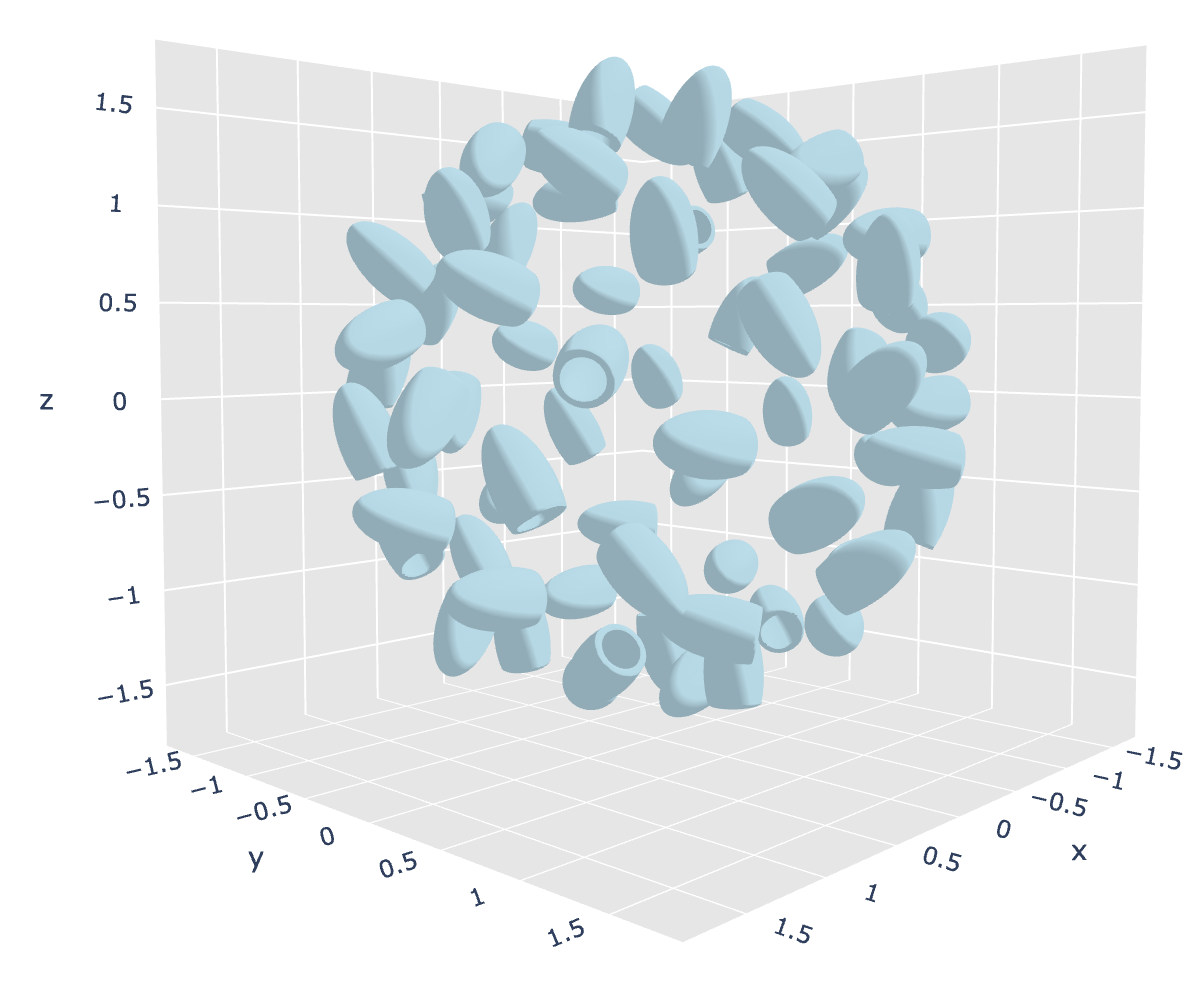}
		\end{minipage}
		\hfill
		\begin{minipage}{0.49\textwidth}
			\centering
			\includegraphics{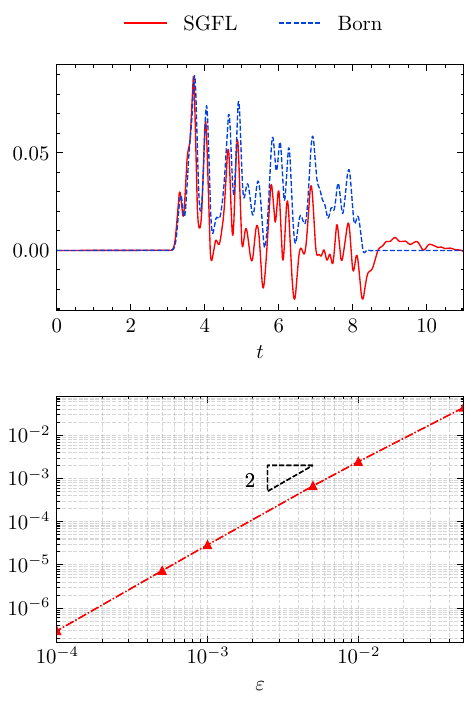}
		\end{minipage}
		\caption{(Left) Geometric configuration from Section \ref{sec:many_non_convex_particles} with $\re = 0.25$; (Top right) Time evolution of the scattered field at $\bx_0 = (-1.5, -1.5, -1.5)$ for $\re = 0.05$. (Bottom right) An absolute error between the scattered field obtained from the simplified and Born models, measured in the $L^{\infty}(0,T)$ norm. 
		}
		\label{fig:many_non_convex_particles}
	\end{figure}
	
	\newpage
	\begin{appendices}

		\section{Scaling of norms}	
		\paragraph{Domains  $\mathcal{O}$ and $\mathcal{O}^{\re}$} Let $\mathcal{O} \subset \mathbb{R}^3$ be a Lipschitz domain such that $\boldsymbol{0} \in \mathcal{O}$. For $\re \in (0, 1]$, let us introduce 
		\begin{align}
			\label{eq:domain_O}
			\mathcal{O}^{\re} :=  \{ \boldsymbol{x} \in \mathbb{R}^3: \boldsymbol{x}/\re  \in \mathcal{O}\}, \text{ so that }
			\mathcal{O}=\mathcal{O}^1, 
			\text{ and } |\partial\mathcal{O}^{\re}|:=\int_{\partial\mathcal{O}^{\re}}d\Gamma_{\bx}.
		\end{align}
		The following results are immediate from the definitions of the corresponding norms and a scaling argument, and thus we leave their proofs to the reader. 
		\begin{lemma}
			\label{lem:norml2}
			Assume $\varphi\in C(\overline{\mathcal{O}})$, then
			\begin{align*}%
				\|\gamma_0 \varphi\|_{L^2( \partial \mathcal{O})}\leq  \|\varphi\|_{L^{\infty}(\mathcal{O})} |\partial \mathcal{O} |^{1/2},
			\end{align*}
			where $|\partial \mathcal{O} | = \int_{\partial \mathcal{O}} d\Gamma_{\bx}$. In particular, if $\partial\mathcal{O}=\Gamma^{\re}_k$, for some $k$, it holds that 
			\begin{align*}
				\| \gamma_0 \varphi \|_{L^2(\Gamma^{\re}_k)} \lesssim \varepsilon \|\varphi\|_{L^{\infty}(B^{\re}_k)}.
			\end{align*}
		\end{lemma}
		\begin{proof}
			The desired bound follows immediately from the continuity of $\varphi$ on $\overline{\O}$. The desired scaling follows from the changing of variables in integral, i.e. $|\Gamma^{\re}_k | = \re^2 |\Gamma_k| $.
			\begin{align*}
			\end{align*}
		\end{proof}
		Since  $\vec{y}\mapsto\int_{\partial\mathcal{O}}|\vec{x}-\vec{y}|^{-1}d\Gamma_{\vec{x}}\in L^{\infty}(\partial\mathcal{O})$,  see \cite[Theorem 3.3.5]{sauter_schwab}, we also have
		\begin{lemma}
			\label{lem:normh12}
			Let $B$ be a convex set containing  $\operatorname{conv}\operatorname{hull}\mathcal{O}$.  For all $\varphi\in W^{1,\infty}(\overline{{B}})$,  
			\begin{align*}%
				|\gamma_0 \varphi|_{H^{1/2}(\partial\mathcal{O})}\leq \|\nabla \varphi\|_{L^{\infty}(B)}\left(\int_{\partial\mathcal{O}}\int_{\partial\mathcal{O}}\frac{1}{| \vec{x}-\vec{y} |}d\Gamma_{\vec{x}}d\Gamma_{\vec{y}}\right)^{1/2}.
			\end{align*}
			In particular, if $\partial\mathcal{O}=\Gamma^{\re}_k$, for some $k$, it holds that $|\gamma_0 \varphi|_{H^{1/2}(\Gamma^{\re}_k)} \lesssim \re^{3/2} \|\nabla \varphi\|_{L^{\infty}(B^{\re}_k)}$. 
		\end{lemma}
		\begin{proof}
			The bounds immediately follows from  the definition of the $H^{1/2} (\bO)$ seminorm and applying the mean value theorem to $\varphi$:
			\begin{align*}
				|\gamma_0 \varphi|^2_{H^{1/2}(\bO)}  &= \iint_{\bO \times \bO} \frac{|\varphi(\bx) - \varphi (\by)|^2}{|\bx - \by|^3} d\Gamma_{\bx} d\Gamma_{\by}\\ &\leq \| \nabla \varphi \|_{L^{\infty}(B)} \iint_{\bO \times \bO} \frac{1}{|\bx - \by|} d\Gamma_{\bx} d\Gamma_{\by}.
			\end{align*}
			Replacing $\bO$ by $\Gamma^{\re}_k$ and changing variables in the above integrals, we derive the desired scaling.
		\end{proof}
		\subsubsection*{Orthogonal projectors associated to $\partial \mathcal{O}^{\re}$ and related estimates}  In this section we will consider the action of $\mathbb{P}_0$, $\mathbb{P}_{\perp}$ (see Section \ref{sec:orthogonal_projectors}) on $\Omega^{\re}$ for the case when $N=1$. To distinguish this case we will state all the results for the re-scaled domain $\mathcal{O}^{\re}$, defined in \eqref{eq:domain_O}. We  update the notation for the counterparts of $\mathbb{S}_0$ and $H^{1/2}_{\perp}$: 
		\begin{align*}%
			\mathbb{S}_0(\partial \mathcal{O}^{\re}) := \operatorname{span}\{1\}, \quad H^{1/2}_{\perp}(\partial \mathcal{O}^{\re}) := \{ \psi \in H^{1/2}(\partial \mathcal{O}^{\re}) : (\psi, 1)_{L^2(\partial \mathcal{O}^{\re})} = 0 \}.
		\end{align*} 
		
		
		%
		\begin{definition}
			\label{def:isomorphism}
			We define the scaling isomorphism between the spaces $L^2(\partial\O^{\re})$ and $L^2(\partial\O)$  (resp. $H^s(\partial\O^{\re})$ and $H^s(\partial\O)$, $s\geq 0$) via:
			\begin{align*}%
				\mathcal{J}^{\re}:\, H^s(\partial\O^{\re})\ni\varphi\mapsto\Phi\in H^s(\partial\O): \, \Phi(\vec{x})=\varphi(\re\vec{x}), \quad \vec{x}\in \partial\O.
			\end{align*}
		\end{definition}
		It indeed holds that $\operatorname{Ran} \left.\mathcal{J}^{\re}\right|_{\mathbb{S}_0(\partial\O^{\re})}=\mathbb{S}_0(\partial\O)$, and $\operatorname{Ran} \left.\mathcal{J}^{\re}\right|_{H^{1/2}_{\perp}(\partial\O^{\re})}=H^{1/2}_{\perp}(\partial\O)$.
		
		%
		The following identities will be of importance later:
		\begin{align}
			\label{eq:ident}
			\mathbb{P}_0\mathcal{J}^{\re}=\mathcal{J}^{\re}\mathbb{P}_0, \quad 	\mathbb{P}_{\perp}\mathcal{J}^{\re}=\mathcal{J}^{\re}\mathbb{P}_{\perp}.
		\end{align}
		Both of them are a corollary of 
		\begin{align*}
			\mathbb{P}_0\mathcal{J}^{\re}\varphi=1_{\partial\mathcal{O}}\frac{1}{|\partial\mathcal{O}|}\int_{\partial\mathcal{O}}\mathcal{J}^{\re}\varphi d\Gamma_{\vec{x}}=1_{\partial\mathcal{O}}\frac{1}{|\partial\mathcal{O}^{\re}|}\int_{\partial\mathcal{O}^{\re}}\varphi(\vec{x}) d\Gamma_{\vec{x}}=\mathcal{J}^{\re}\mathbb{P}_0\varphi.
		\end{align*}
		\begin{lemma}
			\label{lemma:scaling_norm_l2}
			For all $\varphi,\,\psi \in L^2(\partial \mathcal{O}^{\re})$, 
			$(\varphi,\psi)_{L^2(\partial\mathcal{O}^{\re})}= \re^2 (\mathcal{J}^{\re}\varphi,\mathcal{J}^{\re}\psi)_{L^2(\partial\mathcal{O})}$.
		\end{lemma}
		\begin{proof}
			By the change of variables, $(\varphi,\psi)_{L^2(\partial\mathcal{O}^{\re})}=\re^2\int_{\partial\mathcal{O}}\mathcal{J}^{\re}\varphi(\bx) \, \overline{\mathcal{J}^{\re}\psi}(\bx) d\Gamma_{\bx}$. 
		\end{proof}	
		\begin{lemma}
			\label{lem:scaling_norm}
			For all $\varphi \in H^{1/2}(\partial \mathcal{O}^{\re})$, $
			| \varphi |_{H^{1/2}(\partial \mathcal{O}^{\re})} =  \re^{1/2} | \mathcal{J}^{\re}\varphi |_{H^{1/2}(\partial \mathcal{O})}.	$
		\end{lemma}
		\begin{proof}
			By definition of the seminorm \eqref{eq:norm_H_half} and the change of variables, with $\Phi=\mathcal{J}^{\re}\varphi$, it holds that
			\begin{align*}%
				|\varphi |^2_{H^{1/2}(\partial \mathcal{O}^{\re})} &= \int_{\partial \mathcal{O}^{\re}} \int_{\partial \mathcal{O}^{\re}} \frac{|\varphi(\boldsymbol{x})-\varphi(\boldsymbol{y})|^2}{|\boldsymbol{x}-\boldsymbol{y}|^3} d\Gamma_{\boldsymbol{x}} d\Gamma_{\boldsymbol{y}} \\
				&= \re^{4} \re^{-3}  \int_{\partial \mathcal{O}} \int_{\partial \mathcal{O}} \frac{|\Phi(\hat{\boldsymbol{x}})-\Phi(\hat{\boldsymbol{y}})|^2}{|\hat{\boldsymbol{x}}-\hat{\boldsymbol{y}}|^3} d\Gamma_{\hat{\boldsymbol{x}}} d\Gamma_{\hat{\boldsymbol{y}}} = \re |\Phi|^2_{1/2, \partial \mathcal{O}}.
			\end{align*}
		\end{proof}
		We also have a Poincar\`e-type inequaility with a constant explicit in $\re$. 
		\begin{lemma}
			\label{lem:L^2_to_half}
			There exists $C>0$, s.t. for all $0<\re\leq 1$, all $\varphi\in H^{1/2}(\partial\mathcal{O}^{\re})$, 
			$\| \mathbb{P}_{\perp} \varphi \|_{L^2(\partial\mathcal{O}^{\re})} \leq  \re^{1/2} \times C| \varphi |_{H^{1/2}(\partial \mathcal{O}^{\re})}.$ 
		\end{lemma}
		\begin{proof}
			We use $\mathbb{P}_{\perp} \varphi = \varphi - \mathbb{P}_0 \varphi$, and $\mathbb{P}_0 \varphi = 1_{\partial\O^{\re}}\langle \varphi \rangle=1_{\partial\O^{\re}}\frac{1}{|\partial \mathcal{O}^{\re}|} \int_{\partial \mathcal{O}^{\re}} \varphi (\boldsymbol{x}) d\Gamma_{\boldsymbol{x}}$.
			Applying the Cauchy-Schwarz inequality in the third line below yields 
			\begin{align*}%
				\|\mathbb{P}_{\perp} \varphi \|^2_{L^2(\partial\mathcal{O}^{\re})} &= \int_{\partial\mathcal{O}^{\re}} \left|\varphi (\boldsymbol{x}) - \langle \varphi \rangle \right|^2 d \Gamma_{\boldsymbol{x}} \\&= \int_{\partial\mathcal{O}^{\re}} \left|\frac{1}{| \partial\mathcal{O}^{\re} |} \int_{\partial\mathcal{O}^{\re}} \varphi (\boldsymbol{x})  d \Gamma_{\boldsymbol{y}} - \frac{1}{|\partial\mathcal{O}^{\re}|} \int_{\partial\mathcal{O}^{\re}} \varphi(\boldsymbol{y}) d \Gamma_{\boldsymbol{y}} \right|^2 d \Gamma_{\boldsymbol{x}} \\
				&\leq \frac{1}{|\partial\mathcal{O}^{\re}|} \int_{\partial\mathcal{O}^{\re}} \int_{\partial\mathcal{O}^{\re}} | \varphi(\boldsymbol{x}) - \varphi (\boldsymbol{y}) |^2 d \Gamma_{\boldsymbol{x}} d \Gamma_{\boldsymbol{y}} \\
				&\leq  \frac{(\text{diam}(\partial \mathcal{O}^{\re}))^3}{|\partial\mathcal{O}^{\re}|} \int_{\partial\mathcal{O}^{\re}} \int_{\partial\mathcal{O}^{\re}} \frac{| \varphi (\boldsymbol{x}) - \varphi (\boldsymbol{y})|^2}{|\boldsymbol{x} - \boldsymbol{y}|^3} d \Gamma_{\boldsymbol{x}} d \Gamma_{\boldsymbol{y}} \lesssim \re |\varphi_{\perp} |^2_{H^{1/2}(\partial \mathcal{O}^{\re})},
			\end{align*}
			where we used $|\partial\mathcal{O}^{\re}| = \re^2 \times |\partial\mathcal{O}|$ and $\text{diam}(\partial \mathcal{O}^{\re}) = \re  \text{diam}(\partial \mathcal{O})$.
		\end{proof}
		We also have the following observation, which will be used in the subsequent results.
		\begin{corollary}
			\label{corollary:split_norm_H_half}
			Let $0<\re\leq 1$. It holds that  $\|\varphi\|^2_{H^{1/2}(\O^{\re})}=\|\mathbb{P}_0\varphi\|^2_{L^2(\partial\mathcal{O}^{\re})}+\|\mathbb{P}_{\perp}\varphi\|^2_{H^{1/2}(\partial\mathcal{O}^{\re})}$. 
		\end{corollary}
		\begin{proof}
			By Proposition \ref{prop:norm_equivalence}, we have that $\varphi=\mathbb{P}_0\varphi+\mathbb{P}_{\perp}\varphi$, for $\varphi \in H^{1/2}(\bO^{\re})$. Moreover, this decomposition is orthogonal with respect to the $H^{1/2}(\partial\mathcal{O}^{\re})$ scalar product, i.e.
			\begin{align*}
				\|\varphi\|^2_{H^{1/2}(\partial\mathcal{O}^{\re})}=\| \mathbb{P}_0\varphi\|^2_{H^{1/2}(\partial\mathcal{O}^{\re})}+\|\mathbb{P}_{\perp}\varphi\|^2_{H^{1/2}(\partial\mathcal{O}^{\re})},
			\end{align*} 
			The desired identity follows from the fact $|\mathbb{P}_0\varphi|_{H^{1/2}(\bO^{\re})} = 0$.
		\end{proof}
		\begin{corollary}
			\label{cor:equiv_norm}
			There exists $C>0$, s.t. for all $0<\re\leq 1$, $\|\mathbb{P}_{\perp}\varphi\|_{H^{1/2}(\partial \mathcal{O}^{\re})}\leq C|\varphi|_{H^{1/2}(\partial \mathcal{O}^{\re})}$, for all $\varphi\in H^{1/2}(\partial\mathcal{O}^{\re})$. 
		\end{corollary}
		\begin{proof}
			Since $
			\|\mathbb{P}_{\perp} \varphi \|^2_{H^{1/2}(\partial \O^{\re})}  \equiv\|\mathbb{P}_{\perp}\varphi\|^2_{L^2(\partial \O^{\re})}+|\mathbb{P}_{\perp}\varphi|^2_{H^{1/2}(\partial \O^{\re})}$, see Corollary \ref{corollary:split_norm_H_half},
			the desired result follows by Lemma \ref{lem:L^2_to_half} and $|\mathbb{P}_{\perp}\varphi|_{H^{1/2}(\partial \O^{\re})}=|\varphi-\mathbb{P}_0\varphi|_{H^{1/2}(\partial \O^{\re})}=|\varphi|_{H^{1/2}(\partial \O^{\re})}$.
		\end{proof}
		We also have the following, see Proposition \ref{prop:norm_equivalence}: 
		\begin{align}
			\label{eq:split_norm_H_half}
			\|\varphi\|^2_{H^{1/2}(\O^{\re})}=\|\mathbb{P}_0\varphi\|^2_{L^2(\partial\mathcal{O}^{\re})}+\|\mathbb{P}_{\perp}\varphi\|^2_{H^{1/2}(\partial\mathcal{O}^{\re})}.
		\end{align}
		\begin{remark}
			\label{rem:general_rsult}
			We will use the following result: given an orthogonal decomposition of the Hilbert space $E=E_0\overset{\perp}{+}E_{\perp}$, with $E_0$ and $E_{\perp}$ closed subspaces, the corresponding orthogonal projections $L_0$,  $L_{\perp}$, and the adjoint projectors $L_0^*$ and $L^*_{\perp}$, it holds that 
			\begin{align*}
				\|L_{\perp}^*\varphi\|_{E'}=\sup\limits_{\psi\in E_{\perp}: \,\|\psi\|_E=1}\langle \varphi, L_{\perp}\psi\rangle_{E', E}
			\end{align*}
			This is evident by remarking that, on one hand, 
			\begin{align*}
				\|L_{\perp}^*\varphi\|_{E'}\geq \sup\limits_{\psi\in E_{\perp}: \,\|\psi\|_E=1}\langle \varphi, L_{\perp}\psi\rangle_{E', E},
			\end{align*}
			and, on the other hand, for each $\psi\in E: \, \|\psi\|_{E}=1$, the projection $\psi_{\perp}:=L_{\perp}\psi$ is either $0$ or satisfies $\|\psi_{\perp}\|_E\leq 1$, so that $
			|\langle \varphi, L_{\perp}\psi\rangle|\leq \|\psi_{\perp}\|^{-1}_E	|\langle \varphi, \psi_{\perp}\rangle|,$ 
			and thus 
			\begin{align*}
				\|L_{\perp}^*\varphi\|_{E'}\leq \sup\limits_{\psi_{\perp}\in E_{\perp}: \,\|\psi_{\perp}\|_E=1}\langle \varphi, \psi_{\perp}\rangle_{E', E}\equiv \sup\limits_{\psi_{\perp}\in E_{\perp}: \,\|\psi_{\perp}\|_E=1}\langle \varphi, L_{\perp}\psi_{\perp}\rangle_{E', E}.
			\end{align*}
		\end{remark}
		\begin{corollary}
			\label{cor:pperp}
			There exists $C>0$, s.t. for all $0<\re\leq 1$, $\varphi\in H^{-1/2}(\partial \O^{\re})$,  $\|\mathbb{P}_{\perp}^*{\varphi}\|_{H^{-1/2}(\partial\mathcal{O}^{\re})}\leq C\re^{3/2}\|\mathbb{P}_{\perp}^*\mathcal{J}^{\re}{\varphi}\|_{H^{-1/2}(\partial\mathcal{O})}$.
		\end{corollary}
		\begin{proof}
			By the density argument it suffices to prove the result for $\varphi\in L^2(\partial\O^{\re})$. 
			By definition of $\mathbb{P}_{\perp}^*$, it holds that 
			\begin{align*}
				\|\mathbb{P}_{\perp}^*{\varphi}\|_{H^{-1/2}(\partial\mathcal{O}^{\re})}&=\sup\limits_{{\psi}\in H^{1/2}(\partial\mathcal{O}^{\re})}\frac{\langle{\varphi}, \overline{\mathbb{P}_{\perp}{\psi}}\rangle_{-1/2,1/2,\partial\mathcal{O}^{\re}}}{\|{\psi}\|_{H^{1/2}(\partial\mathcal{O}^{\re})}}\\
				&\leq \sup\limits_{{\psi}\in H^{1/2}(\partial\mathcal{O}^{\re})}\frac{\langle{\varphi}, \overline{\mathbb{P}_{\perp}{\psi}}\rangle_{-1/2,1/2,\partial\mathcal{O}^{\re}}}{\sqrt{\|\mathbb{P}_0\psi\|^2_{L^2(\partial\mathcal{O}^{\re})}+|\mathbb{P}_{\perp}\psi|^2_{H^{1/2}(\partial\mathcal{O}^{\re})}}}, 
			\end{align*}
			as it is immediate from Corollary \ref{corollary:split_norm_H_half} and $\|\mathbb{P}_{\perp}\varphi\|_{H^{1/2}(\partial\mathcal{O}^{\re})}\geq |\mathbb{P}_{\perp}\varphi|_{H^{1/2}(\partial\mathcal{O}^{\re})}=|\varphi|_{H^{1/2}(\partial\mathcal{O}^{\re})}$, cf. the proof of Corollary \ref{cor:equiv_norm}. Then 
			\begin{align*}
				\|\mathbb{P}_{\perp}^*{\varphi}\|_{H^{-1/2}(\partial\mathcal{O}^{\re})}&\leq C\sup\limits_{{\psi}\in H^{1/2}_{\perp}(\partial\mathcal{O}^{\re})}\frac{\langle{\varphi}, \overline{\mathbb{P}_{\perp}{\psi}}\rangle_{-1/2,1/2,\partial\mathcal{O}^{\re}}}{|\mathbb{P}_{\perp}\psi|_{H^{1/2}(\partial\mathcal{O}^{\re})}}\\
				&=C\sup\limits_{{\psi}\in H^{1/2}_{\perp}(\partial\mathcal{O}^{\re})}\frac{({\varphi}, {\mathbb{P}_{\perp}{\psi}})_{L^2(\partial\mathcal{O}^{\re})}}{|\mathbb{P}_{\perp}\psi|_{H^{1/2}(\partial\mathcal{O}^{\re})}}, 
			\end{align*} 
			and next we use  Lemma \ref{lemma:scaling_norm_l2}, Lemma \ref{lem:scaling_norm} and \eqref{eq:ident} to obtain 
			\begin{align*}
				\|\mathbb{P}_{\perp}^*&\varphi\|_{H^{-1/2}(\partial\mathcal{O}^{\re})}\leq C\re^2 \sup\limits_{{\psi}\in H^{1/2}_{\perp}(\partial\mathcal{O}^{\re})}\frac{(\mathcal{J}^{\re}{\varphi}, {\mathbb{P}_{\perp}\mathcal{J}^{\re}{\psi}})_{L^2(\partial\mathcal{O})}}{\re^{1/2}|\mathbb{P}_{\perp}\mathcal{J}^{\re}\psi|_{H^{1/2}(\partial\mathcal{O})}}\\
				&=C\re^{3/2} \sup\limits_{{\psi}\in H^{1/2}_{\perp}(\partial\mathcal{O}^{\re})}\frac{\langle\mathbb{P}_{\perp}^*\mathcal{J}^{\re}{\varphi}, {\mathbb{P}_{\perp}\mathcal{J}^{\re}{\psi}}\rangle_{-1/2,1/2,\partial\mathcal{O}}}{|\mathbb{P}_{\perp}\mathcal{J}^{\re}\psi|_{H^{1/2}(\partial\mathcal{O})}}\leq C\re^{3/2}\|\mathbb{P}_{\perp}^*\mathcal{J}^{\re}{\varphi}\|_{H^{-1/2}(\partial\O)},
			\end{align*}
			where we employed as well Corollary \ref{cor:equiv_norm}.
			%
			
		\end{proof}
		The following observation is straightforward: for all $\varphi\in L^2(\partial\O^{\re})$, 	
		\begin{align}
			\label{eq:p00}
			\|\mathbb{P}_0^*\varphi\|_{H^{-1/2}(\partial\O^{\re})}=\|\mathbb{P}_0^*\varphi\|_{L^2(\partial\O^{\re})}, 	
		\end{align} 
		cf. the proof of Corollary \ref{cor:pperp}.
		This is quite trivial in view of 
		\begin{align*}
			\|\mathbb{P}_0^*\varphi\|_{H^{-\frac{1}{2}}(\partial\O^{\re})}&=\sup\limits_{\psi\in H^{\frac{1}{2}}(\partial\O^{\re}), \psi\neq 0}\frac{\langle \varphi, \mathbb{P}_0\psi\rangle}{\|\psi\|_{H^{\frac{1}{2}}(\partial\O^{\re})}}\\
			&=\sup_{\psi\in H^{\frac{1}{2}}(\partial\mathcal{O}^{\re}), \psi\neq 0}\frac{\langle \varphi, \mathbb{P}_0\psi\rangle}{\sqrt{\|\mathbb{P}_0\psi\|^2_{L^2(\partial\O^{\re})}+|\mathbb{P}_{\perp}\psi|^2_{H^{\frac{1}{2}}(\partial\O^{\re})}}},
		\end{align*}
		as follows from Corollary \ref{corollary:split_norm_H_half}. From the above it follows that
		\begin{align*}
			\|\mathbb{P}_0^*\varphi\|_{H^{-1/2}(\partial\O^{\re})}&=\sup_{\psi\in \mathbb{S}_0(\partial\mathcal{O}^{\re}),\, \psi\neq 0}\frac{\langle \varphi, \mathbb{P}_0\psi\rangle}{\|\mathbb{P}_0\psi\|_{L^2(\partial\O^{\re})}}\\
			&=\sup_{\psi\in \mathbb{S}_0(\partial\mathcal{O}^{\re}), \psi\neq 0}\frac{(\mathbb{P}_0\varphi, \psi)_{L^2(\bO^{\re})}}{\|\psi\|_{L^2(\partial\O^{\re})}}=\|\mathbb{P}_0\varphi\|_{L^2(\partial\O^{\re})}
		\end{align*}
		where we used that  $\varphi\in L^2(\partial\O^{\re})$. Hence the desired conclusion.

		\section{Properties of the projectors $\mathbb{Q}_{\sigma}$, $\mathbb{Q}_{\perp}$,  $\mathbb{Q}^*_{\sigma}$, $\mathbb{Q}^*_{\perp}$}
		\label{app:prop_proj_Q}
		For the results that follow, we will need the following observations. 
		First of all, by \eqref{eq:sigmakeps}, with $\mathcal{J}^{\re}_k$ being the scaling transform defined on $\Gamma_k^{\re}$, it holds that $\mathcal{J}^{\re}\sigma_k^{\re}=\re^{-1}\sigma_k^1$. With Corollary \ref{cor:pperp}, we have that  
		\begin{align}
			\label{eq:pperpsigma}
			\|\mathbb{P}_{\perp}^*\sigma_k^{\re}\|_{H^{-1/2}(\Gamma^{\re}_k)}\leq C \re^{1/2}\|\mathbb{P}_{\perp}^*\sigma_k^1\|_{H^{-1/2}(\Gamma_k)}.
		\end{align}
		Since $c_k^{\re}=\int_{\Gamma_k^{\re}}\sigma_k^{\re}d\Gamma_{\bx}=\re c_k^1$, cf. \eqref{eq:sigmakeps}, 
		\begin{align}
			\label{eq:pp0sigma}
			\|\mathbb{P}_0^*\sigma_k^{\re}\|_{H^{-1/2}(\Gamma^{\re}_k)}\equiv		\|\mathbb{P}_0^*\sigma_k^{\re}\|_{L^2(\Gamma^{\re}_k)}=c_k^{\re}/|\Gamma_k^{\re}|^{1/2}=c_k^1/|\Gamma_k|^{1/2}=	\|\mathbb{P}_0^*\sigma_k\|_{L^2(\Gamma_k)}.
		\end{align}
		\begin{proposition}
			\label{prop:bound_Q_perp_P_dual_0}
			There exist $C>0$, s.t. for all $\re \in (0,1)$, all $\boldsymbol{\psi} \in H^{-1/2}(\Gamma^{\re})$, it holds that $
			\| \mathbb{Q}_{\perp} \mathbb{P}^*_0 \boldsymbol{\psi} \|_{H^{-1/2}(\Gamma^{\re})} \leq C \re^{1/2} \| \mathbb{P}^*_0 \boldsymbol{\psi} \|_{L^2(\Gamma^{\re})}.$
		\end{proposition}
		
		\begin{proof}
			Recall that $\mathbb{Q}_{\perp}=\operatorname{Id}-\mathbb{Q}_{\sigma}$, on one hand, and, on the other hand, $\Im \mathbb{Q}_{\perp}=H^{-1/2}_{\perp}(\Gamma^{\re})$, thus $
			\mathbb{Q}_{\perp}=(\mathbb{P}_0^*+\mathbb{P}_{\perp}^*)\mathbb{Q}_{\perp}=\mathbb{P}_{\perp}^*\mathbb{Q}_{\perp}=\mathbb{P}_{\perp}^*(\operatorname{Id}-\mathbb{Q}_{\sigma}).$ 
			Therefore, $
			\mathbb{Q}_{\perp}\mathbb{P}_0^*=\mathbb{P}_{\perp}^*\mathbb{Q}_{\sigma}\mathbb{P}_0^*.$ 
			Using Lemma \ref{lem:expl_char}, for $\vec{\psi}_0=\mathbb{P}_0^*\vec{\psi}$, it holds that 
			\begin{align*}
				\mathbb{Q}_{\perp}\vec{\psi}_0=\sum_{k=1}^N \psi_{0,k}|\Gamma_k^{\re}|(\mathbb{P}_{\perp}^*\sigma_k^{\re}) (c_k^{\re})^{-1}.
			\end{align*}
			With \eqref{eq:pperpsigma} and $c_k^{\re}=\re c_k^1$, we have $
			\|	\mathbb{Q}_{\perp}\vec{\psi}_0\|_{H^{-1/2}(\Gamma^{\re})}^2\lesssim \re^{3}\sum\limits_{k=1}^N |\psi_{0,k}|^2\lesssim \re\|\vec{\psi}_0\|_{L^2(\Gamma^{\re})}^2.$
		\end{proof}
		
		\begin{corollary}
			\label{proof:bound_control}
			There exist $C_{\sigma}, C_{\perp}>0$, s.t., for all $0<\re\leq 1$, all  $\vec{\varphi}\in H^{-1/2}(\Gamma^{\re})$, it holds that
			\begin{align*}%
				\|\mathbb{Q}_{\sigma}\boldsymbol{\varphi}\|_{H^{-1/2}(\Gamma^{\re})} &\leq C_{\sigma} \|\mathbb{P}_0^*\boldsymbol{\varphi}\|_{L^2(\Gamma^{\re})} , \\
				\|\mathbb{Q}_{\perp}\boldsymbol{\varphi}\|_{H^{-1/2}(\Gamma^{\re})} &\leq C_{\perp}(\re^{1/2}\|\mathbb{P}_0^*\boldsymbol{\varphi}\|_{L^2(\Gamma^{\re})}+
				\|\mathbb{P}_{\perp}^*\boldsymbol{\varphi}\|_{H^{-1/2}(\Gamma^{\re})}). 
			\end{align*}
		\end{corollary}
		\begin{proof}
			To obtain the first bound in the statement of proposition, we use \eqref{eq:qperp_pperp}(b), according to which  $\mathbb{Q}_{\sigma}=\mathbb{Q}_{\sigma}\mathbb{P}_0^*$. The desired result follows by the above combined with Proposition \ref{prop:norm_projectors_Q}. 
			
			For the second bound, we again make use  of the decomposition provided by Proposition \ref{ortho_decompositionHm12}:
			\begin{align*}%
				\mathbb{Q}_{\perp}\boldsymbol{\varphi}=\mathbb{Q}_{\perp}\mathbb{P}_0^*\boldsymbol{\varphi}+\mathbb{Q}_{\perp}\mathbb{P}_{\perp}^*\boldsymbol{\varphi}.
			\end{align*}
			For the first term in the right-hand side, we use the bound of Proposition \ref{prop:bound_Q_perp_P_dual_0}, and for the second one the bound of Proposition \ref{prop:norm_projectors_Q}. 
		\end{proof}
		\begin{proposition}
			\label{prop:L^2_H_minus_half_mbound_proof}
			There exists $C>0$, s.t. for all $0<\re\leq 1$, it holds that
			\begin{align*}%
				\|\vec{\psi} \|_{H^{-1/2}(\Gamma^{\re})}\leq 	\|\vec{\psi}  \|_{L^2(\Gamma^{\re})} \leq C \|\vec{\psi} \|_{H^{-1/2}(\Gamma^{\re})}, \quad \forall \vec{\psi}\in \mathcal{V}_G^{\re}.
			\end{align*}
		\end{proposition}
		\begin{proof}
			\label{proof:L^2_H_minus_half_mbound}
			Since the set $\{\vec{\sigma}_k^{\re}, \, k=1,\ldots,N\}$ constitutes a basis in $\mathcal{V}_G^{\re}$, each $\vec{\sigma}_k^{\re}$ is supported on $\Gamma_k^{\re}$, and due to the definition of the norm $H^{-1/2}(\Gamma^{\re})$ \eqref{eq:Hminhalf}, it suffices to prove the inequality in the statement of the Corollary for $\vec{\sigma}_k^{\re}$, $k=1,\ldots, N$. 
			
			The first bound $\|\vec{\sigma}_k^{\re} \|_{H^{-1/2}(\Gamma^{\re})}\leq 	\|\vec{\sigma}_k^{\re}  \|_{L^2(\Gamma^{\re})}$ is immediate from the characterization of the $H^{-1/2}(\Gamma^{\re})$-norm and $\sigma_k^{\re}\in L^2(\Gamma_k^{\re})$:
			\begin{align*}%
				\|\vec{\sigma}_k^{\re} \|_{H^{-1/2}(\Gamma^{\re})}=\|\sigma_k^{\re}\|_{H^{-1/2}(\Gamma^{\re}_k)}=\sup\limits_{v\in H^{1/2}(\Gamma_k^{\re}) \setminus \{0\}}\frac{|\langle \sigma_k^{\re}, v\rangle_{-1/2, 1/2, \Gamma^{\re}_k}|}{\|v\|_{H^{1/2}(\Gamma^{\re}_k)}}\leq \|\sigma_k^{\re}\|_{L^2(\Gamma_k^{\re})},
			\end{align*}
			where the last identity follows from the Cauchy-Schwarz inequality and the definition \eqref{eq:norm_H_half} of the norm $H^{1/2}(\Gamma^{\re}_k)$.
			
			Next, let us argue that $\|\vec{\sigma}_k^{\re}\|_{L^2(\Gamma^{\re})}\leq C_k\|\vec{\sigma}_k^{\re}\|_{H^{-1/2}(\Gamma^{\re})}$. With \eqref{eq:sigmakeps}, we have that   
			\begin{align*}
				\|{\sigma}_k^{\re}\|^2_{L^2(\Gamma^{\re}_k)}=\|{\sigma}_k^{1}\|^2_{L^2(\Gamma_k)}=\frac{\|{\sigma}_k^{1}\|^2_{L^2(\Gamma_k)}}{\|\mathbb{P}_0^*\sigma_k^1\|^2_{H^{-1/2}(\Gamma_k)}}\|\mathbb{P}_0^*\sigma_k^1\|^2_{L^2(\Gamma_k)}= C_k\|\mathbb{P}_0^*\sigma_k^1\|^2_{L^2(\Gamma_k)}, 
			\end{align*}
			where $0<C_k:=\frac{\|{\sigma}_k^{1}\|^2_{L^2(\Gamma_k)}}{\|\mathbb{P}_0^*\sigma_k^1\|^2_{L^2(\Gamma_k)}}<+\infty$, since $\mathbb{P}_0^*\sigma_k^1=\frac{1}{|\Gamma_k|}\int_{\Gamma_k}\sigma_k^1 d\Gamma_{\bx}> 0$, cf. Lemma \ref{lem:expl_char}. By  \eqref{eq:pp0sigma}, the above rewrites, with $C_*:=\max_{k=1, \ldots, N}C_k$, 
			\begin{align*}
				\|{\sigma}_k^{\re}\|^2_{L^2(\Gamma^{\re}_k)}\leq C_*\|\mathbb{P}_0^*\sigma^{\re}_k\|_{H^{-1/2}(\Gamma_k^{\re})}\leq C_*C\|\sigma^{\re}_k\|_{H^{-1/2}(\Gamma_k^{\re})},
			\end{align*}
			where the last result follows by $\|\mathbb{P}_0^*\|_{H^{-1/2}(\Gamma_k^{\re})\rightarrow H^{-1/2}(\Gamma_k^{\re})}=\|\mathbb{P}_0\|_{H^{1/2}(\Gamma_k^{\re})\rightarrow H^{1/2}(\Gamma_k^{\re})}\equiv \|\mathbb{P}_0\|_{H^{1/2}(\Gamma_k^{\re})\rightarrow L^2(\Gamma_k^{\re})}$ and the uniform boundedness of the latter norm in $\re$, cf. Corollary \ref{corollary:split_norm_H_half}.
		\end{proof}
		
		\section{Lifting lemma}
		\label{app:lifting_lemma}
		It can be shown, cf. e.g. Corollary \ref{cor:pperp} and the observation given after its proof, that the scaling of the $H^{-1/2}$ norm differ on the space of constant distributions $\mathbb{S}_0$ and its complement $H^{-1/2}_{\perp}$. As we will see further, this affects coercivity estimates. Thus, we derive lifting lemmas separately for two cases.
		
		
		\mypar{Trace operators and jumps.} 
		We introduce the following spaces, to be used later for the definition of trace operators:
		\begin{align*}%
			H^1(\Delta;\mathbb{R}^3 \setminus \Gamma^{\re})&=H^1(\Delta; \Omega^{\re,c})\oplus H^1(\Delta; \Omega^{\re}), \text{ where }\\
			H^1(\Delta; \mathcal{O})&=\{\varphi\in H^1(\mathcal{O}): \, \Delta \varphi \in L^2(\mathcal{O})\}, \quad \mathcal{O}\in \{\Omega^{\re}, \, \Omega^{\re,c}\}.
		\end{align*}
		Let $\vec{n}$ denote the outward unit normal to $\Omega^{\re}$. 
		
		For $\varphi \in H^1(\Delta;\mathbb{R}^3 \setminus \Gamma^{\re})$, we denote by $\gamma^{-}_0 \varphi$ and $\gamma^{-}_1 \varphi$ its interior trace and its interior normal trace, and by $\gamma^{+}_0 \varphi$ and $\gamma^{+}_1 \varphi$ its exterior trace and its exterior normal trace. For sufficiently regular $\varphi$, we have $\gamma^{\pm }_1 \varphi = \vec{n} \cdot \nabla \gamma^{\pm}_{0} \varphi$. By $[\gamma_0 \phi]$ and $[\gamma_1 \varphi]$, we denote its jumps:
		\begin{align*}%
			[\gamma_0 \varphi] := \gamma^{-}_0 \varphi - \gamma^{+}_0 \varphi, \quad [\gamma_1 \varphi] := \gamma^{-}_1 \varphi - \gamma^{+}_1 \varphi.
		\end{align*}
		\mypar{Auxiliary lifting lemmas.}
		We start with the following set of auxiliary lifting lemmas that we will use in our further proofs. The first result dates back to the seminal work of A. Bamberger and T. Ha Duong \cite{bamberger_ha_duong}. We recall one of its versions below.
		\begin{proposition}[Proposition 2.5.1 in \cite{sayas}]
			\label{prop:sayas}
			Let $\mathcal{O}$ be a Lipschitz domain. Then, there exists $C_{\mathcal{O}}$ such that for all $\xi \in H^{1/2}( \partial \mathcal{O})$ and $a>0$, the unique solution $u \in H^1(\mathcal{O})$ of the Dirichlet boundary value problem 
			\begin{align*}%
				- \Delta u + a^2 u = 0, \quad  &\text{in} ~ \mathcal{O}, \\
				\gamma_0 u = \xi, \quad &\text{on} ~ \partial \mathcal{O},
			\end{align*}
			satisfies the following bound: 
			\begin{align}
				\label{eq:energy_norm_bound}
				\| u \|_{a, \mathcal{O}} \leq C_{\mathcal{O}} \max(1, a^{1/2}) \| \xi \|_{H^{1/2}(\partial \mathcal{O})}.
			\end{align}
			
		\end{proposition}
		In the above the dependence of $C_{\mathcal{O}}$ on the domain is implicit, while we need our estimates to be explicit in the parameter $\varepsilon$. Therefore, we resort to a certain type of a scaling argument. From now on let us assume that $\mathcal{O}$ is a fixed Lipschitz domain, s.t., without loss of generality, $\vec{0}\in\O$, and  $\operatorname{diam}\mathcal{O}=1$. We define the family of the rescaled domains $\O^{\re}:=\{\re\vec{x}, \quad \vec{x}\in \O\}$, $0<\re<1$, and first present a counterpart of Proposition \ref{prop:sayas} for such domains. As discussed before, we consider different possibilities for the boundary value $\xi$, namely, $\xi$ is constant or belongs to the space $H^{1/2}_{\perp}(\partial\O^{\re})$. 
		\begin{lemma}[Lifting lemma for the single-particle case]
			\label{lemma:lifting}
			Let $a > 0 $, $g\in H^{1/2}(\partial\O^{\re})$. Consider, for $0<\re<1$, the family of functions  $(V^{\re})_{\re>0}\subset H^1(\mathbb{R}^3)$ which satisfy the interior boundary-value problem 
			\begin{align*}%
				- \Delta V^{\re} + a^2 V^{\re} = 0~\text{in}~\mathcal{O}^{\re}, \quad \quad \gamma^{-}_0 V^{\re} = g,
			\end{align*}
			and the exterior boundary-value problem 
			\begin{align*}%
				- \Delta V^{\re} + a^2 V^{\re} = 0~\text{in}~\mathcal{O}^{\re,c}, \quad \quad \gamma^{+}_0 V^{\re} = g.
			\end{align*}
			Then there exists $C>0$, independent of $g, \, a, \re$, but depending on $\O$, s.t. for all $a>0$,  $0<\re<1$, the function $V^{\re}$ satisfies the following bound:
			\begin{align*}%
				\| V^{\re} \|^2_{a,\O^{\re}\cup\O^{\re,c}} \leq C \max(1,\re a) \left(\re^{-1} \| \mathbb{P}_0g\|^2_{L^2(\partial \mathcal{O}^{\re})}+ \| \mathbb{P}_{\perp}g\|^2_{H^{1/2}(\partial \mathcal{O}^{\re})}\right).
			\end{align*}
		\end{lemma}
		\begin{proof}
			First of all we remark that both BVPs are well-posed by a standard argument (Lax-Milgram lemma).  Let us define a new function on $\O\cup\O^c$:
			\begin{align*}%
				V^1(\vec{x}):=V^{\re}(\re \vec{x}), \quad \vec{x}\in \O\cup\O^{c}.
			\end{align*}
			Next, we relate energy norms of $V^{\re}$ and $V^1$. We recall the definition 
			\begin{align*}%
				\|V^{\re}\|_{a,\O^{\re}}^2&=a^2\int_{\O^{\re}}|V^{\re}(\vec{x})|^2 d\vec{x}+\int_{\O^{\re}}|\nabla V^{\re}(\vec{x})|^2 d\vec{x},
			\end{align*}
			and after a change of variables $\vec{x}':=\varepsilon^{-1}\vec{x}$ in the integral,
			we obtain 
			\begin{align}
				\label{eq:aore}
				\|V^{\re}\|_{a,\O^{\re}}^2&=\varepsilon\left((a\varepsilon)^2\int_{\O}|V^{1}(\vec{x}')|^2 d\vec{x}'+\int_{\O}|\nabla V^{1}(\vec{x}')|^2 d\vec{x}'\right)=\varepsilon\|V^1\|_{a\varepsilon, \O}^2,
			\end{align}
			where we used $\nabla_{\vec{x}}V^{\varepsilon}(\vec{x})=\varepsilon^{-1}\left.\nabla_{\vec{x'}} V^1(\vec{x'})\right|_{\vec{x'}=\varepsilon^{-1}\vec{x}}$. In a similar manner, using that $\O^{\re,c}=\{\re \vec{x}, \, \vec{x}\in \O^{c}\}$, we prove  an analogous identity on $\O^{\re,c}$:
			\begin{align}
				\label{eq:vre}
				\|V^{\re}\|_{a,\O^{\re,c}}^2=\varepsilon\|V^1\|_{a\varepsilon, \O^c}^2.
			\end{align}
			
			Let us now remark that $V^1$ is defined on $\re$-independent domain $\O\cup\O^{c}$ and is indeed  $H^1(\mathbb{R}^3)$. It satisfies the following BVP, obtained by replacing $V^{\re}(\vec{x})$ by $V^{1}(\re^{-1}\vec{x})$ in the BVPs in the statement of the lemma:
			\begin{align*}%
				&-\Delta V^1+(\varepsilon a)^2 V^1=0 \text{ in }\O\cup\O^c, \\
				&\left.V^1\right|_{\partial\O}=\mathcal{J}^{\re}g,
			\end{align*}
			where $\mathcal{J}^{\re}$ is an isomorphism from Definition \ref{def:isomorphism}. 
			Comparing the above to the statement of Proposition \ref{prop:sayas}, we see that 
			\begin{align*}%
				\|V^1\|_{a\re,\mathcal{D}}\leq C_{\mathcal{D}}\max(1,(\re a)^{1/2})\|\mathcal{J}^{\re}g\|_{H^{1/2}(\partial \mathcal{D})},\quad \mathcal{D} \in \{\O, \O^c\}.
			\end{align*}
			Substituting the above two bounds into \eqref{eq:aore} and \eqref{eq:vre} yields 
			\begin{align*}%
				\|V^{\re}\|_{a,\O^{\re}\cup\O^{\re,c}}^2\leq C\max(1, \re a)\re  \|\mathcal{J}^{\re}g\|_{H^{1/2}(\partial\O)}^2.
			\end{align*} 
			We rewrite  
			\begin{align*}
				\| \mathcal{J}^{\re} g \|^2_{H^{\frac{1}{2}}(\bO)} &= \|  \mathcal{J}^{\re}(\mathbb{P}_0g + \mathbb{P}_{\perp} g) \|^2_{H^{\frac{1}{2}}(\bO)}  = \| \mathbb{P}_0  \mathcal{J}^{\re}g\|^2_{L^2(\bO)} + \| \mathbb{P}_{\perp}  \mathcal{J}^{\re}g\|^2_{H^{\frac{1}{2}}(\bO)} \\
				& = \re^{-2} \| \mathbb{P}_0 g\|^2_{L^2(\bO^{\re})} + \re^{-2}\| \mathbb{P}_{\perp} g\|^2_{L^2(\bO^{\re})}+ \re^{-1} | \mathbb{P}_{\perp} g|^2_{H^{\frac{1}{2}}(\bO^{\re})},
			\end{align*}
			where, in the second equality, we used property \eqref{eq:ident} along with Proposition \ref{prop:norm_equivalence}; in the last line, we applied Lemmas \ref{lemma:scaling_norm_l2} and \ref{lem:scaling_norm}. It remains to conclude by employing Corollary \ref{cor:equiv_norm} to bound $\re^{-2}\| \mathbb{P}_{\perp} g\|^2_{L^2(\bO^{\re})}\lesssim \re^{-1}|\mathbb{P}_{\perp}g|^2_{H^{\frac{1}{2}}(\partial\O^{\re})}$. 
		\end{proof}	
		The above lemma can be generalized to the many-particle case, for a price of an estimate depending on the distance between the particles (see \cite{hassan_stamm} or \cite[Proposition 4.4]{MK}\footnote{In \cite[Equation 4.20, p. 19]{MK}, the definition of the cut-off functions in the proof of Proposition 4.4 contains a minor inaccuracy. This does not affect the correctness of the main result.} for an analogous result).
		\begin{lemma}[Lifting lemma for many particles]
			\label{lifting_many_particles_proof}
			Let $a > 0$, $\vec{\lambda}^{\re} \in H^{1/2}(\Gamma^{\re})$ and $\Lambda^{\re} \in H^1(\mathbb{R}^3)$ satisfy the following interior and exterior boundary problems
			\begin{align*}%
				&-\Delta \Lambda^{\re} + a^2 \Lambda^{\re} = 0 ~in ~ \Omega^{\re}, \quad \gamma^-_0 \Lambda^{\re} = \vec{\lambda}^{\re}, \\ 
				&-\Delta \Lambda^{\re} + a^2 \Lambda^{\re} = 0 ~in ~ \Omega^{\re,c}, \quad \gamma^+_0 \Lambda^{\re} = \vec{\lambda}^{\re}.
			\end{align*} 
			Then, 
			
			\begin{align*}%
				\| \Lambda^{\re} \|^2_{a,\mathbb{R}^3} \leq c (\underline{d}^{\re}_*)^{-2} (1+a) \max(1, a^{-2}) \left(\re^{-1}  \| \mathbb{P}_0\vec{\lambda}^{\re} \|^2_{L^2(\Gamma^{\re})}+ \| \mathbb{P}_{\perp}\vec{\lambda}^{\re} \|^2_{H^{1/2}(\Gamma^{\re})}\right),
			\end{align*}
			
			and $c>0$ is a constant independent of $N, \re, d^{\re}_{ij}, a$. 
		\end{lemma}
		\begin{proof}
			The function $\Lambda^{\re}$ defined as above minimizes the squared energy norm $v\mapsto \|v\|^2_{a,\mathbb{R}^3\setminus\Gamma^{\re}}$ for functions $v\in H^1(\mathbb{R}^3\setminus\Gamma^{\re})$, s.t. $\gamma_0 v=\vec{\lambda}^{\re}$. In other words, 
			\[\Lambda^{\re} = \arg\min_{\substack{v \in H^1(\mathbb{R}^3): \\ \gamma_0 v = \boldsymbol{\Lambda}^{\re}}}\| v \|_{a, \mathbb{R}^3}.\]
			If there exists $v^{\re}\in H^1(\mathbb{R}^3)$ s.t. $\gamma_0 v^{\re}=\vec{\lambda}^{\re}$ and  $\|v^{\re}\|_{a,\mathbb{R}^3}\leq C_{\re}\|\vec{\lambda}^{\re}\|_{H^{1/2}(\Gamma^{\re})}$, then, in virtue of the above observation, 
			\begin{equation}
				\label{eq:lambadre}
				\| \Lambda^{\re} \|_{a, \mathbb{R}^3} \leq \| v^{\re} \|_{a,\mathbb{R}^3} \leq C_{\re} \| \vec{\lambda}^{\re} \|_{H^{1/2}(\Gamma^{\re})}.
			\end{equation}
			
			Our goal is now to construct such a function $v$, for which $C_{\re}$ will be easy to compute. The key idea is to define $v$ as a solution of a certain BVP in the vicinity of $\Omega_k^{\re}$, and $0$ otherwise. This relies on an introduction of certain cutoff functions. 
			
			We will need cutoff functions parametrized by two parameters $r, \, d$. First of all, let  
			\begin{align*}%
				\chi_{d}\in C^{\infty}(\mathbb{R}; [0,1]), \quad	\chi_{d}(\rho):=\left\{
				\begin{array}{ll}
					1, & \rho\leq d/8\\
					0, & \rho\geq d/4,\\
					\in [0,1], &\text{ otherwise },
				\end{array}	
				\right.
			\end{align*}
			Evidently, $\chi_d(\rho)=\chi_1(\rho/d)$. Such a function $\chi_1$ can be constructed e.g. as an appropriate mollification of $\mathbbm{1}_{\rho<1/8}$. 
			Next, we define $\chi_{r,d}(\rho):=\chi_d(\rho-r)=\chi_1(\frac{\rho-r}{d})$, $\rho\geq 0$. 
			
			With the above definition, let us introduce a truncation function associated to $\Omega^{\re}_k$:
			\begin{align*}%
				\chi_k^{\re}(\boldsymbol{x}) := \chi_{r^{\re}_k, \underline{d}^{\re}_*} (|\boldsymbol{x} - \boldsymbol{c}_k|), \quad \boldsymbol{x} \in \mathbb{R}^3.
			\end{align*}	 	
			Note that, by construction 
			\begin{align}
				\label{eq:chione}
				\text{supp} \,\chi_k^{\re} &\subseteq B(\boldsymbol{c}_k, r^{\re}_k+\underline{d}^{\re}_*/4),\quad \text{ therefore, }\quad\overline{\Omega^{\re}_k}\subset \operatorname{supp}\chi_k^{\re},
				\quad 
				\left.\chi_k^{\re}\right|_{\overline{\Omega^{\re}_k}}=1,\\
				\label{eq:supchione}
				\text{supp}\,\chi_k^{\re}&  \cap \text{supp}\,\chi_{\ell}^{\re} = \emptyset, \quad \text{if}~k \neq \ell.
			\end{align} 
			Moreover, because $\chi_{r,d}(\rho)=\chi_1\left(\frac{\rho-r}{d}\right)$, with some $C>0$, it holds that  
			\begin{align}
				\label{eq:bounds_on_cut_off}
				\| \chi_k^{\re} \|_{L^{\infty}(\mathbb{R}^3)} \leq 1, \quad \| \nabla \chi_k^{\re} \|_{L^{\infty}(\mathbb{R}^3)}\equiv \|\partial_{\rho}\chi_{r_k^{\re}, \underline{d}^{\re}_*}\|_{L^{\infty}(\mathbb{R}^3)} \leq C (\underline{d}^{\re}_*)^{-1}.	
			\end{align}
			Using the above definition of cut-off functions, let us introduce a lifting $\Lambda^{\re}_{\chi}$ of $\vec{\lambda}^{\re}$ (this will play a role of the function $v^{\re}$ in the beginning of the proof) as follows:
			\begin{align*}%
				v^{\re}:=\Lambda^{\re}_{\chi} := \sum_{k=1}^N \chi_k^{\re} \Lambda^{\re}_k,
			\end{align*}
			where $\Lambda^{\re}_k \in H^{1}(\mathbb{R}^3)$ satisfies the following Dirichlet boundary problem:
			\begin{align*}%
				&- \Delta \Lambda^{\re}_k + a^2  \Lambda^{\re}_k = 0 \text{ in } \Omega^{\re}_k, \quad  \gamma^{-}_0 \Lambda^{\re}_k = \lambda^{\re}_k, \\
				&- \Delta \Lambda^{\re}_k + a^2  \Lambda^{\re}_k = 0 \text{ in } \Omega^{\re, c}_k, \quad  \gamma^{+}_0 \Lambda^{\re}_k = \lambda^{\re}_k.
			\end{align*}
			The properties \eqref{eq:chione} and \eqref{eq:supchione} ensure that  $\gamma_0\Lambda_{\chi}^{\re}=\vec{\lambda}^{\re}$. Let us now estimate the norm of $\Lambda^{\re}_{\chi}$ by definition. In virtue of \eqref{eq:supchione}, $\chi_k\Lambda^{\re}_k$ are mutually orthogonal in $H^1(\mathbb{R}^3)$, and thus 
			\begin{align*}%
				\| \Lambda^{\re}_{\chi} \|^2_{a, \mathbb{R}^3} &=  \|  \sum_{k=1}^N \chi_k^{\re} \Lambda^{\re}_k \|^2_{a,\mathbb{R}^3} = \sum_{k=1}^N \|  \chi_k^{\re} \Lambda^{\re}_k \|^2_{a,\mathbb{R}^3}\\ 
				&=\sum_{k=1}^N \left( a^2 \|  \chi_k \Lambda^{\re}_k \|^2_{L^2(\mathbb{R}^3)}+\|  \nabla\chi_k^{\re} \Lambda^{\re}_k+\chi_k^{\re}\nabla \Lambda^{\re}_k \|^2_{L^2(\mathbb{R}^3)}\right)\\
				&\leq 2\sum_{k=1}^{N}  \left[a^2 \| \chi_k \|^2_{L^{\infty}(\mathbb{R}^3)} \| \Lambda^{\re}_k \|^2_{L^2(\mathbb{R}^3)}  + \| \nabla \chi_k \|^2_{L^{\infty}(\mathbb{R}^3)} \| \Lambda^{\re}_k \|^2_{L^2(\mathbb{R}^3)} + \| \chi_k \|^2_{L^{\infty}(\mathbb{R}^3)} \|\nabla \Lambda^{\re}_k \|^2_{L^2(\mathbb{R}^3)} \right]\\
				&\overset{	\eqref{eq:bounds_on_cut_off}}{\leq }2C\sum_{k=1}^{N} \left[\| \Lambda^{\re}_k \|^2_{a, \mathbb{R}^3} + (\underline{d}^{\re}_*)^{-2} \| \Lambda^{\re}_k \|^2_{L^2(\mathbb{R}^3)}\right] \\
				&\leq 2C \sum_{k=1}^{N} \left[\| \Lambda^{\re}_k \|^2_{a, \mathbb{R}^3} + ( \underline{d}^{\re}_*)^{-2} a^{-2} \| \Lambda^{\re}_k \|^2_{a, \mathbb{R}^3}\right] \\
				&\leq 2C(1+ a^{-2}(\underline{d}^{\re}_*)^{-2}) \sum_{k=1}^{N} \| \Lambda^{\re}_k \|^2_{a, \mathbb{R}^3}\leq 2C(1+ a^{-2})\max(1, (\underline{d}^{\re}_*)^{-2})\sum_{k=1}^{N} \| \Lambda^{\re}_k \|^2_{a, \mathbb{R}^3}\\
				&= 4C\max(1,a^{-2}) (\underline{d}^{\re}_*)^{-2}\sum_{k=1}^{N} \| \Lambda^{\re}_k \|^2_{a, \mathbb{R}^3}.
			\end{align*}
			Next, we apply Lemma \ref{lemma:lifting} to bound each $\Lambda_k^{\re}$ in terms of $\lambda^{\re}_k$ with the bound $\max(1, \re a) \leq (1+a)$, and the desired result follows at once by \eqref{eq:lambadre} with $v^{\re}=\Lambda^{\re}_{\chi}$.
		\end{proof}

		\section{A counterpart of Proposition \ref{prop:data}}
		\begin{proposition}
			\label{prop:td2}
			Let $B:= \cup_{k=1}^N B_k$ and $\hat{u}^{\operatorname{inc}} \in H^5(B)$. Then, there exists $C > 0$, such that the following bounds hold true:
			\begin{align*}%
				\|\hat{\boldsymbol{g}}^{\re,(1)}\|_{L^2(\Gamma^{\re})}&\leq CN^{1/2}\re\|\hat{u}^{\inc}\|_{H^3(B)}, \\
				| \hat{\boldsymbol{g}}^{\re,(1)}|_{H^{1/2}(\Gamma^{\re})}&\leq CN^{1/2}\re^{3/2}\|\hat{u}^{\inc}\|_{H^3(B)},\\
				\| \hat{\boldsymbol{g}}^{\re}-\hat {\boldsymbol{g}}^{\re,(1)}\|_{L^2(\Gamma^{\re})}&\leq CN^{1/2}\varepsilon^3 \|\hat{ u}^{\operatorname{inc}}\|_{H^4(B)}, \\
				| \hat{\boldsymbol{g}}^{\re}-\hat {\boldsymbol{g}}^{\re,(1)}|_{H^{1/2}(\Gamma^{\re})}&\leq CN^{1/2}\varepsilon^{5/2}\|\hat{ u}^{\operatorname{inc}}\|_{H^5(B)}. 
			\end{align*}
		\end{proposition}
		\begin{proof}
			The argument follows the same lines as the proof of Proposition \ref{prop:data}. Remark that, since $\hat{u}^{\inc} \in H^5(B)$, by the Sobolev embedding \cite[Theorem 4.12, part 2]{adams}, $\hat{u}^{\inc}\in C^{3,\alpha}(\overline{B})$, $\alpha<1/2$. 
			
			Next, we recall that 
			\begin{align*}
				\hat{g}^{\re, (1)}_k(\bx) := - \hat{u}^{\inc}(\bc_k) - \nabla \hat{u}^{\inc}(\bc_k) \cdot (\bx - \bc_k), \quad \bx \in \Gamma^{\re}_k.
			\end{align*}
			Using the above identity and Lemma \ref{lem:norml2}, we have that 
			\begin{align*}
				\| \hat{\vec{g}}^{\re, (1)} \|^2_{L^2(\Gamma^{\re})} \leq C \re^2 \sum_{k=1}^N \| \hat{u}^{\inc} \|^2_{W^{1, \infty}(B^{\re}_k)} \leq C N \re^2 \| \hat{u}^{\inc} \|_{H^{3}(B)}.
			\end{align*}
			Next, we remark that
			\begin{align*}
				\hat{g}^{\re, (1)}_k(\bx) - \hat{g}^{\re, (1)}_k(\by) = - \nabla \hat{u}^{\inc}(\bc_k) \cdot (\bx -\by), \quad \bx, \by \in \Gamma^{\re}_k.
			\end{align*} 
			Using the above identity, we have that, for $C>0$,
			\begin{align*}
				|\hat{g}^{\re, (1)}_k|^2_{H^{1/2}(\Gamma^{\re}_k)} &= \iint_{\Gamma^{\re}_k \times \Gamma^{\re}_k} \frac{|\nabla \hat{u}^{\inc}(\bc_k) \cdot (\bx -\by)|^2}{|\bx-\by|^3} d\Gamma_{\bx} d\Gamma_{\by}    \\
				& \leq \re^{3/2} \iint_{\Gamma_k \times \Gamma_k} \frac{|\nabla \hat{u}^{\inc}(\bc_k) |^2  |\hat{\bx} -\hat{\by}|^2}{|\hat{\bx}-\hat{\by}|^3} d\Gamma_{\hat{\bx}} d\Gamma_{\hat{\by}} \\
				&\leq C \re^3 \| \hat{u}^{\inc}  \|^2_{W^{1,\infty}(B_k)},
			\end{align*}
			which implies 
			\begin{align*}
				| \hat{\vec{g}}^{\re, (1)} |^2_{H^{1/2}(\Gamma^{\re})} \leq  C \re^3 \sum_{k=1}^N \| \hat{u}^{\inc}  \|^2_{W^{1,\infty}(B_k)} \leq C N \re^{3} \| \hat{u}^{\inc}  \|^2_{H^{3}(B)}.
			\end{align*}
			
			Let us prove last two bounds. It is straightforward to verify that 
			\begin{align*}%
				\varphi_k(\bx) := \hat{g}_k^{\re}(\bx)-\hat{g}_k^{\re,(1)}(\bx)&=\hat{u}^{\inc}(\bx)-\hat{u}^{\inc}(\bc_k)-\nabla \hat{u}^{\inc}(\bc_k)\cdot(\vec{x}-\vec{c}_k) \\&=\sum\limits_{|\vec{\alpha}|=2}R_{k,\vec{\alpha}}(\bx) (\vec{x}-\vec{c}_k)^{\vec{\alpha}}, \quad \bx \in \Gamma^{\re}_k,
			\end{align*} 
			where $R_{k,\vec{\alpha}}(\vec{x})=\frac{2}{\vec{\alpha}!}\int_0^1(1-s)D^{\vec{\alpha}}_{\vec{x}}\hat{u}^{\inc}(\vec{c}_k+s(\vec{x}-\vec{c}_k))ds$. Evidently, $\|R_{k,\vec{\alpha}}\|_{L^{\infty}(B^{\re}_k)}\leq  C \|\hat{u}^{\inc}\|_{W^{2,\infty}(B^{\re}_k)}$ and $(\bx-\bc_k)^{\alpha} = \re^2 (\hat{\bx}-\bc_k)^{\alpha}$ for $\hat{\bx} \in B_{k}$ and $|\alpha|=2$. 
			
			Then, using Lemma \ref{lem:norml2}, we obtain
			\begin{align*}
				\| \varphi_k \|_{L^2(\Gamma^{\re}_k)} \leq C \re \| \varphi_k \|_{L^{\infty}(B^{\re}_k)} \leq C \re^3 \| \hat{u}^{\inc} \|_{W^{2,\infty}(B^{\re}_k)} \leq C \re^3\| \hat{u}^{\inc} \|_{H^4(B)},
			\end{align*}
			which immediately implies the desired bound.
			%
			
			Similarly, using Lemma \ref{lem:normh12} yields 
			\begin{align*}
				| \varphi_k |_{H^{1/2}(\Gamma^{\re}_k)} \leq C \re^{3/2} \| \nabla \varphi_k \|_{L^{\infty}(\overline{B^{\re}_k})} \leq  C \re^{5/2}  \|\hat{u}^{\inc}\|_{W^{3,\infty}(\overline{B^{\re}_k})} \leq  C\re^{5/2} \|\hat{u}^{\inc}\|_{H^5(B)}, 
			\end{align*}
			from which the desired bound follows directly.
		\end{proof}
		
		\section{Derivation and analysis of the Born model \eqref{eq:tdsysborn}}
		\label{app:born_analysis}
		The key idea in the analysis of the Born model \eqref{eq:tdsysborn} lays in rewriting it in the form suggested by \eqref{eq:tdsys} and \eqref{eq:tdsys2} and analysing the underlying error just like in Section \ref{sec:simplified}. It is straightforward to see that \eqref{eq:tdsysborn} is equivalent to
		\begin{align*}
			u^{\re}_{B, \operatorname{app}} (\bx, t) =\sum\limits_{k=1}^N \frac{\lambda_{B,k}^{\re}(t-|\vec{x}-\vec{c}_k|)}{4\pi|\vec{x}-\vec{c}_k|}c_k^{\re}, \quad (\bx, t) \in \Gamma^{\re} \times \mathbb{R}_{>0},
		\end{align*}
		where the time-dependent functions $\{ \lambda_{B,k}^{\re} (t) \}_{k=1}^N$ solves the following convolutional in time system  
		\begin{align}
			\label{eq:tdnewborn}
			&	\sum_{k=1}^{N}(K_{B,\ell k}^{\re}*\lambda_{B,k}^{\re})(t)=-u^{\operatorname{inc}}(\vec{c}_{\ell} , t)c_{\ell}^{\re}, \quad\text{with}\\
			&	K_{B,\ell k}^{\re}(t)= 
			\left\{
			\begin{array}{ll}
				\delta(t)c_k^{\re}, &\ell = k, \\
				0, & \ell \neq k.
			\end{array}
			\right.\nonumber
		\end{align}
		To obtain a suitable convergence estimate, we again work in the frequency domain, and compare the above model to \eqref{eq:tdsys2}, which we have shown to be convergent of order $p=2$.
		\subsubsection*{Density error}
		In the frequency domain,  the Born model rewrites
		\begin{align}
			\label{eq:born_problem}
			\begin{split}
				&\text{Find}~\vv{{\lambda}}^{\re}_B(\omega) \in \mathbb{C}^N \quad \text{s.t.} \quad \mathbb{B}^{\re}(\omega) \vv{\lambda}^{\re}_B(\omega) = \vv{q}^{\re, (2)}(\omega) \quad \text{with} \\
				&\mathbb{B}^{\re}_{\ell k}(\omega) = c^{\re}_k \delta_{\ell k}, \quad q^{\re, (2)}_{\ell}(\omega) = -\int_{\Gamma^{\varepsilon}_{\ell} }\hat{u}^{\operatorname{inc}}(\vec{c}_{\ell} , \omega){\sigma}_{\ell} ^{\varepsilon}(\bx)d\Gamma_{\bx}.
			\end{split}
		\end{align}
		As a reference point, we take \eqref{eq:problem_M_s}:
		\begin{align*}%
			&\mathbb{M}^{\re}_s(\omega) \vv{\lambda}^{\re}_s(\omega) = \vv{q}^{\re, (1)}(\omega).	
		\end{align*}
		Our goal is to obtain the bound on $\| \vv{\lambda}^{\re}_s-\vv{\lambda}^{\re}_B \|_{\mathbb{C}^N}$. 
		We start with the following result about the difference of the matrix inverses. 
		\begin{lemma}
			The matrix $\mathbb{B}^{\re}$
			is obviously invertible and satisfies
			\begin{align*}%
				&\| (\mathbb{B}^{\re})^{-1} \|_F \leq \re^{-1} \times C_{\operatorname{geom}}(N, \underline{d}_*^{\re})p(|\omega|, \Im \omega), \\
				&\|(\mathbb{B}^{\re})^{-1}-(\mathbb{M}^{\re})^{-1}\|_F \leq C_{\operatorname{geom}}(N, \underline{d}_*^{\re})p(|\omega|, \Im \omega),
			\end{align*}
			for all $\omega \in \mathbb{C}^{+}$.
		\end{lemma}
		\begin{proof}
			The matrix $\mathbb{B}^{\re}$ is evidently invertible since $c_k^{\re}>0$ for all $k$. Also, by \eqref{eq:sigmakeps} we have that $c_k^{\re}=\re c_k^1$, thus $\|(\mathbb{B}^{\re})^{-1}\|_F\leq C\sqrt{N}\re^{-1}$. Following the proof of Proposition \ref{prop:difM1}, we see that, cf. \eqref{eq:dev_entries}, 
			\begin{align*}%
				|\mathbb{B}_{kk}^{\re}-\mathbb{M}_{s,kk}^{\re}|\leq (4\pi)^{-1}|\omega|(c_k^{\re})^2+\varepsilon^3|\omega|^2\leq C\varepsilon^2(1+|\omega|)^2.
			\end{align*}
			Similarly, for the off-diagonal elements, 
			\begin{align*}%
				|\mathbb{B}_{\ell k}^{\re}-\mathbb{M}_{s,\ell k}^{\re}|\leq |G_{\omega}(\vec{c}_{\ell}, \bc_k)|c_{\ell}^{\re}c_k^{\re}+\widetilde{C}_{\ell k}\re^3 (\underline{d}_*^{\re})^{-2}(1+|\omega|)\leq C\re^2 (\underline{d}_*^{\re})^{-2}(1+|\omega|).
			\end{align*}
			We conclude like in Proposition \ref{prop:difM1}.  
		\end{proof}
				
		With this bound we have the following estimate: 
		\begin{align}
			\label{eq:dens_error}
			\begin{split}
				\|\vv{\lambda}_s^{\re}-\vv{\lambda}_B^{\re} \|_{\mathbb{C}^N}&\leq \|(\mathbb{B}^{\re})^{-1}\|_F \|\vv{q}^{\re,(2)}-\vv{q}^{\re,(1)}\|_{\mathbb{C}^N}+\|(\mathbb{B}^{\re})^{-1}-(\mathbb{M}^{\re})^{-1}\|_F \|\vv{q}_{\ell}^{\re,(1)}\|_{\mathbb{C}^N}\\
				&\leq C_{\operatorname{geom}}(N, \underline{d}_*^{\re})p(|\omega|, \Im \omega)\left(\re^{-1}\|\vv{q}^{\re,(2)}-\vv{q}^{\re,(1)}\|_{\mathbb{C}^N}+\|\vv{q}_{\ell}^{\re,(1)}\|_{\mathbb{C}^N}\right).
			\end{split}
		\end{align}
		Introducing 
		\begin{align*}%
			\hat{\boldsymbol{g}}^{\re,(2)}:=\sum\limits_{k=1}^N \hat{u}^{\operatorname{inc}}(\vec{c}_k)\mathbf{1}_{\vec{c}_k},
		\end{align*}
		we estimate, using the $L^2(\Gamma_{\ell}^{\re})$-Cauchy-Schwarz together with $\|\sigma_{\ell}^{\re}\|_{L^2(\Gamma_{\ell}^{\re})}\lesssim 1$, cf. \eqref{eq:sigmakeps}:
		\begin{align*}%
			\|\vv{q}^{\re,(2)}-\vv{q}^{\re,(1)}\|_{\mathbb{C}^N} \lesssim \|	\hat{\vec{g}}^{\re,(2)}-	\hat{\vec{g}}^{\re,(1)}\|_{L^2(\Gamma^{\re})}, \quad 
			\|\vv{q}_{\ell}^{\re,(1)}\|_{\mathbb{C}^N}\lesssim \| \hat{\vec{g}}^{\re,(1)}\|_{L^2(\Gamma^{\re})},
		\end{align*}
		and \eqref{eq:dens_error} rewrites 
		\begin{align}
			\label{eq:dens_error_main}
			\begin{split}
				\|\vv{\lambda}_s^{\re}-\vv{\lambda}_B^{\re} \|_{\mathbb{C}^N}
				&\leq C_{\operatorname{geom}}(N, \underline{d}_*^{\re})p(|\omega|, \Im \omega)\left(\re^{-1}\|	\hat{\vec{g}}^{\re,(2)}-	\hat{\vec{g}}^{\re,(1)}\|_{L^2(\Gamma^{\re})}+ \|	\hat{\vec{g}}^{\re,(1)}\|_{L^2(\Gamma^{\re})}\right).
			\end{split}
		\end{align}
		
		\subsubsection*{The error of the field}
		In the frequency domain, the Born model yields the following scattered field approximation
		\begin{align}
			\label{eq:born}
			\hat{u}^{\re}_{B,\operatorname{app}}(\bx,\omega)=\sum\limits_{k=1}^N\frac{\mathrm{e}^{i\omega|\bx-\bc_k|}}{4\pi|\bx-\bc_k|}c_k^{\re}\hat{\lambda}_{B,k}^{\re}.
		\end{align}
		Again, we will compare the above to the approximation obtained with \eqref{eq:tdsys2}:
		\begin{align*}%
			\hat{u}^{\re}_{s,\operatorname{app}}(\bx,\omega)=\mathcal{S}^{\re}_s(\omega)\vv{\lambda}_s^{\re}.
		\end{align*}
		Introducing the operator 
		\begin{align*}%
			\mathcal{S}^{\re}_{B}(\omega): \, \mathbb{C}^N\rightarrow L^2(\Omega^{\re,c}), \quad \vv{\varphi}\mapsto 	\left(\mathcal{S}^{\re}_{B}(\omega)\vv{\varphi}\right)(\vec{x})=\sum\limits_{k=1}^N\frac{\mathrm{e}^{i\omega|\bx-\bc_k|}}{4\pi|\bx-\bc_k|}c_k^{\re}\varphi_k, 
		\end{align*}
		and using bounds \eqref{eq:dens_error_main} along with \eqref{eq:bound_density_simplified}, we quantify the error between the two models via 
		\begin{align}
			\label{eq:dens_err}
			\begin{split}
				\|	\hat{u}^{\re}_{B,\operatorname{app}} - \hat{u}^{\re}_{s,\operatorname{app}}\|_{L^{\infty}(K)} &\leq 
				\|\mathcal{S}^{\re}_B\|_{\mathbb{C}^N\rightarrow L^{\infty}(K)} \|\vv{\lambda}_s^{\re}-\vv{\lambda}_B^{\re}\|_{\mathbb{C}^N}+\|\mathcal{S}^{\re}_B-\mathcal{S}^{\re}_s\|_{\mathbb{C}^N\rightarrow L^{\infty}(K)}\|\vv{\lambda}_s^{\re}\|_{\mathbb{C}^N} \\
				&\leq C(N, \underline{d}_*^{\re},|\omega|, \Im \omega) \left[ \|\mathcal{S}^{\re}_B\|_{\mathbb{C}^N\rightarrow L^{\infty}(K)} \left(\re^{-1}\|	\hat{\vec{g}}^{\re,(2)}-	\hat{\vec{g}}^{\re,(1)}\|_{L^2(\Gamma^{\re})}+ \|	\hat{\vec{g}}^{\re,(1)}\|_{L^2(\Gamma^{\re})}\right) \right.\\
				&\left.+ \re^{-1} \|\mathcal{S}^{\re}_B-\mathcal{S}^{\re}_s\|_{\mathbb{C}^N\rightarrow L^{\infty}(K)} \|\hat{\vec{g}}^{\re,(1)}\|_{L^2(\Gamma^{\re})} \right],
			\end{split}
			\end{align} 
			where $C(N, \underline{d}_*^{\re},|\omega|, \Im \omega)>0$ independent of $\re$.
		Next, we introduce a counterpart of Lemma \ref{lem:appx_slp_vol}.
		\begin{lemma}[Approximating the volume operator]
			\label{lem:appx_slp}
			Let $K\subset \Omega^c$ be a compact set. There exists $C_K>0$, s.t. the following bounds hold true for all $0<\re<1$, $\omega\in \mathbb{C}^+$:
			\begin{align*}%
				&\|\mathcal{S}^{\re}_B\|_{\mathcal{L}(\mathbb{C}^N, L^{\infty}(K))}\leq \re  C_KN^{1/2},\\
				&\|\mathcal{S}^{\re}_s-\mathcal{S}^{\re}_B\|_{\mathcal{L}(\mathbb{C}^N, L^{\infty}(K))}\leq \varepsilon^2N^{1/2}C_K(1+|\omega|).
			\end{align*}
		\end{lemma}
		\begin{proof}
			The first bound is obtained by the $\ell_2$ Cauchy-Schwarz inequality together with  $c^{\re}_k=\re c_k^1$, see \eqref{eq:sigmakeps}. The second bound is obtained like in Lemma \ref{lem:appx_slp_vol}, and thus we leave its proof for the reader.  
		\end{proof}
		\begin{lemma}
			\label{lem:bound_rhs_born}
			Let $\hat{u}^{\inc}\in H^3(B)$. Then, there exists $C>0$ such that for all $\re \in(0,1]$, we have that
			\begin{align*}
				\|	\hat{\vec{g}}^{\re,(2)}-	\hat{\vec{g}}^{\re,(1)}\|_{L^2(\Gamma^{\re})} \leq \re^{2} \times CN^{1/2}\| \hat{u}^{\inc} \|_{H^3(B)}.
			\end{align*}
		\end{lemma}
		\begin{proof}
			We remark that 
			\begin{align*}
				\varphi_{k}(\bx) := \hat{g}^{\re,(2)}_k(\bx) - \hat{g}^{\re,(1)}_k(\bx) = \nabla \hat{u}^{\inc}(\bc_k) \cdot (\bx-\bc_k), \quad \bx \in \Gamma^{\re}_k.
			\end{align*}
			which yields 
			\begin{align*}
				\| \varphi_{k} \|_{L^2(\Gamma^{\re}_k)} \leq \| \nabla \hat{u}^{\inc}\|_{L^{\infty}(B_k)} \| \bx-\bc_k \|_{L^2(\Gamma^{\re}_k)} \leq C\re^2\| \hat{u}^{\inc} \|_{H^{3}(B)},
			\end{align*}
			where the last inequality holds true by the Sobolev embedding \cite[Theorem 4.12, part 2]{adams}, in particular, we have that $\hat{u}^{\inc}\in C^{1,\alpha}(\overline{B})$, $\alpha<1/2$. Hence, the desired bound follows immediately. 
		\end{proof}
		Combining \eqref{eq:dens_err} with Lemma \ref{lem:appx_slp}, \ref{lem:bound_rhs_born} and Proposition \ref{prop:td2}, we obtain the following result.
		\begin{theorem}
			\label{th:conv3}
			Let $K\subset B^c$ be compact; $\omega\in \mathbb{C}^+$ and $\hat{u}^{\inc}(\omega)\in H^3(B)$. The error between $\hat{u}^{\re}_{s,\operatorname{app}}$ and $\hat{u}^{\re}_{B,\operatorname{app}}$, defined in \eqref{eq:fdsys2} and \eqref{eq:born}, is bounded by
			\begin{align*}
				\|\hat{u}^{\re}_{s, \app} - \hat{u}^{\re}_{B, \app}\|_{L^{\infty}(K)}\leq \re^2 \times C_KC_{\operatorname{geom}}(N, \underline{d}_*^{\re}) \, p(|\omega|, \Im \omega)\|\hat{u}^{\inc}\|_{H^3(B)},
			\end{align*}
			where $C_K>0$ depends on $K$.
		\end{theorem}
		\subsubsection*{Proof of Theorem \ref{theorem:convergence_of_models} for the model (\ref{eq:tdsysborn})}
		
		Finally, to prove the desired statement of Theorem \ref{theorem:convergence_of_models} for the Born model, we follow the lines of the proof of Theorem \ref{theorem:convergence_of_models} for the model \eqref{eq:tdsys2}, see Section \ref{sec:convmod}, employing Theorem \ref{th:conv3}.
		Since the model \eqref{eq:tdsys2} is convergent of order $p=2$, we conclude about the validity of the statement of Theorem \ref{theorem:convergence_of_models} for the model \eqref{eq:tdsysborn}. 
		
	\end{appendices}
	
	\bibliographystyle{plain} 
	\bibliography{ref}
\end{document}